\documentclass[10pt]{article}
\usepackage{amssymb,amsbsy,latexsym,amsmath,bbm,epsfig,psfrag,amsthm,mathrsfs,mathabx}
\usepackage[swedish,english]{babel}
\usepackage[T1]{fontenc}
\usepackage[latin1]{inputenc}
\usepackage{amsfonts}
\usepackage{amsmath}
\usepackage{amssymb}
\usepackage{epsf}
\usepackage{graphics}
\usepackage{graphicx}
\usepackage{latexsym}
\usepackage{psfrag}
\usepackage{relsize}
\usepackage{tikz}
\usepackage{esint}
\usepackage{verbatim}
\usepackage{hyperref}
\usepackage{float}
\usepackage{pgfplots}

\usetikzlibrary{positioning}
\usetikzlibrary{arrows}
\usetikzlibrary{decorations.pathreplacing}

\newcommand{\loza}
{--++ (1/2,{sqrt(3)/2})
--++ (-1/2,{sqrt(3)/2})
--++ (-1/2,{-sqrt(3)/2})
--++ (1/2,{-sqrt(3)/2})
[fill=yellow!100!white!]}

\newcommand{\lozb}
{--++ (1,0)
--++ (-1/2,{sqrt(3)/2})
--++ (-1,0)
--++ (1/2,{-sqrt(3)/2})
[fill=blue!80!white!]}

\newcommand{\lozc}
{--++ (1,0)
--++ (1/2,{sqrt(3)/2})
--++ (-1,0)
--++ (-1/2,{-sqrt(3)/2})
[fill=red!60!white!]}

\newcommand{\lozyt}
{--++ (0,1)
--++ (-1,1)
--++ (0,-1)
--++ (1,-1)
[fill=yellow!100!white!]}

\newcommand{\lozbt}
{--++ (1,0)
--++ (-1,1)
--++ (-1,0)
--++ (1,-1)
[fill=blue!80!white!]}

\newcommand{\lozrt}
{--++ (1,0)
--++ (0,1)
--++ (-1,0)
--++ (0,-1)
[fill=red!60!white!]}

\newcommand{\lozy}
{--++ (1/4,{sqrt(3)/4})
--++ (-1/4,{sqrt(3)/4})
--++ (-1/4,{-sqrt(3)/4})
--++ (1/4,{-sqrt(3)/4})
[fill=yellow!100!white!]}

\newcommand{\lozr}
{--++({sqrt(3)/8},-1/8)
--++({sqrt(3)/8},1/8)--++(-{sqrt(3)/8},1/8)
--++(-{sqrt(3)/8},-1/8)
[fill=red!60!white!]}

\newcommand{\lozg}
{--++({sqrt(3)/8},1/8)
--++(0,1/4)--++(-{sqrt(3)/8},-1/8)
--++(0,-1/4)
[fill=yellow!60!white!]}

\newcommand{\lozbl}
{--++({sqrt(3)/8},-1/8)
--++(0,1/4)--++(-{sqrt(3)/8},1/8)
--++(0,-1/4)
[fill=blue!60!white!]}

\newcommand{\lozp}
{--++({sqrt(3)/8},-1/8)
--++({sqrt(3)/8},1/8)--++(-{sqrt(3)/8},1/8)
--++(-{sqrt(3)/8},-1/8)
[fill=purple!80!white!]}

\newcommand{\lozblu}
{--++({sqrt(3)/8},-1/8)
--++({sqrt(3)/8},1/8)--++(-{sqrt(3)/8},1/8)
--++(-{sqrt(3)/8},-1/8)
[fill=red!100!white!]}

\theoremstyle{plain}

\newtheorem{Lem}{Lemma}
\numberwithin{Lem}{section}
\newtheorem{Prop}{Proposition}
\numberwithin{Prop}{section}
\newtheorem{Thm}{Theorem}
\numberwithin{Thm}{section}
\newtheorem{Cor}{Corollary}
\numberwithin{Cor}{section}

\numberwithin{Con}{section}

\theoremstyle{definition}
\newtheorem{Def}{Definition}
\numberwithin{Def}{section}
\newtheorem{Assump}{Assumption}

\numberwithin{conj}{section}
\newtheorem{ex}{Example}
\numberwithin{ex}{section}

\theoremstyle{remark}
\newtheorem{rem}{Remark}
\numberwithin{rem}{section}

\numberwithin{equation}{section}

\newcommand{\dv}{\partial}

\newcommand{\eps}{\varepsilon}

\newcommand{\E}{{\mathbb E}}
\newcommand{\R}{{\mathbb R}}
\newcommand{\Pp}{{\mathbb P}}
\newcommand{\C}{{\mathbb C}}
\newcommand{\Z}{{\mathbb Z}}

\newcommand{\N}{{\mathbb N}}
\newcommand{\G}{\Gamma}

\newcommand{\x}{\chi}

\newcommand{\p}{\mathbb{P}}

\newcommand{\Hp}{\mathbb{H}}

\newcommand{\supp}{\text{supp}}

\newcommand{\EE}{\mathcal{E}}
\newcommand{\g}{\gamma}

\newcommand{\LL}{\mathcal{L}}

\begin{document}

\vspace{.4cm}

\title{The Cusp-Airy Process}
\author{Erik Duse \thanks{Supported the grant KAW 2010.0063 from the Knut and Alice Wallenberg
Foundation} \and  Kurt Johansson \thanks{Supported the grant KAW 2010.0063 from the Knut and Alice Wallenberg
Foundation} \and Anthony Metcalfe}

\maketitle

%\begin{center}
%Preliminary version. Not for circulation.
%\end{center}
\begin{abstract}
At a typical cusp point of the disordered region in a random tiling model we expect to see a determinantal process called the Pearcey process in the
appropriate scaling limit. However, in certain situations another limiting point process appears that we call the Cusp-Airy process, which is a kind of two sided extension of the Airy kernel point process. We will study this problem in a class of random lozenge tiling models coming from interlacing particle systems. The situation was briefly studied
previously by Okounkov and Reshetikhin under the name cuspidal turning point.
\end{abstract}
\tableofcontents
\clearpage
\section{Introduction and Results}

\subsection{The Cusp-Airy kernel}\label{sec:intro}

In this paper we will study random discrete interlacing particle systems which can also be interpreted as certain random lozenge tiling models.
The particles, or lozenges, form a random point process which is determinantal. We are particularly interested in the limiting
point process around the type of cusp point we see in the arctic curve in figure \ref{figFrozenBoundary}, see figure 1 in \cite{KeOk1} or \cite{KeOk2} for
a simulation. Figure \ref{figFrozenBoundary} illustrates the \emph{liquid region} $\LL$ and its boundary $\EE$, the {\it arctic curve}. Inside the liquid region one expects to see the \emph{extended sine-kernel} point process in the limit. At the tangency points of the polygon and the boundary $\EE$ one expects to see the \emph{GUE-corner} process. At all other points of the curve $\EE$, except the cusp, one expects either the \emph{Airy kernel} or \emph{Id-Airy  kernel} point processes in the appropriate scaling limit.

To clarify the situation around the cusp point consider figure \ref{figCardioid1}.  All possible tiling configurations of the polygon can be encoded by the red rhombi. This is illustrated in figure \ref{figInterlacing1}, where we see that the positions of the red rhombi form  two interlacing regions that meet at the dashed line. The rhombi at the common line have been coloured in purple. The purple rhombi form a \emph{discrete orthogonal polynomial ensemble}, DOPE, see Remark \ref{remDOPE}. The dashed line is also a symmetry line (coloured blue in figure \ref{figCardioid1}). The fact that we have two symmetric interlacing systems meeting at common line, will imply that the frozen boundary has a reflection symmetry in the symmetry line. It will also imply that the particle system consisting of red rhombi will have no horizontal oscillations. Therefore, when considering a scaling limit at the cusp on the symmetry line of this determinantal point process, the correct scaling is discrete in the horizontal direction and continuous of size $n^{1/3}$ in the vertical direction, where $n$ is the size of the hexagon. Going back to figure \ref{figCardioid1}, we see that directly above the tip of the cusp the blue and yellow rhombi form a corner. This implies that the height function in \cite{KeOk2} will have a jump above the tip of the cusp. This should be contrasted to the situation where one expects to find a Pearcey process in the scaling limit around a cusp point of the arctic curve. Then one has only one type of rhombi in the frozen configuration inside the cusp. This also implies that the height function is flat inside the cusp.

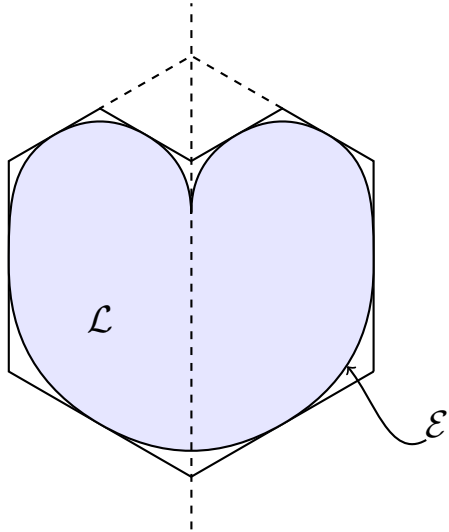
\begin{figure}[H]
\centering
\begin{tikzpicture}[xscale=0.7,yscale=0.7]

%Hexagon
\draw[thick] (0,0) -- (0,4) --(-{sqrt(3)},5) -- (-{2*sqrt(3)},4) -- (-{3*sqrt(3)},5) -- (-{4*sqrt(3)},4) -- (-{4*sqrt(3)},0) -- (-{2*sqrt(3)},-2) --(0,0);
\filldraw[blue!10!white!] (-{2*sqrt(3)},3) to [out=90,in=330] (-{2.5*sqrt(3)},4.5) to [out=150,in=30]  (-{3.5*sqrt(3)},4.5) to [out=210,in=90] (-{4*sqrt(3)},2)
to [out=270,in=150] (-{3*sqrt(3)},-1) to [out=330,in=210] (-{sqrt(3)},-1) to [out=30,in=270] (0,2) to [out=90,in=330] (-{0.5*sqrt(3)},4.5)
to [out=150,in=30] (-{1.5*sqrt(3)},4.5) to [out=210,in=90] (-{2*sqrt(3)},3);
\draw[thick] (-{2*sqrt(3)},3) to [out=90,in=330] (-{2.5*sqrt(3)},4.5) to [out=150,in=30]  (-{3.5*sqrt(3)},4.5) to [out=210,in=90] (-{4*sqrt(3)},2)
to [out=270,in=150] (-{3*sqrt(3)},-1) to [out=330,in=210] (-{sqrt(3)},-1) to [out=30,in=270] (0,2) to [out=90,in=330] (-{0.5*sqrt(3)},4.5)
to [out=150,in=30] (-{1.5*sqrt(3)},4.5) to [out=210,in=90] (-{2*sqrt(3)},3);
\draw[thick,dashed] (-1.73,5)--(-3.46,6) -- (-5.20,5);
\draw[thick,dashed] (-3.46,-3)--(-3.46,7);
\draw (-{3*sqrt(3)},1) node {\Large$\mathcal{L}$};
\draw (1.2,-1) node {\Large$\mathcal{E}$};
\draw[->,thick] (1,-1.3) to [out=210,in=310] (-0.5,0.1);
\end{tikzpicture}
\caption{\label{figFrozenBoundary}The \emph{liquid region} $\LL$ is coloured in blue. Its boundary is $\EE$.}

\end{figure}

\begin{figure}[H]
\centering
\begin{tikzpicture}[xscale=0.7,yscale=0.7]

%lozenge symmetry line
\draw (-{2.125*sqrt(3)},4.125) \lozr;
\draw (-{2.125*sqrt(3)},{4.125+0.25}) \lozr;
\draw (-{2.125*sqrt(3)},{4.125+0.25}) \lozr;
\draw (-{2.125*sqrt(3)},{4.125+0.5}) \lozr;
\draw (-{2.125*sqrt(3)},{4.125+0.75}) \lozr;
\draw (-{2.125*sqrt(3)},{4.125+1}) \lozr;
\draw (-{2.125*sqrt(3)},{4.125+1.25}) \lozr;
\draw (-{2.125*sqrt(3)},{4.125+1.5}) \lozr;
\draw (-{2.125*sqrt(3)},{4.125+1.75}) \lozr;
\draw (-{2.125*sqrt(3)},{2.125}) \lozr;
\draw (-{2.125*sqrt(3)},{2.125-1}) \lozr;
\draw (-{2.125*sqrt(3)},{2.125-1.25-0.5}) \lozr;
\draw (-{2.125*sqrt(3)},{2.125-2}) \lozr;
\draw (-{2.125*sqrt(3)},{2.125-2.5}) \lozr;
\draw (-{2.125*sqrt(3)},{2.125-3}) \lozr;
\draw (-{2.125*sqrt(3)},{2.125-3.75}) \lozr;
\draw (-{2.125*sqrt(3)},{2.125-4}) \lozr;
%lozenge left side
\draw (-{2.125*sqrt(3)-sqrt(3)/8},{2.125+0.875}) \lozr;
\draw (-{2.125*sqrt(3)-sqrt(3)/4},{2.125+1}) \lozr;
\draw (-{2.125*sqrt(3)-3*sqrt(3)/8},{2.125+1.5+0.125}) \lozr;
\draw (-{2.125*sqrt(3)-4*sqrt(3)/8},{2.125+2}) \lozr;
\draw (-{2.125*sqrt(3)-5*sqrt(3)/8},{2.125+2.375}) \lozr;
%lozenge left side yellowblue
\draw (-{2.125*sqrt(3)},{2.125}) \lozg;
\draw (-{2.125*sqrt(3)},{2.125+0.25}) \lozg;
\draw (-{2.125*sqrt(3)},{2.125+0.5}) \lozg;
\draw (-{2.125*sqrt(3)},{2.125+1}) \lozbl;
\draw (-{2.125*sqrt(3)},{2.125+1.25}) \lozbl;
\draw (-{2.125*sqrt(3)},{2.125+1.5}) \lozbl;
\draw (-{2.125*sqrt(3)},{2.125+1.75}) \lozbl;
\draw (-{2.125*sqrt(3)-sqrt(3)/8},{2.125+1.125}) \lozbl;
\draw (-{2.125*sqrt(3)-sqrt(3)/8},{2.125+1.25+0.125}) \lozbl;
\draw (-{2.125*sqrt(3)-sqrt(3)/8},{2.125+1.5+0.125}) \lozbl;
\draw (-{2.125*sqrt(3)-sqrt(3)/8},{2.125+1.75+0.125}) \lozbl;
\draw (-{2.125*sqrt(3)-sqrt(3)/4},{2.125+1}) \lozg;
\draw (-{2.125*sqrt(3)-sqrt(3)/4},{2.125+1.25}) \lozg;
\draw (-{2.125*sqrt(3)-sqrt(3)/4},{2.125+1.75}) \lozbl;
\draw (-{2.125*sqrt(3)-sqrt(3)/4},{2.125+2}) \lozbl;
\draw (-{2.125*sqrt(3)-3*sqrt(3)/8},{2.125+2.125}) \lozbl;
\draw (-{2.125*sqrt(3)-3*sqrt(3)/8},{2.125+1.625}) \lozg;
\draw (-{2.125*sqrt(3)-4*sqrt(3)/8},{2.125+2}) \lozg;
%lozenge right side
\draw (-{2.125*sqrt(3)+sqrt(3)/8},{2.125+0.625}) \lozr;
\draw (-{2.125*sqrt(3)+sqrt(3)/4},{2.125+1}) \lozr;
\draw (-{2.125*sqrt(3)+3*sqrt(3)/8},{2.125+1.5+0.125}) \lozr;
\draw (-{2.125*sqrt(3)+4*sqrt(3)/8},{2.125+2.25}) \lozr;
\draw (-{2.125*sqrt(3)+5*sqrt(3)/8},{2.125+2.375}) \lozr;
%lozenge right side yellowblue
\draw (-{2.125*sqrt(3)+sqrt(3)/8},{2.25}) \lozbl;
\draw (-{2.125*sqrt(3)+sqrt(3)/8},{2+0.5}) \lozbl;
\draw (-{2.125*sqrt(3)+sqrt(3)/8},{2.75}) \lozg;
\draw (-{2.125*sqrt(3)+sqrt(3)/8},{2+1}) \lozg;
\draw (-{2.125*sqrt(3)+sqrt(3)/8},{2+1.25}) \lozg;
\draw (-{2.125*sqrt(3)+sqrt(3)/8},{2+1.5}) \lozg;
\draw (-{2.125*sqrt(3)+sqrt(3)/8},{2+1.75}) \lozg;
\draw (-{2.125*sqrt(3)+2*sqrt(3)/8},{2+0.75+0.125}) \lozbl;
\draw (-{2.125*sqrt(3)+2*sqrt(3)/8},{2+1.125}) \lozg;
\draw (-{2.125*sqrt(3)+2*sqrt(3)/8},{2+1.25+0.125}) \lozg;
\draw (-{2.125*sqrt(3)+2*sqrt(3)/8},{2+1.5+0.125}) \lozg;
\draw (-{2.125*sqrt(3)+2*sqrt(3)/8},{2+1.75+0.125}) \lozg;
\draw (-{2.125*sqrt(3)+3*sqrt(3)/8},{2+1.125+0.125}) \lozbl;
\draw (-{2.125*sqrt(3)+3*sqrt(3)/8},{2+1.125+0.125+0.25}) \lozbl;
\draw (-{2.125*sqrt(3)+3*sqrt(3)/8},{2+1.5+0.25}) \lozg;
\draw (-{2.125*sqrt(3)+3*sqrt(3)/8},{2+1.75+0.25}) \lozg;
\draw (-{2.125*sqrt(3)+4*sqrt(3)/8},{2+1.125+0.125+0.625}) \lozbl;
\draw (-{2.125*sqrt(3)+4*sqrt(3)/8},{2+1.125+0.125+0.25+0.625}) \lozbl;
%Hexagon
\draw[thick] (0,0) -- (0,4) --(-{sqrt(3)},5) -- (-{2*sqrt(3)},4) -- (-{3*sqrt(3)},5) -- (-{4*sqrt(3)},4) -- (-{4*sqrt(3)},0) -- (-{2*sqrt(3)},-2) --(0,0);
\draw[thick,dashed] (-1.73,5)--(-3.46,6) -- (-5.20,5);
\draw[thick,dashed,blue] (-3.46,-3)--(-3.46,7);

\end{tikzpicture}
\caption{\label{figCardioid1}}

\end{figure}

\begin{figure}[H]
\centering
\begin{tikzpicture}[xscale=1,yscale=1]

%lozenge symmetry line
\draw ({-sqrt(3)/8},{3}) \lozp;
\draw ({-sqrt(3)/8},{3.25}) \lozp;
\draw ({-sqrt(3)/8},{2.75}) \lozp;
\draw ({-sqrt(3)/8},{2.5}) \lozp;
\draw ({-sqrt(3)/8},{1}) \lozp;
\draw ({-sqrt(3)/8},{0.25}) \lozp;
\draw ({-sqrt(3)/8},{-0.25}) \lozp;
\draw ({-sqrt(3)/8},{-0.5}) \lozp;
%Lozenge to the left
\draw ({-sqrt(3)/8-0.5},{3-0.125}) \lozr;
\draw ({-sqrt(3)/8-0.5},{3.25-0.125}) \lozr;
\draw ({-sqrt(3)/8-0.5},{2.75-0.125}) \lozr;
\draw ({-sqrt(3)/8-1},{3-0.25}) \lozr;
\draw ({-sqrt(3)/8-1},{3.25-0.25}) \lozr;
\draw ({-sqrt(3)/8-1.5},{3.25-0.25-0.125}) \lozr;
\draw ({-sqrt(3)/8-0.5},{1.75+0.125}) \lozr;
\draw ({-sqrt(3)/8-0.5},{0.75+0.125}) \lozr;
\draw ({-sqrt(3)/8-0.5},{0.125}) \lozr;
\draw ({-sqrt(3)/8-0.5},{-0.25-0.125}) \lozr;
\draw ({-sqrt(3)/8-1},{2.25}) \lozr;
\draw ({-sqrt(3)/8-1},{1.25}) \lozr;
\draw ({-sqrt(3)/8-1},{0.5}) \lozr;
\draw ({-sqrt(3)/8-1},{0}) \lozr;
\draw ({-sqrt(3)/8-1.5},{2.25+0.125+0.25}) \lozr;
\draw ({-sqrt(3)/8-1.5},{1.25+0.125+0.5}) \lozr;
\draw ({-sqrt(3)/8-1.5},{0.5+0.125+0.25}) \lozr;
\draw ({-sqrt(3)/8-1.5},{0.125}) \lozr;
\draw ({-sqrt(3)/8-2},{2.25+0.5}) \lozr;
\draw ({-sqrt(3)/8-2},{1.25+0.125+0.5+0.25}) \lozr;
\draw ({-sqrt(3)/8-2},{0.5+0.125+0.25+0.5}) \lozr;
\draw ({-sqrt(3)/8-2},{0.125+0.5}) \lozr;
\draw ({-sqrt(3)/8-2.5},{1.25+0.125+0.75+0.25}) \lozr;
\draw ({-sqrt(3)/8-2.5},{0.5+0.125+0.25+0.5+0.5}) \lozr;
\draw ({-sqrt(3)/8-2.5},{0.125+0.5+0.25}) \lozr;
\draw ({-sqrt(3)/8-3},{0.5+0.125+0.25+0.5+0.5+0.25}) \lozr;
\draw ({-sqrt(3)/8-3},{0.125+0.5+0.25+0.75}) \lozr;
\draw ({-sqrt(3)/8-3.5},{0.125+0.5+0.125+1}) \lozr;
%Lozenge to the right
\draw ({-sqrt(3)/8+0.5},{3-0.125}) \lozblu;
\draw ({-sqrt(3)/8+0.5},{3.25-0.125}) \lozblu;
\draw ({-sqrt(3)/8+0.5},{2.75-0.125}) \lozblu;
\draw ({-sqrt(3)/8+1},{3-0.25}) \lozblu;
\draw ({-sqrt(3)/8+1},{3.25-0.25}) \lozblu;
\draw ({-sqrt(3)/8+1.5},{3.25-0.25-0.125}) \lozblu;
\draw ({-sqrt(3)/8+0.5},{1.75+0.125+0.25}) \lozblu;
\draw ({-sqrt(3)/8+0.5},{0.75+0.125-0.25}) \lozblu;
\draw ({-sqrt(3)/8+0.5},{0}) \lozblu;
\draw ({-sqrt(3)/8+0.5},{-0.25-0.125}) \lozblu;
\draw ({-sqrt(3)/8+1},{2.25+0.25}) \lozblu;
\draw ({-sqrt(3)/8+1},{1.25+0.25}) \lozblu;
\draw ({-sqrt(3)/8+1},{0.5-0.125}) \lozblu;
\draw ({-sqrt(3)/8+1},{-0.125}) \lozblu;
\draw ({-sqrt(3)/8+1.5},{2.25+0.25+0.125}) \lozblu;
\draw ({-sqrt(3)/8+1.5},{1.25+0.125+0.25}) \lozblu;
\draw ({-sqrt(3)/8+1.5},{0.5+0.125+0.25}) \lozblu;
\draw ({-sqrt(3)/8+1.5},{0.125}) \lozblu;
\draw ({-sqrt(3)/8+2},{2.25+0.5}) \lozblu;
\draw ({-sqrt(3)/8+2},{1.25+0.125+0.5}) \lozblu;
\draw ({-sqrt(3)/8+2},{0.5+0.125+0.25+0.5}) \lozblu;
\draw ({-sqrt(3)/8+2},{0.125+0.5}) \lozblu;
\draw ({-sqrt(3)/8+2.5},{1.25+0.125+0.75+0.25}) \lozblu;
\draw ({-sqrt(3)/8+2.5},{0.5+0.125+0.25+0.5+0.25}) \lozblu;
\draw ({-sqrt(3)/8+2.5},{0.125+0.5+0.25}) \lozblu;
\draw ({-sqrt(3)/8+3},{0.5+0.125+0.25+0.5+0.5}) \lozblu;
\draw ({-sqrt(3)/8+3},{0.125+0.5+0.25+0.75+0.25}) \lozblu;
\draw ({-sqrt(3)/8+3},{0.125+0.5+0.75}) \lozblu;
\draw ({-sqrt(3)/8+3.5},{0.125+0.125+1+0.25}) \lozblu;
% Interlacing lines
\draw[thick,dashed] (0,-2) -- (0,4); 
\draw[thick] (-0.5,-2) -- (-0.5,4); 
\draw[thick] (-1,-2) -- (-1,4);
\draw[thick] (-1.5,-2) -- (-1.5,4);
\draw[thick] (-2,-2) -- (-2,4);
\draw[thick] (-2.5,-2) -- (-2.5,4);
\draw[thick] (-3,-2) -- (-3,4);
\draw[thick] (-3.5,-2) -- (-3.5,4);
\draw[thick] (0.5,-2) -- (0.5,4); 
\draw[thick] (1,-2) -- (1,4);
\draw[thick] (1.5,-2) -- (1.5,4);
\draw[thick] (2,-2) -- (2,4);
\draw[thick] (2.5,-2) -- (2.5,4);
\draw[thick] (3,-2) -- (3,4);
\draw[thick] (3.5,-2) -- (3.5,4);
%Coordinate system
\draw[->,thick] (4.5,1)--(4.5,2);
\draw[->,thick] (4.5,1)--(5.5,1);
\draw (6.1,2) node { \small$\xi$, continuous scaling};
\draw (6.5,1.2) node { \small$r$, discrete scaling};
\end{tikzpicture}
\caption{\label{figInterlacing1}}

\end{figure}

At the cusp, in an appropriate scaling limit, we will see that the correlation kernel for the determinantal point process given by the red rhombi converges to a process with a kernel that we call the {\it Cusp-Airy kernel}. We will show this for the model corresponding to figure \ref{figFrozenBoundary} up to certain natural technical conditions for a particular DOPE. In a simpler model of the type studied by Petrov in \cite{Pet12} we will give the full proof. 
This type of cusp situation in a random lozenge tiling model was discovered and discussed briefly by Okounkov and Reshetikhin in \cite{OkRe} who called it a {\it Cuspidal turning point}. The interpretation of their formula is however not completely clear, see remark \ref{Cuspidal} below.

Let us give the expression for the Cusp-Airy kernel.

\begin{Def}
\label{CAkerneldefinition}
For $r,s\in \Z$ and $\xi,\tau\in \R$ we define the \emph{Cusp-Airy kernel} by
\begin{align} 
\label{CAKernel}
\mathcal{K}_{CA}((\xi,r),(\tau,s))=-1_{\tau\geq \xi}1_{s>r}\frac{(\tau-\xi)^{s-r-1}}{(s-r-1)!}+\frac{1}{(2\pi i)^2}\int_{\mathscr{L}_L+\mathscr{C}_{out}}dz\int_{\mathscr{L}_R+\mathscr{C}_{in}}dw\frac{1}{w-z}\frac{w^r}{z^s}e^{\frac{1}{3}w^3-\frac{1}{3}z^3-\xi w+\tau z},
\end{align}

where the contours are defined in figure \ref{figCuspContour}, and $1_{a<b}$ is the indicator function for $a<b$.
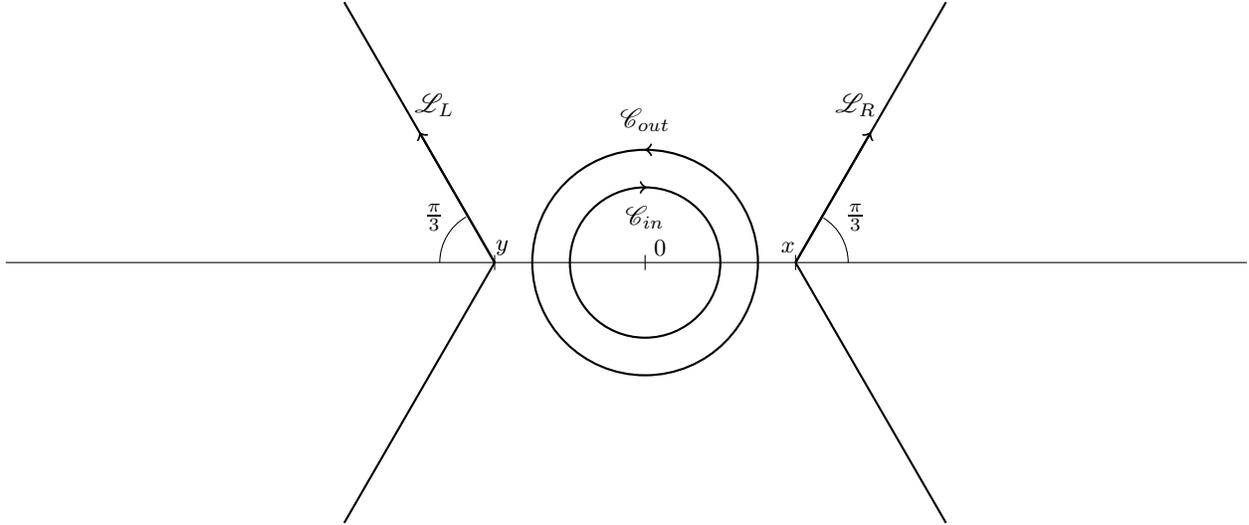
\begin{figure}[H]

\centering
\begin{tikzpicture}

\draw (-8.5,0) -- (8,0);
\draw (0,-0.1) -- (0,0.1);
\draw[thick] (2,0) --({2+4/2},{4*sqrt(3)/2}) ;
\draw[->,thick] (2,0) --({2+2/2},{2*sqrt(3)/2}) ;
\draw[thick]  (2,0) --({2+4/2},{-4*sqrt(3)/2}) ;
\draw[thick]  (-2,0) --({-2-4/2},{4*sqrt(3)/2}) ; 
\draw[->,thick] (-2,0)-- ({-2-2/2},{2*sqrt(3)/2});
\draw[thick]  (-2,0) --({-2-4/2},{-4*sqrt(3)/2}) ;
\draw[thick]  (0,0) circle (1.5cm);
\draw[thick]  (0,0) circle (1.0cm);
\draw[<-,thick] (0,1.5) -- (0.02,1.5);
\draw[->,thick] (0,1.0) -- (0.02,1.0);
\draw (-2.8,2.1) node {$\mathscr{L}_L$}; 
\draw (2.8,2.1) node {$\mathscr{L}_R$};
\draw (0,1.9) node {$\mathscr{C}_{out}$}; 
\draw (0,0.6) node {$\mathscr{C}_{in}$}; 
\draw (0.2,0.2) node {\small$0$}; 
\draw (2.0,-0.1) -- (2.0,0.1);
\draw (-2.0,-0.1) -- (-2.0,0.1);
\draw (1.9,0.2) node {\small$x$}; 
\draw (-1.9,0.2) node {\small$y$}; 
\draw (2.7,0) arc (0:58:0.7);
\draw (2.8,0.6) node {$\frac{\pi}{3}$};
\draw ({-2-0.38},{0.7*sqrt(3)/2}) arc (120:180:0.7);
\draw (-2.8,0.6) node {$\frac{\pi}{3}$};
\end{tikzpicture}
\caption{\label{figCuspContour}Integration contours for the Cusp-Airy kernel.}

\end{figure}
\end{Def}

In section \ref{sec:rAiry} we give a different formula for the kernel in terms of r-Airy integrals and some polynomials. 
When $r=s=0$ we see that we get the Airy kernel. Hence, we expect the last red particle on the line $r=0$ to have Tracy-Widom fluctuations.
In the case $r=s\neq 0$ we interestingly get the {\it r-Airy kernel}, which has appeared in previous work, see \cite{ADvM} and \cite{Peche}.

%\begin{rem}
%It is interesting to compare the Cusp-Airy kernel to the GUE-corner kernel, see Remark \ref{GUEcorner}.
%\end{rem}

\begin{rem}\label{Cuspidal}
In \cite{OkRe}, Okounkov and Reshetikhin give a formula, without proof, for the correlation kernel in the
appropriate scaling limit around the type of cusp point studied in the present paper, but in a different model. The definition of the
kernel in \cite{OkRe} is somewhat formal due to the fact that the factor $\frac{1}{w-z}\frac{w^r}{z^s}$ in
formula (\ref{CAKernel}) above is interpreted via a ``time-ordered expansion'', see (13) in \cite{OkRe}. However,
for the contours used in their formula (18) these expansions are not convergent. Similarly, in our formula
(\ref{CAKernel}) above we can not expand $\frac{1}{w-z}$ in a power series when $z\in\mathscr{L}_L$ and $w\in\mathscr{L}_R$.
In this case it seems more natural to rewrite $\frac{1}{w-z}$ in a different way, see (\ref{IntTrick}) in section \ref{sec:rAiry} and
compare with the formulas derived there, see Proposition \ref{rAirycorrelationformula}.
\end{rem}

\subsection{Random Lozenge Tiling Model}
Consider three different types of {\em lozenges}
(rhombi with angles $\frac{\pi}3$ and $\frac{2\pi}3$) with sides of length
$1$. We label these as types Y, B, and R as shown in figure \ref{figRegHexLoz}.

\begin{figure}[H]
\centering
\begin{tikzpicture}[xscale=1/2,yscale=1/2]
%lozenges
\draw (8.5,{(1.5)*sqrt(3)}) \loza;
\draw (8.5,4.8) node {Y};
\draw (10.5,{(1.75)*sqrt(3)}) \lozb;
\draw (10.5,4.8) node {B};
\draw (12.5,{(1.75)*sqrt(3)}) \lozc;
\draw (13.5,4.8) node {R};
\end{tikzpicture}
\caption{Three types of lozenges with sides of length $1$.}
\label{figRegHexLoz}
\end{figure}
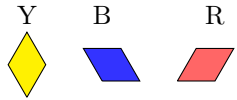

Consider a 'half-hexagon' as shown in figure \ref{HalfHex}.
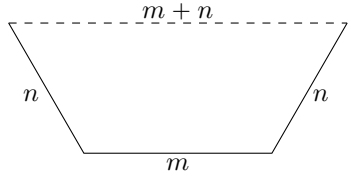
\begin{figure}[H]
\centering
\begin{tikzpicture}[xscale=1/2,yscale=1/2]

%half-hexagon
\draw (-13,{2*sqrt(3)})
--++ (2,{-2*sqrt(3)})
--++ (5,0)
--++ (2,{2*sqrt(3)});
\draw [dashed] (-13,{2*sqrt(3)}) --++ (9,0); 
\draw (-8.5,-.3) node {$m$};
\draw (-4.7,{sqrt(3)-.2}) node {$n$};
\draw (-8.5,{2*sqrt(3)+.3}) node {$m+n$};
\draw (-12.4,{sqrt(3)-.2}) node {$n$};
\end{tikzpicture}
\caption{A 'half-hexagon' of height $n$ and width $n+m$.}
\label{HalfHex}
\end{figure}
Suppose that we place $n$ lozenge tiles of type Y on the top line of the 'half-hexagon' according to figure \ref{figTopHalfHex}.

\begin{figure}[h]
\centering
\begin{tikzpicture}[xscale=1/4,yscale=1/4]
%Half-hexagon
\draw (-2.5,{(4)*sqrt(3)}) -- (12.5,{(4)*sqrt(3)});
\draw (1,0) -- (8.9,0);
\draw (-3,{(4)*sqrt(3)}) -- (1,0);
\draw (8.9,0) -- (13,{(4)*sqrt(3)});
%eight row
%\draw (-3.5,{(3.5)*sqrt(3)}) \lozb;
\draw (-2.5,{(3.5)*sqrt(3)}) \loza;
\draw (-1.5,{(3.5)*sqrt(3)}) \loza;
\draw (-.5,{(3.5)*sqrt(3)}) \loza;

\draw (3.5,{(3.5)*sqrt(3)}) \loza;

\draw (6.5,{(3.5)*sqrt(3)}) \loza;
\draw (7.5,{(3.5)*sqrt(3)}) \loza;

\draw (10.5,{(3.5)*sqrt(3)}) \loza;

\draw (12.5,{(3.5)*sqrt(3)}) \loza;

%equivalent interlaced particle system
%eight row
\draw (-2.5,{4*sqrt(3)}) circle (3pt);
\draw (-1.5,{4*sqrt(3)}) circle (3pt);
\draw (-.5,{4*sqrt(3)}) circle (3pt);
\draw (3.5,{4*sqrt(3)}) circle (3pt);
\draw (6.5,{4*sqrt(3)}) circle (3pt);
\draw (7.5,{4*sqrt(3)}) circle (3pt);
\draw (10.5,{4*sqrt(3)}) circle (3pt);
\draw (12.5,{4*sqrt(3)}) circle (3pt);

\end{tikzpicture}
\caption{$n$ tiles of type Y on the top line of the half-hexagon}
\label{figTopHalfHex}
\end{figure}
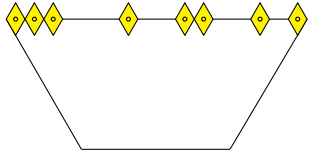
Fix this configuration of $n$ tiles of type Y on the top line, at say positions $\beta_1^{(n)},\beta_2^{(n)},....,\beta_n^{(n)}$, and consider all possible tessellations of the 'half-hexagon' with uniform probability distribution. A possible such tessellation is shown in figure \ref{figTilHalfHex}.

\begin{figure}[H]
\centering
\begin{tikzpicture}[xscale=1/3,yscale=1/3]
\draw (1,0) \lozb;
\draw (2,0) \lozb;
\draw (3,0) \lozb;
\draw (4,0) \loza;
\draw (4,0) \lozc;
\draw (5,0) \lozc;
\draw (6,0) \lozc;
\draw (7,0) \lozc;
\draw (8,0) \lozc;
%second row
%\draw (-.5,{sqrt(3)/2}) \lozb;
\draw (.5,{sqrt(3)/2}) \lozb;
\draw (1.5,{sqrt(3)/2}) \loza;
\draw (1.5,{sqrt(3)/2}) \lozc;
\draw (2.5,{sqrt(3)/2}) \lozc;
\draw (4.5,{sqrt(3)/2}) \lozb;
\draw (5.5,{sqrt(3)/2}) \lozb;
\draw (6.5,{sqrt(3)/2}) \loza;
\draw (6.5,{sqrt(3)/2}) \lozc;
\draw (7.5,{sqrt(3)/2}) \lozc;
\draw (8.5,{sqrt(3)/2}) \lozc;
%third row
%\draw (-1,{sqrt(3)}) \lozb;
\draw (0,{sqrt(3)}) \lozb;
\draw (1,{sqrt(3)}) \loza;
\draw (2,{sqrt(3)}) \lozb;
\draw (3,{sqrt(3)}) \lozb;
\draw (4,{sqrt(3)}) \loza;
\draw (4,{sqrt(3)}) \lozc;
\draw (5,{sqrt(3)}) \lozc;
\draw (7,{sqrt(3)}) \lozb;
\draw (8,{sqrt(3)}) \loza;
\draw (8,{sqrt(3)}) \lozc;
\draw (9,{sqrt(3)}) \lozc;
%fourth row
%\draw (-1.5,{(1.5)*sqrt(3)}) \lozb;
\draw (-.5,{(1.5)*sqrt(3)}) \lozb;
\draw (.5,{(1.5)*sqrt(3)}) \loza;
\draw (1.5,{(1.5)*sqrt(3)}) \lozb;
\draw (2.5,{(1.5)*sqrt(3)}) \loza;
\draw (2.5,{(1.5)*sqrt(3)}) \lozc;
\draw (4.5,{(1.5)*sqrt(3)}) \lozb;
\draw (5.5,{(1.5)*sqrt(3)}) \lozb;
\draw (6.5,{(1.5)*sqrt(3)}) \lozb;
\draw (7.5,{(1.5)*sqrt(3)}) \loza;
\draw (8.5,{(1.5)*sqrt(3)}) \lozb;
\draw (9.5,{(1.5)*sqrt(3)}) \loza;
\draw (9.5,{(1.5)*sqrt(3)}) \lozc;
%fifth row
%\draw (-2,{2*sqrt(3)}) \lozb;
\draw (-1,{2*sqrt(3)}) \loza;
\draw (-1,{2*sqrt(3)}) \lozc;
\draw (1,{2*sqrt(3)}) \loza;
\draw (1,{2*sqrt(3)}) \lozc;
\draw (3,{2*sqrt(3)}) \lozb;
\draw (4,{2*sqrt(3)}) \lozb;
\draw (5,{2*sqrt(3)}) \loza;
\draw (5,{2*sqrt(3)}) \lozc;
\draw (6,{2*sqrt(3)}) \lozc;
\draw (8,{2*sqrt(3)}) \loza;
\draw (8,{2*sqrt(3)}) \lozc;
\draw (10,{2*sqrt(3)}) \loza;
\draw (10,{2*sqrt(3)}) \lozc;
%sixth row
%\draw (-2.5,{(2.5)*sqrt(3)}) \lozb;
\draw (-1.5,{(2.5)*sqrt(3)}) \loza;
\draw (-.5,{(2.5)*sqrt(3)}) \lozb;
\draw (.5,{(2.5)*sqrt(3)}) \loza;
\draw (1.5,{(2.5)*sqrt(3)}) \lozb;
\draw (2.5,{(2.5)*sqrt(3)}) \lozb;
\draw (3.5,{(2.5)*sqrt(3)}) \loza;
\draw (3.5,{(2.5)*sqrt(3)}) \lozc;
\draw (5.5,{(2.5)*sqrt(3)}) \lozb;
\draw (6.5,{(2.5)*sqrt(3)}) \loza;
\draw (6.5,{(2.5)*sqrt(3)}) \lozc;
\draw (8.5,{(2.5)*sqrt(3)}) \loza;
\draw (8.5,{(2.5)*sqrt(3)}) \lozc;
\draw (10.5,{(2.5)*sqrt(3)}) \lozb;
\draw (11.5,{(2.5)*sqrt(3)}) \loza;
%seventh row
%\draw (-3,{3*sqrt(3)}) \lozb;
\draw (-2,{3*sqrt(3)}) \loza;
\draw (-1,{3*sqrt(3)}) \loza;
\draw (-1,{3*sqrt(3)}) \lozc;
\draw (1,{3*sqrt(3)}) \loza;
\draw (1,{3*sqrt(3)}) \lozc;
\draw (2,{3*sqrt(3)}) \lozc;
\draw (4,{3*sqrt(3)}) \lozb;
\draw (5,{3*sqrt(3)}) \loza;
\draw (5,{3*sqrt(3)}) \lozc;
\draw (7,{3*sqrt(3)}) \loza;
\draw (7,{3*sqrt(3)}) \lozc;
\draw (9,{3*sqrt(3)}) \lozb;
\draw (10,{3*sqrt(3)}) \loza;
\draw (10,{3*sqrt(3)}) \lozc;
\draw (12,{3*sqrt(3)}) \loza;
%eight row
%\draw (-3.5,{(3.5)*sqrt(3)}) \lozb;
\draw (-2.5,{(3.5)*sqrt(3)}) \loza;
\draw (-1.5,{(3.5)*sqrt(3)}) \loza;
\draw (-.5,{(3.5)*sqrt(3)}) \loza;
\draw (-.5,{(3.5)*sqrt(3)}) \lozc;
\draw (1.5,{(3.5)*sqrt(3)}) \lozb;
\draw (2.5,{(3.5)*sqrt(3)}) \lozb;
\draw (3.5,{(3.5)*sqrt(3)}) \loza;
\draw (3.5,{(3.5)*sqrt(3)}) \lozc;
\draw (5.5,{(3.5)*sqrt(3)}) \lozb;
\draw (6.5,{(3.5)*sqrt(3)}) \loza;
\draw (7.5,{(3.5)*sqrt(3)}) \loza;
\draw (7.5,{(3.5)*sqrt(3)}) \lozc;
\draw (8.5,{(3.5)*sqrt(3)}) \lozc;
\draw (10.5,{(3.5)*sqrt(3)}) \loza;
\draw (10.5,{(3.5)*sqrt(3)}) \lozc;
\draw (12.5,{(3.5)*sqrt(3)}) \loza;

%equivalent interlaced particle system
%bottom row
\filldraw (4,{sqrt(3)/2}) circle (3pt);
%second row
\filldraw (1.5,{sqrt(3)}) circle (3pt);
\filldraw (6.5,{sqrt(3)}) circle (3pt);
%third row
\filldraw (1,{(1.5)*sqrt(3)}) circle (3pt);
\filldraw (4,{(1.5)*sqrt(3)}) circle (3pt);
\filldraw (8,{(1.5)*sqrt(3)}) circle (3pt);
%fourth row;
\filldraw (.5,{2*sqrt(3)}) circle (3pt);
\filldraw (2.5,{2*sqrt(3)}) circle (3pt);
\filldraw (7.5,{2*sqrt(3)}) circle (3pt);
\filldraw (9.5,{2*sqrt(3)}) circle (3pt);
%fifth row
\filldraw (-1,{(2.5)*sqrt(3)}) circle (3pt);
\filldraw (1,{(2.5)*sqrt(3)}) circle (3pt);
\filldraw (5,{(2.5)*sqrt(3)}) circle (3pt);
\filldraw (8,{(2.5)*sqrt(3)}) circle (3pt);
\filldraw (10,{(2.5)*sqrt(3)}) circle (3pt);
%sixth row
\filldraw (-1.5,{3*sqrt(3)}) circle (3pt);
\filldraw (.5,{3*sqrt(3)}) circle (3pt);
\filldraw (3.5,{3*sqrt(3)}) circle (3pt);
\filldraw (6.5,{3*sqrt(3)}) circle (3pt);
\filldraw (8.5,{3*sqrt(3)}) circle (3pt);
\filldraw (11.5,{3*sqrt(3)}) circle (3pt);
%seventh row
\filldraw (-2,{(3.5)*sqrt(3)}) circle (3pt);
\filldraw (-1,{(3.5)*sqrt(3)}) circle (3pt);
\filldraw (1,{(3.5)*sqrt(3)}) circle (3pt);
\filldraw (5,{(3.5)*sqrt(3)}) circle (3pt);
\filldraw (7,{(3.5)*sqrt(3)}) circle (3pt);
\filldraw (10,{(3.5)*sqrt(3)}) circle (3pt);
\filldraw (12,{(3.5)*sqrt(3)}) circle (3pt);
%eight row
\draw (-2.5,{4*sqrt(3)}) circle (3pt);
\draw (-1.5,{4*sqrt(3)}) circle (3pt);
\draw (-.5,{4*sqrt(3)}) circle (3pt);
\draw (3.5,{4*sqrt(3)}) circle (3pt);
\draw (6.5,{4*sqrt(3)}) circle (3pt);
\draw (7.5,{4*sqrt(3)}) circle (3pt);
\draw (10.5,{4*sqrt(3)}) circle (3pt);
\draw (12.5,{4*sqrt(3)}) circle (3pt);

\draw (4.5,-.4) node {$9$};
\draw (11.5,{2*sqrt(3)-.2}) node {$8$};
\draw (-2.4,{2*sqrt(3)-.2}) node {$8$};

\end{tikzpicture}
\caption{An example of a tiling and its equivalent interlaced particle configuration
when $n=8$ and $m=9$. The unfilled circles represent the
deterministic lozenges/particles.}
\label{figTilHalfHex}
\end{figure}
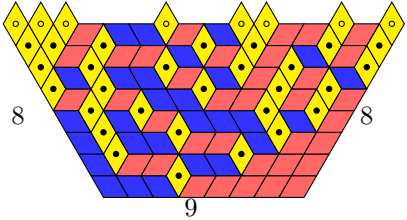
By considering the positions of the yellow tiles, such a tessellation can be encoded as an interlaced particle system. More precisely, let $y_i^{(r)}$ denote the position of the $i$:th particle on the $r$:th row. Then the particles on row $r+1$ will interlace with the particles on row $r$ according to  
\begin{align*}
y_1^{(r+1)}>y_1^{(r)}>y_2^{(r+1)}>y_2^{(r)}...>y_{r}^{(r)}>y_{r+1}^{(r+1)},
\end{align*}
for every $r=1,...,n-1$, where $y_i^{(n)}=\beta_i^{(n)}$. It will be convenient to make a coordinate transformation according to figure \ref{figRegHexLozTransform}. For more details see section 1.4 in \cite{Duse14a} and section 2.1 in \cite{Pet12} .
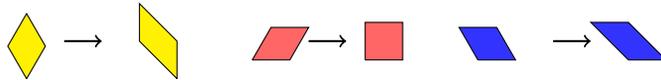
\begin{figure}[H]
\centering
\begin{tikzpicture}[xscale=1/2,yscale=1/2]
%lozenges
\draw (0-4,0) \loza;
\draw[thick,->] (-3,1) --(-2,1);
\draw (0,0) \lozyt;
\draw (2,0.5) \lozc;
\draw[thick,->] (3.5,1) --(4.5,1);
\draw (5,0.5) \lozrt;
\draw (8,0.5) \lozb;
\draw[thick,->] (10,1) --(11,1);
\draw (12,0.5) \lozbt;

\end{tikzpicture}
\caption{\label{figRegHexLozTransform}Coordinate transformation of lozenge tiles.}
\end{figure}

After the coordinate transformation, figure \ref{figTilHalfHex} becomes figure \ref{figTilHalfHexT}. Furthermore the interlacing condition between row $r+1$ and row $r$ has changed into
\begin{align*}
y_1^{(r+1)}\geq y_1^{(r)}>y_2^{(r+1)}\geq y_2^{(r)}...\geq y_{r}^{(r)}>y_{r+1}^{(r+1)}.
\end{align*}

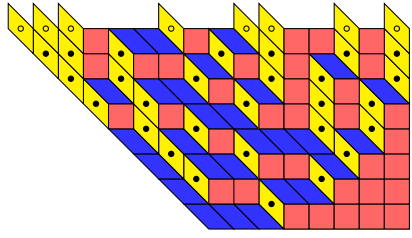
\begin{figure}[H]
\centering
\begin{tikzpicture}[xscale=1/3,yscale=1/3]
\draw (1,0) \lozbt;
\draw (2,0) \lozbt;
\draw (3,0) \lozbt;
\draw (4,0) \lozyt;
\draw (4,0) \lozrt;
\draw (5,0) \lozrt;
\draw (6,0) \lozrt;
\draw (7,0) \lozrt;
\draw (8,0) \lozrt;
%second row
%\draw (-.5,{sqrt(3)/2}) \lozb;
\draw (0,1) \lozbt;
\draw (1,1) \lozyt;
\draw (1,1) \lozrt;
\draw (2,1) \lozrt;
\draw (4,1) \lozbt;
\draw (5,1) \lozbt;
\draw (6,1) \lozyt;
\draw (6,1) \lozrt;
\draw (7,1) \lozrt;
\draw (8,1) \lozrt;

%third row
%\draw (-1,{sqrt(3)}) \lozb;
\draw (-1,2) \lozbt;
\draw (0,2) \lozyt;
\draw (1,2) \lozbt;
\draw (2,2) \lozbt;
\draw (3,2) \lozyt;
\draw (3,2) \lozrt;
\draw (4,2) \lozrt;
\draw (6,2) \lozbt;
\draw (7,2) \lozyt;
\draw (7,2) \lozrt;
\draw (8,2) \lozrt;
%fourth row
%\draw (-1.5,{(1.5)*sqrt(3)}) \lozb;
\draw (-2,3) \lozbt;
\draw (-1,3) \lozyt;
\draw (0,3) \lozbt;
\draw (1,3) \lozyt;
\draw (1,3) \lozrt;
\draw (3,3) \lozbt;
\draw (4,3) \lozbt;
\draw (5,3) \lozbt;
\draw (6,3) \lozyt;
\draw (7,3) \lozbt;
\draw (8,3) \lozyt;
\draw (8,3) \lozrt;
%fifth row
%\draw (-2,{2*sqrt(3)}) \lozb;
\draw (-3,4) \lozyt;
\draw (-3,4) \lozrt;
\draw (-1,4) \lozyt;
\draw (-1,4) \lozrt;
\draw (1,4) \lozbt;
\draw (2,4) \lozbt;
\draw (3,4) \lozyt;
\draw (3,4) \lozrt;
\draw (4,4) \lozrt;
\draw (6,4) \lozyt;
\draw (6,4) \lozrt;
\draw (8,4) \lozyt;
\draw (8,4) \lozrt;
%sixth row
%\draw (-2.5,{(2.5)*sqrt(3)}) \lozb;
\draw (-4,5) \lozyt;
\draw (-3,5) \lozbt;
\draw (-2,5) \lozyt;
\draw (-1,5) \lozbt;
\draw (0,5) \lozbt;
\draw (1,5) \lozyt;
\draw (1,5) \lozrt;
\draw (3,5) \lozbt;
\draw (4,5) \lozyt;
\draw (4,5) \lozrt;
\draw (6,5) \lozyt;
\draw (6,5) \lozrt;
\draw (8,5) \lozbt;
\draw (9,5) \lozyt;
%seventh row
%\draw (-3,{3*sqrt(3)}) \lozb;
\draw (-5,6) \lozyt;
\draw (-4,6) \lozyt;
\draw (-4,6) \lozrt;
\draw (-2,6) \lozyt;
\draw (-2,6) \lozrt;
\draw (-1,6) \lozrt;
\draw (1,6) \lozbt;
\draw (2,6) \lozyt;
\draw (2,6) \lozrt;
\draw (4,6) \lozyt;
\draw (4,6) \lozrt;
\draw (6,6) \lozbt;
\draw (7,6) \lozyt;
\draw (7,6) \lozrt;
\draw (9,6) \lozyt;
%eight row
%\draw (-3.5,{(3.5)*sqrt(3)}) \lozb;
\draw (-6,7) \lozyt;
\draw (-5,7) \lozyt;
\draw (-4,7) \lozyt;
\draw (-4,7) \lozrt;
\draw (-2,7) \lozbt;
\draw (-1,7) \lozbt;
\draw (0,7) \lozyt;
\draw (0,7) \lozrt;
\draw (2,7) \lozbt;
\draw (3,7) \lozyt;
\draw (4,7) \lozyt;
\draw (4,7) \lozrt;
\draw (5,7) \lozrt;
\draw (7,7) \lozyt;
\draw (7,7) \lozrt;
\draw (9,7) \lozyt;

%equivalent interlaced particle system
%bottom row
\filldraw (3.5,1)  circle (3pt);

%second row
%\draw (-.5,{sqrt(3)/2}) \lozb;

\filldraw (0.5,2) circle (3pt);
\filldraw (5.5,2) circle (3pt);

%third row
%\draw (-1,{sqrt(3)}) \lozb;

\filldraw (-0.5,3) circle (3pt);

\filldraw (2.5,3) circle (3pt);

\filldraw (6.5,3) circle (3pt);
%fourth row
%\draw (-1.5,{(1.5)*sqrt(3)}) \lozb;

\filldraw (-1.5,4) circle (3pt);

\filldraw (0.5,4) circle (3pt);

\filldraw (5.5,4) circle (3pt);

\filldraw (7.5,4) circle (3pt);
%fifth row
%\draw (-2,{2*sqrt(3)}) \lozb;
\filldraw (-3.5,5) circle (3pt);

\filldraw (-1.5,5) circle (3pt);

\filldraw (2.5,5) circle (3pt);

\filldraw (5.5,5) circle (3pt);

\filldraw (7.5,5) circle (3pt);

%sixth row
%\draw (-2.5,{(2.5)*sqrt(3)}) \lozb;
\filldraw (-4.5,6) circle (3pt);

\filldraw (-2.5,6) circle (3pt);

\filldraw (0.5,6) circle (3pt);

\filldraw (3.5,6) circle (3pt);

\filldraw (5.5,6) circle (3pt);

\filldraw (8.5,6) circle (3pt);

%seventh row
%\draw (-3,{3*sqrt(3)}) \lozb;
\filldraw (-5.5,7) circle (3pt);

\filldraw (-4.5,7) circle (3pt);

\filldraw (-2.5,7) circle (3pt);

\filldraw (1.5,7) circle (3pt);

\filldraw (3.5,7) circle (3pt);

\filldraw (6.5,7) circle (3pt);

\filldraw (8.5,7) circle (3pt);
%eight row
%\draw (-3.5,{(3.5)*sqrt(3)}) \lozb;
\draw (-6.5,8) circle (3pt);

\draw (-5.5,8) circle (3pt);

\draw (-4.5,8) circle (3pt);

\draw (-0.5,8) circle (3pt);

\draw (2.5,8) circle (3pt);

\draw (3.5,8) circle (3pt);

\draw (6.5,8) circle (3pt);

\draw (8.5,8) circle (3pt);

\end{tikzpicture}
\caption{\label{figTilHalfHexT}An example tiling and its equivalent interlaced particle configuration
after the coordinate transformation. The unfilled circles represent the
deterministic lozenges/particles.}
\end{figure}

\subsection{Interlacing Model}
We begin by briefly recalling the underlying probabilistic model described in \cite{Duse14a}.
A {\em discrete Gelfand-Tsetlin pattern} of depth $n$ is an $n$-tuple, denoted
$(y^{(1)},y^{(2)},\ldots,y^{(n)}) \in  \Z \times \Z^2 \times \cdots \times \Z^n$,
which satisfies the interlacing constraint
\begin{equation*}
y_1^{(r+1)} \; \ge \; y_1^{(r)} \; > \; y_2^{(r+1)} \; \ge \; y_2^{(r)}
\; > \cdots \ge \; y_r^{(r)} \; > \; y_{r+1}^{(r+1)},
\end{equation*}
denoted $y^{(r+1)} \succ y^{(r)}$, for all $r \in \{1,\ldots,n-1\}$.
For each $n\ge1$, fix $\beta^{(n)} \in \Z^n$ with $\beta_1^{(n)} > \beta_2^{(n)} > \cdots > \beta_n^{(n)}$,
and consider the following probability measure on the set of patterns of depth $n$:
\begin{equation*}
q_n(y^{(1)},\ldots,y^{(n)})
:= \frac1{Z_n} \cdot \left\{
\begin{array}{rcl}
1 & ; &
\text{when} \; \beta^{(n)} = y^{(n)} \succ y^{(n-1)} \succ \cdots \succ y^{(1)}, \\
0 & ; & \text{otherwise},
\end{array}
\right.
\end{equation*}
where $Z_n > 0$ is a normalisation constant. This can equivalently
be considered as a measure on configurations of interlaced particles in
$\Z \times \{1,\ldots,n\}$ by placing a particle at position
$(x,r) \in \Z \times \{1,\ldots,n\}$ whenever $x$ is an element of $y^{(r)}$. The measure
$q_n$ is then the uniform probability measure on the set of all such interlaced
configurations with the particles on the top row in the deterministic positions
defined by $\beta^{(n)}$. This measure also arises naturally from tiling
models as was indicated above.
In \cite{Duse14a} and \cite{Pet12} it was shown that this process is
determinantal. Note that the fixed top row and the
interlacing constraint implies that it is sufficient to restrict to those positions,
$(x_1,y_1), (x_2,y_2) \in \Z \times \{1,\ldots,n-1\}$, with $x_1 \ge \beta_n^{(n)}+n-y_1$ and
$x_2 \ge \beta_n^{(n)}+n-y_2$. For all such $(x_1,y_1)$ and $(x_2,y_2)$,  we give an integral representation of the correlation kernel $K_n((x_1,y_1),(x_2,y_2))$
in section \ref{sec:interlacingkernel}.

In terms of tiling models, $K_n((x_1,y_1)(x_2,y_2))$ is equal to a correlation kernel for the yellow particles $K_{\mathcal{Y}}^{(n)}((x_1,y_1),(x_2,y_2))$. However at the cusp one should not consider the correlation kernel of the yellow particles, but the correlation kernel of the red particles instead. The correlation kernels of the different particle species are related according to Lemma \ref{Lem:RedBlue} below, which we will prove in sec. 4.2.

\begin{Prop}
\label{Transform}
The red tiles (particles) form a determinantal point process with correlation kernel
\begin{align} 
\label{Red}
K_{\mathcal{R}}^{(n)}((x_1,y_1),(x_2,y_2))&=-1_{x_1<x_2}\frac{1}{(2\pi i)^2}\oint_{\mathscr{Z}_n}dz\oint_{\mathscr{W}_n}dw  \frac{\prod_{k=x_2+y_2-n}^{x_2-1}(z-k)}{\prod_{\substack{k=x_1+y_1-n}}^{x_1}(w-k)}\frac{1}{w-z}\prod_{i=1}^n\bigg(\frac{w-\beta_i^{(n)}}{z-\beta_i^{(n)}}\bigg) \\&
+1_{x_1\geq x_2}\frac{1}{(2\pi i)^2}\oint_{\mathscr{Z}'_n}dz\oint_{\mathscr{W}_n}dw  \frac{\prod_{k=x_2+y_2-n}^{x_2-1}(z-k)}{\prod_{\substack{k=x_1+y_1-n}}^{x_1}(w-k)}\frac{1}{w-z}\prod_{i=1}^n\bigg(\frac{w-\beta_i^{(n)}}{z-\beta_i^{(n)}}\bigg),\nonumber
\end{align}
where $\mathscr{Z}_n$ is a counterclockwise oriented contour containing $\{\beta_j^{(n)}: \beta_j^{(n)}\geq x_2\}$ but not the set $\{\beta_j^{(n)}\leq x_2-1\}$, and $\mathscr{Z}'_n$ is a counterclockwise oriented contour containing $\{\beta_j^{(n)}: \beta_j^{(n)}<x_2\}$ but not the set $\{\beta_j^{(n)}\geq x_2+1\}$. In addition, $\mathscr{W}_n$ contains the set 
$\{x_1+y_1-n,...,x_1\}$ and $\mathscr{Z}_n$ or $\mathscr{Z}'_n$. The integration contours are shown in figure \ref{figIntKernelContour1}. Here $A_n=\min_j\{\beta_j^{(n)}:j=1,...,n\}$ and $B_n=\max_j\{\beta_j^{(n)}:j=1,...,n\}$.
\end{Prop}

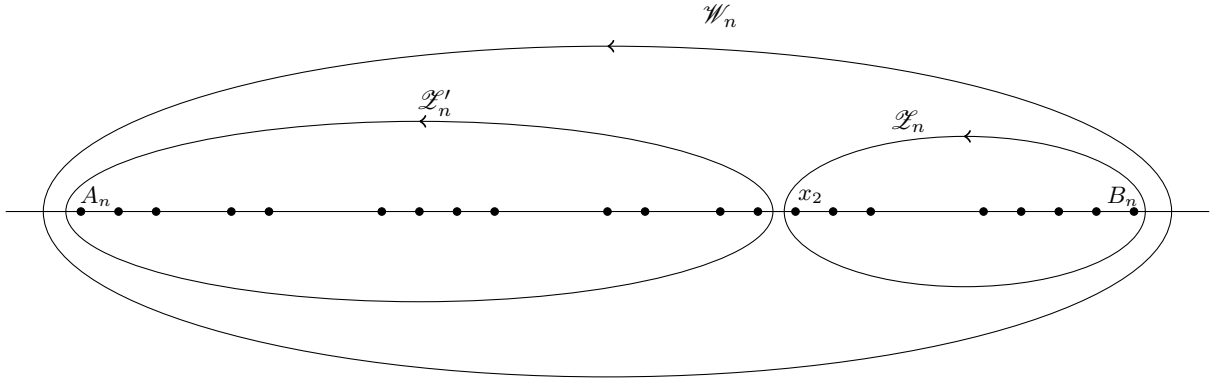
\begin{figure}[H]
\centering
\begin{tikzpicture}

\draw (-3.5,0) ellipse (4.7cm and 1.2cm);
\draw (3.75,0) ellipse (2.4cm and 1.0cm);
\draw (-1,0) ellipse (7.5cm and 2.2cm);
\draw(-9,0) --++(16,0);

\draw (-7.8,0.2) node {\small$A_n$};
\filldraw[fill=black](-8,0) circle (0.05 cm);
\filldraw[fill=black](-7.5,0) circle (0.05 cm);
\filldraw[fill=black](-7,0) circle (0.05 cm);
\filldraw[fill=black](-6,0) circle (0.05 cm);
\filldraw[fill=black](-5.5,0) circle (0.05 cm);
\filldraw[fill=black](-4,0) circle (0.05 cm);
\filldraw[fill=black](-3.5,0) circle (0.05 cm);
\filldraw[fill=black](-3,0) circle (0.05 cm);
\filldraw[fill=black](-2.5,0) circle (0.05 cm);
\filldraw[fill=black](-1,0) circle (0.05 cm);
\filldraw[fill=black](-0.5,0) circle (0.05 cm);
\filldraw[fill=black](0.5,0) circle (0.05 cm);
\filldraw[fill=black](1.0,0) circle (0.05 cm);
\filldraw[fill=black](1.5,0) circle (0.05 cm);
\filldraw[fill=black](2.0,0) circle (0.05 cm);
\filldraw[fill=black](2.5,0) circle (0.05 cm);
\filldraw[fill=black](4.0,0) circle (0.05 cm);
\filldraw[fill=black](4.5,0) circle (0.05 cm);
\filldraw[fill=black](5.0,0) circle (0.05 cm);
\filldraw[fill=black](5.5,0) circle (0.05 cm);
\filldraw[fill=black](6.0,0) circle (0.05 cm);
\draw (5.85,0.2) node {\small$B_n$};
\draw (-3.3,1.45) node {$\mathscr{Z}'_n$};
\draw (1.7,0.2) node {\small$x_2$};
\draw (3.0,1.2) node {$\mathscr{Z}_n$};
\draw (0.5,2.6) node {$\mathscr{W}_n$};
\draw[<-,thick] (-3.5, 1.2) -- (-3.49,1.2);
\draw[<-,thick] (-1, 2.2) -- (-0.99,2.2);
\draw[<-,thick] (3.74, 1.0) -- (3.75,1.0);

\end{tikzpicture}
\caption{Integration contours}
\label{figIntKernelContour1}
\end{figure}

\subsection{Asymptotic Geometry of Discrete Interlaced Patterns}

It is natural to consider the asymptotic behaviour of the determinantal process
introduced in the previous section as $n \to \infty$, under the assumption
that the (rescaled) empirical distribution of the fixed particles
on the top row converges weakly to a measure with compact support. More exactly:
\begin{Assump}
\label{A1}
Assume that
\begin{equation*}
\mu_n:=\frac1n \sum_{i=1}^n \delta_{\beta_i^{(n)}/n} \rightharpoonup \mu
\end{equation*}
as $n \to \infty$, in the sense of weak convergence of measures, where $\mu$ is a positive Borel measure on $\R$.
\end{Assump}

We see that $\mu\leq \lambda$ where $\lambda$ is Lebesgue measure
(recall $\beta^{(n)} \in \Z^n$), $\Vert \mu\Vert=1$, $\mu$ has compact support, and
$b-a>1$ where $[a,b]$ is the convex hull of $\supp(\mu)$. We write
$\mu \in \mathcal{M}_{c,1}^{\lambda}(\R)$. Additionally we note that $\mu$
admits a density w.r.t. $\lambda$, which is uniquely defined up to a set of
zero Lebesgue measure. Denoting the density by $\varphi$, it satisfies
$\varphi \in L^{\infty}(\R)$, $\varphi(x) = 0$ for all $x \in \R \setminus [a,b]$,
and $0 \le \varphi(x) \le 1$ for all $x \in [a,b]$. We write
$\varphi \in \rho_{c,1}^{\lambda}(\R)$. Note that
$\R \setminus \supp(\mu)$ is the largest open set on which $\varphi=0$ almost
everywhere, and $\R \setminus \supp(\lambda-\mu)$ is the largest open
set on which $\varphi=1$ almost everywhere.

Note that, rescaling the vertical and horizontal positions of the particles
of the Gelfand-Tsetlin patterns by $\frac1n$, the above assumption
and the interlacing constraint imply that the rescaled particles lie asymptotically in the the following set:
\begin{align}
\label{Polygon}
\mathcal{P}=\{(\chi,\eta)\in\R^2:a\leq \chi+\eta-1\leq \chi\leq b,0\leq\eta\leq 1\}.
\end{align} 
Fixing $(\chi,\eta) \in \mathcal{P}$, the local asymptotic behaviour of particles
near $(\chi,\eta)$ can be examined by considering the asymptotic behaviour of
$K_n((x_1^{(n)},y_1^{(n)}),(x_2^{(n)},y_2^{(n)}))$ as $n \to \infty$, where $\{(x_1^{(n)},y_1^{(n)})\}_{n\ge1} \subset \Z^2$
and $\{(x_2^{(n)},y_2^{(n)})\}_{n\ge1} \subset \Z^2$ satisfy
$
\frac{1}{n} (x_j^{(n)},y_j^{(n)})\to (\chi,\eta)
$
as $n \to \infty$, $j=1,2$. Assume this asymptotic behaviour, substitute $(x_1^{(n)},y_1^{(n)})$ and
$(x_2^{(n)},y_2^{(n)})$ into the expression (\ref{Red}) for the correlation kernel, and rescale the contours by $\frac1n$
to get,
\begin{align}
\label{eqJnrnunsnvn1}
K_{\mathcal{R}}^{(n)}((x_1^{(n)},y_1^{(n)}),(x_2^{(n)},y_2^{(n)}))
&= -1_{x_1<x_2}\frac{1}{(2\pi i)^2} \oint_{\frac{1}{n}\mathscr{Z}_n} dz \oint_{\frac{1}{n}\mathscr{W}_n} dw \;
\frac{\exp(n f_n(w) - n \tilde{f}_n(z))}{w-z}\nonumber\\
&+1_{x_1\geq x_2}\frac{1}{(2\pi i)^2} \oint_{\frac{1}{n}\mathscr{Z}_n'} dz \oint_{\frac{1}{n}\mathscr{W}_n} dw \;
\frac{\exp(n f_n(w) - n \tilde{f}_n(z))}{w-z},
\end{align}
for all $n \in \N$. Here,
\begin{align*}
f_n(w)
& := \frac1n \sum_{i=1}^n \log \bigg( w - \frac{\beta_i^{(n)}}n \bigg) -
\frac1n \sum_{j=x_1^{(n)}+y_1^{(n)}-n}^{x_1^{(n)}} \log \left( w - \frac{j}n \right), \\
\tilde{f}_n(z)
& := \frac1n \sum_{i=1}^n \log \bigg( z - \frac{\beta_i^{(n)}}n \bigg) -
\frac1n \sum_{j=x_2^{(n)}+y_2^{(n)}-n+1}^{x_2^{(n)} -1} \log \left( z - \frac{j}n \right).
\end{align*}
These expressions and assumption 1 leads us to define
\begin{equation}
\label{eqasymptoticFunction}
f(w;\chi,\eta) := \int_{\R}\log(w-t)d\mu(t)-\int_{\chi+\eta-1}^{\chi}\log(w-t)dt,
\end{equation}
for all $w \in \C \setminus \R$.
Steepest descent analysis and equation 
(\ref{eqJnrnunsnvn1}) suggest that, as $n \to \infty$,
the asymptotic behaviour of $ K_{\mathcal{R}}^{(n)}((x_1^{(n)},y_1^{(n)}),(x_2^{(n)},y_2^{(n)}))$ depends on the
behaviour of the roots of 
\begin{equation}
\label{eqf'}
f' (w;\chi,\eta)
:=\frac{d f}{dw}= \int_{\R}\frac{d\mu(t)}{w-t}-\int_{\chi+\eta-1}^{\chi}\frac{dt}{w-t},
\end{equation}
for $w \in \C \setminus \R$. In \cite{Duse14a}, we define the {\em liquid region},
$\LL$, as the set of all $(\chi,\eta) \in \mathcal{P}$ for which
$f_{(\chi,\eta)}'$ has a unique root in the upper-half plane
$\Hp := \{w \in \C: \text{Im}(w) > 0 \}$. Whenever $(\chi,\eta) \in \LL$, one
expects universal bulk asymptotic behaviour, i.e., that the local
asymptotic behaviour of the particles near $(\chi,\eta)$ are governed by
the  \emph{extended discrete sine kernel} as $n\to +\infty$. Also, one expects
that the particles are not asymptotically densely packed. Moreover, when
considering the corresponding tiling model and its associated height function,
one would expect to see the Gaussian Free Field asymptotically.
For a special case see \cite{Pet15}.

Let $W_\LL:\LL\to\Hp$ map $(\chi,\eta) \in \LL$ to the corresponding unique root
of $f_{(\chi,\eta)}'$ in $\Hp$. In \cite{Duse14a}, we show that $W_\LL$ is a
homeomorphism with inverse $W_\LL^{-1} (w) = (\chi_\LL(w), \eta_\LL(w))$ for all
$w \in \Hp$, where
\begin{align}
\label{InverseReal1}
\chi_\LL(w) &:= w + \frac{(w - \bar{w}) (e^{C(\bar{w})}-1)}{e^{C(w)} - e^{C(\bar{w})}}, \\
\label{InverseReal2}
\eta_\LL(w) &:=
1 + \frac{(w - \bar{w}) (e^{C(w)}-1) (e^{C(\bar{w})}-1) }{e^{C(w)} - e^{C(\bar{w})}},
\end{align}
and $C : \C \setminus \supp(\mu) \to \C$ is the {\em Cauchy} transform of $\mu$:
\begin{align}
\label{eqCauTrans}
C(w):=\int_{\R}\frac{d\mu(t)}{w-t},
\end{align}
Thus $\LL$ is a non-empty, open (w.r.t to $\R^2$), simply connected subset of $\mathcal{P}$.
\newline \newline 

In \cite{Duse14a} a subset of $\dv\mathcal{L}$ called \emph{the edge} $\mathcal{E}$ was determined and its geometry classified. More precisely, the boundary behavior of $W_{\mathcal{L}}^{-1}$ for the open subset $R\subset\dv\Hp=\R$ was studied, where 
\begin{equation}
\label{eqR}
R
:= (\R \setminus \supp(\mu)) \cup (\R \setminus \supp(\lambda-\mu)) \cup R_1 \cup R_2
= R_\mu \cup R_{\lambda-\mu} \cup R_0 \cup R_1 \cup R_2,
\end{equation}
where
\begin{itemize}
\item
$R_\mu := \{t \in \R \setminus \supp(\mu) : C(t) \neq 0\}$.
\item
$R_{\lambda-\mu} := \R \setminus \supp(\lambda-\mu)$.
\item
$R_0 := \{t \in \R \setminus \supp(\mu) : C(t) = 0\}$.
\item
$R_1$ is the set of all
$t \in \partial (\R \setminus \supp(\mu)) \cap \partial (\R \setminus \supp(\lambda-\mu))$
for which there exists an $\epsilon>0$ such that $(t,t+\epsilon) \subset \R \setminus \supp(\mu)$ and $(t-\epsilon,t) \subset \R \setminus \supp(\lambda-\mu)$.
\item
$R_2$ is the set of all
$t \in \partial (\R \setminus \supp(\mu)) \cap \partial (\R \setminus \supp(\lambda-\mu))$
for which there exists an $\epsilon>0$ such that $(t,t+\epsilon)
\subset \R \setminus \supp(\lambda-\mu)$ and $(t-\epsilon,t) \subset \R \setminus \supp(\mu)$.
\end{itemize}
Note that $R_1 \cap R_2 = \emptyset$. $R_1 \cup R_2
:= \partial (\R \setminus \supp(\mu)) \cap \partial (\R \setminus \supp(\lambda-\mu))$, the
set of all common boundary points of the disjoint open sets $\R \setminus \supp(\mu)$
and $\R \setminus \supp(\lambda-\mu)$. Also $R_1 \cup R_2$ is a discrete subset of $[a,b]$. In words, $R_\mu\cup R_0$ is the interior of the set where the density $\varphi(x)=0$ almost everywhere, and $R_{\lambda-\mu}$ is the interior of the set where $\varphi(x)=1$ almost everywhere. We see that $R_1$ are the jumps from $1$ to $0$ and
$R_2$ are the jumps from $0$ to $1$.

\begin{Def}
Let $t^*\in R_2$. We set
\begin{align}
\label{T1a}
 t_1:=\sup\{t\in \R: (t^*,t)\subset R_{\lambda-\mu}\}
\end{align}
and
\begin{align}
 \label{T2a}
 t_2:=\inf\{t\in \R: (t,t^*)\subset R_{\mu}\cup R_0\}.
\end{align}
Then in particular $\varphi(t)=0$ for all $t\in(t_2,t^*)$ and $\varphi(t)=1$ for all $t\in(t^*,t_1)$. If $t=t^*\in R_1$ we interchange $R_{\lambda-\mu}$ and $R_{\mu}\cup R_0$
in (\ref{T1a}) and (\ref{T2a})
\end{Def}

It was shown in \cite{Duse14a} that by considering a sequence $\{w_n\}_n\in \Hp$ such that $\lim_{n\rightarrow +\infty}w_n=t\in\R=\partial\Hp$ one gets a parametrization of the edge $\EE$,
\begin{align}
\label{Chi0}
\lim_{n\to\infty}\chi(w_n)&=t+\frac{1-e^{-C(t)}}{C'(t)}:=\chi_{\EE}(t)\\
\label{Eta0}
\lim_{n\to\infty}\eta(w_n)&=1+\frac{e^{C(t)}+e^{-C(t)}-2}{C'(t)}:=\eta_{\EE}(t)
\end{align}
for $t\in R_{\mu}$ and
\begin{align}
\label{Chi1}
\lim_{n\to\infty}\chi(w_n)&=t + \frac{ 1-(\frac{t-t_1}{t-t_2}) e^{-C_I(t)}}
{C_I'(t) + \frac{1}{t-t_2} - \frac{1}{t-t_1}}:=\chi_{\EE}(t)\\
\label{Eta1}
\lim_{n\to\infty}\eta(w_n)&=1 + \frac{(\frac{t-t_2}{t-t_1})e^{C_I(t)} +(\frac{t-t_1}{t-t_2})e^{-C_I(t)} }
{C_I'(t) + \frac{1}{t-t_2} - \frac{1}{t-t_1}}:=\eta_{\EE}(t)
\end{align}
for $t\in R_{\lambda- \mu}$, where $C_I(t)=\int_{\R\backslash I}\frac{d\mu(x)}{t-x}$, and $I$ is any open interval such that $\mu\big\vert_I=\lambda$ and $t\in I$. Moreover the collection of these smooth parametrized curves has analytic extensions across the set $R_1\cup R_2$. In particular,
\begin{align}
\label{EqCusp1}
(\chi_{\mathcal{E}}(t),\eta_{\mathcal{E}}(t))=(t,1-(t-t_2)e^{C_I(t)})
\end{align}
for $t\in R_1$ with $I=(t_2,t)$, and
\begin{align}
\label{EqCusp2}
(\chi_{\mathcal{E}}(t),\eta_{\mathcal{E}}(t))=(t+(t_1-t)e^{-C_I(t)},1-(t_1-t)e^{-C_I(t)})
\end{align}
for $t\in R_2$, with $I=(t,t_1)$. 
Finally, it is proven in \cite{Duse14a} that this gives a bijection
$W_{\mathcal{E}}:\mathcal{E}\to R$, with inverse $W_{\mathcal{E}}^{-1}(t)=(\chi_{\mathcal{E}}(t),\eta_{\mathcal{E}}(t))$, where $\chi_{\mathcal{E}}(t)$ and $\eta_{\mathcal{E}}(t)$ are real analytic functions. \newline \newline
The importance of the \emph{edge} $\mathcal{E}$ is that one expects universal edge fluctuations at $\mathcal{E}$. In particular one expects that the local asymptotics in a neighborhood of a generic point of $\mathcal{E}$ is either given by the $Airy$ kernel or the $Id-Airy$ kernel. For a special case see \cite{Pet12}, and more generally \cite{Duse15c}. In this paper we will consider certain singular points on the curve $\mathcal{E}$. At these points the curve will have a cusp. Typically one would expect that the local fluctuations at these points in the limit $n\to \infty$ is a \emph{Pearcey} process. However, for the situations considered in this paper this will not be the case. In fact we will show that at these points one gets the \emph{Cusp-Airy process} given by the kernel (\ref{CAKernel}).

\subsection{Conditions for Cusps}
One can rewrite the derivative given  by (\ref{eqf'}) as
\begin{align}
\label{AsympFunc1}
f'(w;\chi,\eta)=\int_{a}^{\chi+\eta-1}\frac{d\mu(x)}{w-x}-\int_{\chi+\eta-1}^{\chi}\frac{d(\lambda-\mu)(x)}{w-x}+\int_{\chi}^{b}\frac{d\mu(x)}{w-x}.
\end{align}
This implies that the function $f'(w;\chi,\eta)$ has an analytic extension to the set \newline$\C\backslash (\text{supp}(\mu)\big\vert_{[a,\chi+\eta-1,\chi]}\bigcup \text{supp}(\lambda-\mu)\big\vert_{[\chi+\eta-1,\chi]}\bigcup \text{supp}(\mu)\big\vert_{[\chi,b]})$. It is shown in Lemma 2.6 in \cite{Duse14a} that 
the \emph{edge,} $\EE$, is the disjoint union
$\EE := \EE_\mu \cup \EE_{\lambda-\mu} \cup \EE_0 \cup \EE_1 \cup \EE_2$, where
\begin{itemize}
\item
$\EE_\mu$ is the set of all $(\chi,\eta) \in\mathcal{P}$
for which $f'$ has a repeated root in $\R \setminus [\chi+\eta-1,\chi]$.
\item
$\EE_{\lambda-\mu}$ is the set of all $(\chi,\eta)$ for which $f'$ has a
repeated root in $(\chi+\eta-1,\chi)$.
\item
$\EE_0$ is the set of all $(\chi,\eta)$ for which $\eta=1$ and $f'$ has a root
at $\chi \; (=\chi+\eta-1)$. In particular, $\EE$ is tangent to the line $\eta=1$.
\item
$\EE_1$ is the set of all $(\chi,\eta)$ for which $\eta<1$ and $f'$ has a root at $\chi$. In particular, $\EE$ is tangent to the line $\chi=t\in R_1$.
\item
$\EE_2$ is the set of all $(\chi,\eta)$ for which $\eta<1$ and $f'$ has a root
at $\chi+\eta-1$. In particular, $\EE$ is tangent to the line $\chi+\eta-1=t\in R_2$.
\end{itemize}

Moreover it is shown in Lemma 2.6 in \cite{Duse14a} that one has following equivalent characterization of the behavior of the roots of $f'_t:=f'(w;\chi_\EE(t),\eta_\EE(t))\big\vert_{w=t}$ whenever $(\chi,\eta)\in \mathcal{E}:$
\begin{enumerate}
\item[(a)]
$(\chi_\EE(t), \eta_\EE(t)) \in \EE_\mu$ if and only if $t \in R_\mu$.
Moreover, in this case, $t$ is a root of $f_t'$ of multiplicity either
$2$ or $3$.
\item[(b)]
$(\chi_\EE(t), \eta_\EE(t)) \in \EE_{\lambda-\mu}$ if and only $t \in R_{\lambda-\mu}$.
Moreover, in this case, $t$ is a root of $f_t'$ of multiplicity either
$2$ or $3$.
\item[(c)]
$(\chi_\EE(t), \eta_\EE(t)) \in \EE_0$ if and only if $t \in R_0$. Moreover,
in this case, $f_t' = C$ and $t$ is a root of $f_t'$ of multiplicity $1$.
\item[(d)]
$(\chi_\EE(t), \eta_\EE(t)) \in \EE_1$ if and only $t \in R_1$. Moreover,
in this case, $t$ is a root of $f_t'$ of multiplicity either $1$ or $2$.
\item[(e)]
$(\chi_\EE(t), \eta_\EE(t)) \in \EE_2$ if and only $t \in R_2$. Moreover,
in this case, $t$ is a root of $f_t'$ of multiplicity either $1$ or $2$. 
\end{enumerate}
If case (a) holds and $t$ is a root of multiplicity 2 of $f_t'$ one expects to see the \emph{extended Airy kernel process}. If on the other hand $t$ is a root of multiplicity 3 of $f_t'$ one expects to see the \emph{Pearcey process}. If case (b) holds and $t$ is a root of multiplicity 2 of $f_t'$ one expects to see the \emph{Id-extended Airy kernel process}. 
What we mean by this is that we have to make a particle/hole transformation, i.e. change the type of tiles we are considering.
If on the other hand $t$ is a root of multiplicity 3 of $f_t'$ one again expects to see the \emph{Id-Pearcey process}. If case (c) holds one expects to see the \emph{GUE corner process}, and similarly if case (d) and (e) holds and $t$ is a root of multiplicity 1 of $f_t'$. In the remaining cases (d) and (e) when $t$ is a root of multiplicity 2 of $f_t'$ one will see the \emph{Cusp-Airy process}, which will be shown in this paper.

 In this article we will assume that $t=t_c\in R_1\cup R_2$ and that $t_c$ is a root of $f_t'$ of multiplicity $2$. If these conditions hold, then it is shown in Lemma 2.9 in \cite{Duse14a} that the edge at $(\chi_{\EE}(t),\eta_{\EE}(t))$ is locally an algebraic cusp of first order, that is the curve locally looks like the algebraic curve $y^3=x^2$ in a neighborhood of the origin. Moreover, if $t_c\in R_2$, then by Theorem 3.1 in \cite{Duse14a} $\chi_c>t_1$. Then (\ref{EqCusp2}) together with the fact that $\chi_c+\eta_c-1=t_c$, where $\chi_c:=\chi_{\mathcal{E}}(t_c)$ and $\eta_c:=\eta_{\mathcal{E}}(t_c)$, gives
\begin{align*}
f'(t_c;\chi_c,\eta_c)&=C_I(t_c)-\int_{t_1}^{\chi_c}\frac{dx}{t_c-x}=C_I(t_c)+\log\vert t_c-\chi_c\vert-\log\vert t_c-t_1\vert\\
&=C_I(t_c)+\log\vert (t_c-t_1)e^{-C_I(t_c)}\vert-\log\vert t_c-t_1\vert=0.
\end{align*}
A similar computation gives $f'(t_c;\chi_c,\eta_c)=0$ whenever $t_c\in R_1$. Therefore, $f'(t_c;\chi_c,\eta_c)\equiv0$ whenever $t_c\in R_1\cup R_2$. This implies that that $(\chi_c,\eta_c)$ is a cusp of $\mathcal{E}$ for $t_c\in R_1\cup R_2$, if and only if $f''(t_c;\chi_c,\eta_c)=0$. However, a direct computation shows that if $t_c\in R_2$, then
\begin{align*}
f''(t_c;\chi_c,\eta_c)=C_I'(t_c)+\frac{e^{C_I(t_c)}-1}{t_c-t_1}
\end{align*}
and if $t_c\in R_1$, then
\begin{align*}
f''(t_c;\chi_c,\eta_c)=C_I'(t_c)+\frac{1-e^{-C_I(t_c)}}{t_c-t_2}.
\end{align*}
Therefore, $\mathcal{E}$ has a cups at $(\chi_c,\eta_c)$, if and only if 
\begin{align}
\label{CupCond1}
C_I'(t_c)=\frac{e^{C_I(t_c)}-1}{t_1-t_c}
\end{align}
 if $t_c\in R_2$, and
\begin{align}
\label{CupCond2}
C_I'(t_c)=\frac{1-e^{-C_I(t_c)}}{t_2-t_c}.
\end{align}
 if $t_c\in R_1$.  
 From now on we will consider only the case $t_c\in R_2$. The case when $t_c\in R_1$ can be treated analogously. 
 
\begin{Assump}
\label{A2}
Let $t_c\in R_2$ and let $(\chi_c,\eta_c):=(\chi_\EE(t_c),\eta_\EE(t_c))$. Assume that $f'(t_c;\chi_c,\eta_c)=f''(t_c;\chi_c,\eta_c)=0$ so that $(\chi_c,\eta_c)$ is the asymptotic cusp point. 
\end{Assump}

\begin{Lem}
\label{Derivative3}
If $t_c\in R_2$ and $f''(t_c;\chi_c,\eta_c)=0$, then $f'''(t_c;\chi_c,\eta_c)>0$.
 \end{Lem}
\begin{proof}
By Lemma 2.9 in \cite{Duse14a}, the signed extrinsic curvature $k_{\EE}(t)$ is negative for all smooth points of the curve $\EE$. From this it follows that all cusps point into the liquid region $\LL$. Together with Lemma 2.9 case (9) and Lemma 2.8 formula (g) in \cite{Duse14a} it follows that $f'''(t_c;\chi_c,\eta_c)>0$.
\end{proof}

%\subsection{Convergence of the Support of the Empirical Measure}

In order to prove convergence to the Cusp-Airy kernel in our discrete model we will have to assume that $\mu_n$ will jump from empty to full not just asymptotically in terms of $\varphi$,
but already at the discrete level. This is the content of the next assumption.

\begin{Assump}
\label{A3}
Assume that for every $\eps>0$ and  $n$ large enough we have for $t_c\in R_2$,
\begin{align}
\label{assumpLocalRegR2}
\mu_n\big \vert_{[t_2+\eps,t_1-\eps]} =\frac{1}{n}\sum_{\lfloor nt_c\rfloor\leq k\leq n(t_1-\eps)}\delta_{k/n}.
\end{align} 
where $\lfloor x\rfloor=\max\{m\in \Z: m\leq x\}$.
\end{Assump}

\begin{rem}
Assuming only weak convergence of the empirical measures $\{\mu_n\}_n$ to a limiting measure $\mu$ will not be sufficient when considering fluctuation of the \emph{edge} $\EE$. We will need better control of the convergence of sequence of empirical measures. More precisely, it is necessary to assume that their supports converge in an appropriate sense, see \cite{Duse15c}. This will not be needed here since the other assumptions that we make are enough.
\end{rem}

\subsection{Rescaled Variables}

Introduce the following notation
\begin{align*}
a_n&:=n^{-1}\min\{t\in \text{supp}(\mu_n)\}=n^{-1}\beta_n^{(n)}
\end{align*}
and
\begin{align*}
b_n:=n^{-1}\max\{t\in \text{supp}(\mu_n)\}=n^{-1}\beta_1^{(n)}.
\end{align*}
Furthermore, we write
\begin{align}
\label{CriticalPoint}
t_c^{(n)}&:=n^{-1}\min\{\beta_i^{(n)}: \beta_i^{(n)}\geq \lfloor nt_c\rfloor\}=n^{-1}\min \{t\in \text{supp}(\mu_n):t\geq  \lfloor nt_c\rfloor\},
\end{align}
\begin{align}
\label{T1}
t_1^{(n)}+n^{-1}=n^{-1}\min\{t\in \text{supp}(\lambda_n-\mu_n):t>nt_c^{(n)}\},
\end{align}
and 
\begin{align}
\label{T2}
t_2^{(n)}=n^{-1}\max\{t\in \text{supp}(\mu_n):t<nt_c^{(n)}\}.
\end{align}
In words, $nt_1^{(n)}+1$ is the position of the first hole after $nt_c^{(n)}$, i.e., $nt_1^{(n)}$ is the position or the last particle in a densely packed block after $nt_c^{(n)}$, and $nt_2^{(n)}$ is the position of the last particle before $nt_c^{(n)}$. In particular $(nt_2^{(n)},nt_c^{(n)})$ is an empty block and $[nt_c^{(n)},nt_1^{(n)}]$ is a densely packed block. 
Finally, we define 
\begin{align}
\label{AsympCriticalPoint}
\text{$x_c^{(n)}:=n^{-1}\lfloor n\chi_c\rfloor$ and $y_c^{(n)}:=1+t_c^{(n)}-x_c^{(n)}$}.
\end{align}

We note that by Assumption \ref{A3}, it follows that $\lim_{n\to \infty}a_n=a$, $\lim_{n\to \infty}b_n=b$, $\lim_{n\to \infty}t_2^{(n)}=t_2$, $\lim_{n\to \infty}t_1^{(n)}=t_1$ and $\lim_{n\to \infty}t_c^{(n)}=t_c$.
Let $\sigma$ be a signed Borel measure on $\R$ and let 
\begin{align}
\label{AsympMeasureCuspLogPotential}
U^{\sigma}(x)=\int_{\R}\log\vert x-t \vert d\sigma(t)
\end{align}
denote the logarithmic potential of $\sigma$. 
\begin{rem}
Note that sometimes the logarithmic potential is defined with opposite sign. We will however follow the convention in \cite{Ransford}. 
\end{rem}
Consider the signed measure 
\begin{align}
\label{AsympMeasureCusp}
d\nu(x)=(\chi_{[a,t_2]}(x)+\chi_{[\chi_c,b](x)})\varphi(x)dx-(1-\varphi(x))\chi_{[t_1,\chi_c]}(x)dx.
\end{align}
Then at the cusp $\chi_c+\eta_c-1=t_c$, and
\begin{align}
\label{AsympFuncCusp}
f(w;\chi_c,\eta_c)=f(w;\chi_c)&=\int_{a}^{t_2}\log(w-x)d\mu(x)-\int_{t_1}^{\chi_c}\log(w-x)d(\lambda-\mu)(x)+\int_{\chi_c}^{b}\log(w-x)d\mu(x)\nonumber \\
&=\int_{\R}\log(w-x)d\nu(x)=U^{\nu}(w)+i\int_{\R}\text{arg}(w-t)d\nu(t),
\end{align}
where $\log$ denotes the principal branch of the complex logarithm function.

In particular we note that $f$ is independent of $\eta_c$. By assumption on $\mu$ we have a complete cancelation of the measures $\mu$ and $\lambda$ on $[t_c,t_1]$, that is $\lambda-\mu\big\vert_{[t_c,t_1]}=0$. For the exact kernel however, the cancelation of factors is not necessarily complete. 
Moreover due to the rigidity of the interlacing system, as shown in figure \ref{figInterlacing1}, their will be no fluctuations around the frozen boundary at the cusp in the orthogonal direction to the tangential direction of the cusp. It is therefore natural to assume a discrete variation in the orthogonal direction. We therefore assume the following scaling:
\begin{Def}\label{def:scaling}
Introduce the fixed limiting variables $r,s\in \Z$, and  $\xi,\tau\in \R$.
%\begin{equation}
%\label{rescaled}
 %\left\{
%\begin{array}{llll} r:=y_1+x_1-n-nt_c^{(n)}\\
%\xi_n:=\frac{1}{c_0n^{1/3}}(y_1-x_1-(ny_c^{(n)}-nx_c^{(n)}))\\
%s:=y_2+x_2-n-t_c^{(n)}\\
%\tau_n:= \frac{1}{c_0n^{1/3}}(y_2-x_2-ny_c^{(n)}+nx_c^{(n)})
 % \end{array} \right.,
%\end{equation}
Let
\begin{align}
\label{DefC0}
c_0=\frac{2(\chi_c-t_c)}{d_0},
\end{align}
where
\begin{align}
\label{DefD0}
d_0=\bigg(\frac{2}{f'''(t_c)}\bigg)^{1/3}.
\end{align}
Define the rescaled variables $\xi_n,\tau_n\in \R$, by
\begin{equation}
\label{coord}
 \left\{
\begin{array}{l}
x_1=nx_c^{(n)}+\frac{1}{2}\big(r-c_0n^{1/3}\xi_n\big)\\
y_1=ny_c^{(n)}+\frac{1}{2}\big(r+c_0n^{1/3}\xi_n\big)\\
x_2=nx_c^{(n)}+\frac{1}{2}\big(s-c_0n^{1/3}\tau_n\big)\\
y_2=ny_c^{(n)}+\frac{1}{2}\big(s+c_0n^{1/3}\tau_n\big)
 \end{array} \right..
\end{equation}
We assume that 
\begin{align}
\lim_{n \rightarrow\infty}\xi_n=\xi,\quad \lim_{n \rightarrow\infty}\tau_n&=\tau.
\end{align}
\end{Def}
When taking limits of the correlation kernel we will always use the scaling (\ref{coord}).

Note that by (\ref{AsympCriticalPoint}), the rescaled variables satisfy 
\begin{equation} \left\{
\begin{array}{llll} y_1+x_1=n+nt_c^{(n)}+r \\
y_1-x_1=ny_c^{(n)}-nx_c^{(n)}+c_0n^{1/3}\xi_n\\
y_2+x_2=n+nt_c^{(n)}+s \\
 y_2-x_2=ny_c^{(n)}-nx_c^{(n)}+c_0n^{1/3}\tau_n\\
  \end{array} \right..
\end{equation}

It will be convenient to introduce the notation
\begin{align}
\label{ScalingDeltaX}
\Delta x_1^{(n)}=\frac{1}{2}(r-c_0n^{1/3}\xi_n), \quad \Delta x_2^{(n)}=\frac{1}{2}(s-c_0n^{1/3}\tau_n).
\end{align}
The rescaled coordinate system is depicted in figure \ref{figRescaledCoord}.
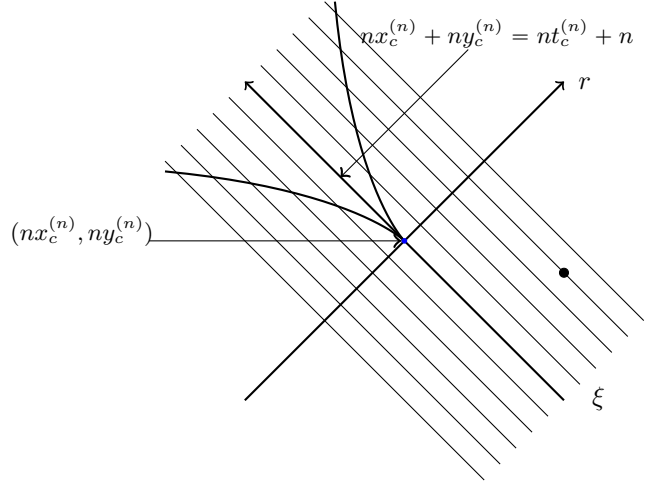
\begin{figure}[H]

\centering
\begin{tikzpicture}[xscale=0.6,yscale=0.6]

\draw[thick,->,rotate=-45] (5,0) -- (-5,0);
\draw[thick, ->,rotate=-45] (0,-5) -- (0,5);
\draw[rotate=-45] (-5,0.5) -- (5,0.5);
\draw[rotate=-45] (-5,1) -- (5,1);
\draw[rotate=-45] (-5,1.5) -- (5,1.5);
\draw[rotate=-45] (-5,2) -- (5,2);
\draw[rotate=-45] (-5,2.5) -- (5,2.5);
\draw[rotate=-45] (-5,-0.5) -- (5,-0.5);
\draw[rotate=-45] (-5,-1) -- (5,-1);
\draw[rotate=-45] (-5,-1.5) -- (5,-1.5);
\draw[rotate=-45] (-5,-2) -- (5,-2);
\draw[rotate=-45] (-5,-2.5) -- (5,-2.5);
\draw[thick,rotate=225,domain=-2.2:0] plot (0.25*\x^3, \x^2);
\draw[thick,rotate=225,domain=-2.2:0] plot (-0.25*\x^3, \x^2);
\draw (4,3.5) node {$r$};
\draw (4.3,-3.5) node {$\xi$};
\draw[rotate=-45] (-4,-4)--(0,0);
\draw[arrows=->,line width=1pt,rotate=-45] (-0.1,-0.1)--(-0.05,-0.05);
\filldraw[black,rotate=-45] (3,2) circle (0.1cm);
\filldraw[blue] (0,0) circle (0.05cm);
\draw[rotate=-45] (-5.2,-4.9) node {\small$(nx_c^{(n)},ny_c^{(n)})$};
\draw[rotate=-45] (-2,0)--(-2,4);
\draw[arrows=->,line width=1pt,rotate=-45] (-2,0.05)--(-2,0);
\draw[rotate=-45] (-1.8,4.7) node {\small$nx_c^{(n)}+ny_c^{(n)}=nt_c^{(n)}+n$};
\end{tikzpicture}
\caption{\label{figRescaledCoord} Rescaled coordinate system at the cusp.}
\end{figure}

\subsection{Integrand}

In order to write the integrand in (\ref{Red}) in a convenient way we introduce some notation. Let
\begin{align}
\label{Funcpole}
q_n(w;r)&=\frac{\prod_{k=nt_c^{(n)}}^{x_1}(w-k)}{\prod_{k=x_1+y_1-n}^{x_1}(w-k)}\\
&=1_{r>0}\prod_{k=nt_c^{(n)}}^{nt_c^{(n)}+r-1}(w-k)+1_{r=0}+1_{r<0}\prod_{k=nt_c^{(n)}+r}^{nt_c^{(n)}-1}(w-k)^{-1},\nonumber
\end{align}
and
\begin{align}
\label{FuncCont}
Q_n(w;\Delta x_1^{(n)})&=\frac{\prod_{k=nt_c^{(n)}}^{nx_c^{(n)}}(w-k)}{\prod_{k=nt_c^{(n)}}^{x_1}(w-k)}\\
&=1_{\Delta x_1^{(n)}>0}\prod_{k=nx_c^{(n)}+1}^{nx_c^{(n)}+\Delta x_1^{(n)}}(w-k)^{-1}+1_{\Delta x_1^{(n)}=0}+1_{\Delta x_1^{(n)}<0}\prod_{k=nx_c^{(n)}+\Delta x_1^{(n)}+1}^{nx_c^{(n)}}(w-k).\nonumber
\end{align}
We see that $q_n(w;r)$ only depends on the parameters through $r$. Also, since $x_1=nx_c^{(n)}+\Delta x_1^{(n)}$, $Q_n(w;\Delta x_1^{(n)})$ only depends on the parameters through $\Delta x_1^{(n)}$. 

Furthermore, let
\begin{align}
\label{FuncE}
E_n(w)=\frac{\prod_{\beta_i^{(n)}\leq nt_2^{(n)}}(w-\beta_i^{(n)})\prod_{nt_1^{(n)}<\beta_i^{(n)}}(w-\beta_i^{(n)})}{\prod_{k=nt_1^{(n)}+1}^{nx_c^{(n)}}(w-k)}.
\end{align}

\begin{Lem}
\label{RedCancel}
We have that 
\begin{align}
\label{lemCancelKernel}
K_{\mathcal{R}}^{(n)}((x_1,y_1),(x_2,y_2))&=-1_{x_1<x_2}\frac{1}{(2\pi i)^2}\oint_{\G_n}dz\oint_{\g_n} dw\frac{1}{(w-z)(z-n^{-1}x_2)}\frac{q_n(nw;r)Q_n(nw,\Delta x_1^{(n)})E_n(nw)}{q_n(nz;s)Q_n(nz;\Delta x_2^{(n)})E_n(nz)}\nonumber\\
&+1_{x_1\geq x_2}\frac{1}{(2\pi i)^2}\oint_{\tilde{\G}_n}dz\oint_{\tilde{\g}_n} dw\frac{1}{(w-z)(z-n^{-1}x_2)}\frac{q_n(nw;r)Q_n(nw,\Delta x_1^{(n)})E_n(nw)}{q_n(nz;s)Q_n(nz;\Delta x_2^{(n)})E_n(nz)},
\end{align}
where where $\Gamma_n$ is a counterclockwise oriented contour that contains the set $\{n^{-1}\beta_i^{(n)}\geq t_1^{(n)}\}$ but not the set $\{\beta_i^{(n)}\leq t_c^{(n)}+s\}$, and $\tilde{\Gamma}_n$ is a counterclockwise oriented contour containing the set $\{n^{-1}\beta_i^{(n)}\leq t_c^{(n)}+s\}$ but not the set $\{n^{-1}\beta_i^{(n)}\geq t_1^{(n)}\}$. In addition, $\gamma_n$ is the counterclockwise oriented contour that contains the set $n^{-1}\{x_1+y_1-n,...,x_1\}$ and $\Gamma_n$ and $\tilde{\g}_n$ is the counterclockwise oriented contour that contains the set $n^{-1}\{x_1+y_1-n,...,x_1\}$ and $\tilde{\Gamma}_n$. See figure \ref{figCancelledContour1} and \ref{figCancelledContour2}.
\end{Lem}
\begin{figure}[H]
\centering
\begin{tikzpicture}

\draw (0.7,0) ellipse (5.9cm and 2.2cm);
\draw (1.0,0) ellipse (4.2cm and 1.8cm);
\draw(-9,0) --++(16,0);

\draw[arrows=->,line width=1pt](0.56,2.2)--(0.55,2.2);
\draw (1.4,2.4) node {$\g_n$};
\draw[arrows=->,line width=1pt](0.61,1.8)--(0.6,1.8);
\draw (1.0,1.3) node {$\G_n$};

\draw (-8,0.5) node {$a_n$};
\filldraw[fill=red, draw=red] (-8,0) circle (0.05 cm);
\filldraw[fill=red, draw=red] (-7.8,0) circle (0.05 cm);
\filldraw[fill=red, draw=red] (-7.6,0) circle (0.05 cm);
\filldraw[fill=red, draw=red] (-7.4,0) circle (0.05 cm);
\filldraw[fill=red, draw=red] (-7.2,0) circle (0.05 cm);
\filldraw[fill=red, draw=red] (-7.0,0) circle (0.05 cm);
\filldraw[fill=red, draw=red] (-6.8,0) circle (0.05 cm);
\filldraw[fill=red, draw=red] (-6.6,0) circle (0.05 cm);
\filldraw[fill=red, draw=red] (-6.4,0) circle (0.05 cm);
\draw (-6.4,0.5) node {$t_2^{(n)}$};
\draw (-4.5,0.5) node {$t_c^{(n)}$};
\draw(-4.5,-.1) --++(0,.2);
\filldraw[fill=red, draw=red] (-4.4,0) circle (0.05 cm);
\filldraw[fill=blue, draw=blue] (-4.6,0) circle (0.05 cm);
\filldraw[fill=blue, draw=blue] (-4.8,0) circle (0.05 cm);
\draw (-1.8,0.5) node {$t_1^{(n)}$};
\filldraw[fill=blue, draw=blue] (-1.8,0) circle (0.05 cm);
\filldraw[fill=blue, draw=blue] (-1.6,0) circle (0.05 cm);
\filldraw[fill=blue, draw=blue] (-1.2,0) circle (0.05 cm);
\filldraw[fill=blue, draw=blue] (-0.9,0) circle (0.05 cm);
\filldraw[fill=blue, draw=blue] (-0.6,0) circle (0.05 cm);
\filldraw[fill=blue, draw=blue] (-0.4,0) circle (0.05 cm);
\draw (-0.4,0.5) node {$x_c^{(n)}$};
\filldraw[fill=red, draw=red] (-0.2,0) circle (0.05 cm);
\filldraw[fill=red, draw=red] (0,0) circle (0.05 cm);
\filldraw[fill=red, draw=red] (0.2,0) circle (0.05 cm);
\filldraw[fill=red, draw=red] (0.5,0) circle (0.05 cm);
\filldraw[fill=red, draw=red] (0.8,0) circle (0.05 cm);
\filldraw[fill=red, draw=red] (1.1,0) circle (0.05 cm);
\filldraw[fill=red, draw=red] (1.4,0) circle (0.05 cm);
\filldraw[fill=red, draw=red] (1.8,0) circle (0.05 cm);
\filldraw[fill=red, draw=red] (2.2,0) circle (0.05 cm);
\filldraw[fill=red, draw=red] (2.4,0) circle (0.05 cm);

\filldraw[fill=red, draw=red] (3.2,0) circle (0.05 cm);
\filldraw[fill=red, draw=red] (3.4,0) circle (0.05 cm);
\filldraw[fill=red, draw=red] (3.6,0) circle (0.05 cm);
\filldraw[fill=red, draw=red] (3.8,0) circle (0.05 cm);
\filldraw[fill=red, draw=red] (4,0) circle (0.05 cm);
\filldraw[fill=red, draw=red] (4.2,0) circle (0.05 cm);
\filldraw[fill=red, draw=red] (4.4,0) circle (0.05 cm);
\filldraw[fill=red, draw=red] (4.6,0) circle (0.05 cm);
\draw (4.6,0.5) node {$b_n$};

\end{tikzpicture}
\caption{\protect \label{figCancelledContour1} Integration contours for the correlation kernel. Blue dots indicate the positions of poles and red dots indicate the position of zeroes of the function $q_n(nw;r)E_n(nw)$.}
\end{figure}

\begin{figure}[H]

\centering
\begin{tikzpicture}

\draw (-2,0) ellipse (6.8cm and 2.4cm);
\draw (-5.7,0) ellipse (2.6cm and 1.4cm);
\draw(-10,0) --++(16,0);

\draw[arrows=->,line width=1pt](0.56,2.2)--(0.55,2.2);
\draw (1.4,2.4) node {$\tilde{\g}_n$};
\draw[arrows=->,line width=1pt](-5.69,1.4)--(-5.7,1.4);
\draw (-5.1,1.6) node {$\tilde{\G}_n$};

\draw (-8,0.3) node {$a_n$};
\filldraw[fill=red, draw=red] (-8,0) circle (0.05 cm);
\filldraw[fill=red, draw=red] (-7.8,0) circle (0.05 cm);
\filldraw[fill=red, draw=red] (-7.6,0) circle (0.05 cm);
\filldraw[fill=red, draw=red] (-7.4,0) circle (0.05 cm);
\filldraw[fill=red, draw=red] (-7.2,0) circle (0.05 cm);
\filldraw[fill=red, draw=red] (-7.0,0) circle (0.05 cm);
\filldraw[fill=red, draw=red] (-6.8,0) circle (0.05 cm);
\filldraw[fill=red, draw=red] (-6.6,0) circle (0.05 cm);
\filldraw[fill=red, draw=red] (-6.4,0) circle (0.05 cm);
\draw (-6.4,0.5) node {$t_2^{(n)}$};
\draw (-4.5,0.5) node {$t_c^{(n)}$};
\draw(-4.5,-.1) --++(0,.2);
\filldraw[fill=red, draw=red] (-4.4,0) circle (0.05 cm);
\filldraw[fill=blue, draw=blue] (-4.6,0) circle (0.05 cm);
\filldraw[fill=blue, draw=blue] (-4.8,0) circle (0.05 cm);
\draw (-1.8,0.5) node {$t_1^{(n)}$};
\filldraw[fill=blue, draw=blue] (-1.8,0) circle (0.05 cm);
\filldraw[fill=blue, draw=blue] (-1.6,0) circle (0.05 cm);
\filldraw[fill=blue, draw=blue] (-1.2,0) circle (0.05 cm);
\filldraw[fill=blue, draw=blue] (-0.9,0) circle (0.05 cm);
\filldraw[fill=blue, draw=blue] (-0.6,0) circle (0.05 cm);
\filldraw[fill=blue, draw=blue] (-0.4,0) circle (0.05 cm);
\draw (-0.4,0.5) node {$x_c^{(n)}$};
\filldraw[fill=red, draw=red] (-0.2,0) circle (0.05 cm);
\filldraw[fill=red, draw=red] (0,0) circle (0.05 cm);
\filldraw[fill=red, draw=red] (0.2,0) circle (0.05 cm);
\filldraw[fill=red, draw=red] (0.5,0) circle (0.05 cm);
\filldraw[fill=red, draw=red] (0.8,0) circle (0.05 cm);
\filldraw[fill=red, draw=red] (1.1,0) circle (0.05 cm);
\filldraw[fill=red, draw=red] (1.4,0) circle (0.05 cm);
\filldraw[fill=red, draw=red] (1.8,0) circle (0.05 cm);
\filldraw[fill=red, draw=red] (2.2,0) circle (0.05 cm);
\filldraw[fill=red, draw=red] (2.4,0) circle (0.05 cm);

\filldraw[fill=red, draw=red] (3.2,0) circle (0.05 cm);
\filldraw[fill=red, draw=red] (3.4,0) circle (0.05 cm);
\filldraw[fill=red, draw=red] (3.6,0) circle (0.05 cm);
\filldraw[fill=red, draw=red] (3.8,0) circle (0.05 cm);
\filldraw[fill=red, draw=red] (4,0) circle (0.05 cm);
\filldraw[fill=red, draw=red] (4.2,0) circle (0.05 cm);
\filldraw[fill=red, draw=red] (4.4,0) circle (0.05 cm);
\filldraw[fill=red, draw=red] (4.6,0) circle (0.05 cm);
\draw (4.45,0.3) node {$b_n$};

\end{tikzpicture}
\caption{\label{figCancelledContour2}  Integration contours for the correlation kernel. Blue dots indicate the positions of poles and red dots indicate the position of zeroes of the function $q_n(nw;r)E_n(nw)$.}
\end{figure}
The lemma will be proven in section 2.1. 

Let $\nu_n$ and $\nu_{n,j}$, $j=1,2$, be the signed measures
\begin{align}
\label{SignedMeasure}
&\nu_n(t)=\frac{1}{n}\sum_{\beta_i^{(n)}\leq nt_2^{(n)}}\delta_{\beta_i^{(n)}/n}-\frac{1}{n}\bigg(\sum_{k=nt_1^{(n)}+1}^{nx_c^{(n)}}\delta_{k/n}-\sum_{nt_1^{(n)}<\beta_i^{(n)}\leq nx_c^{(n)}}\delta_{\beta_i^{(n)}/n}\bigg)+\frac{1}{n}\sum_{nx_c^{(n)}<\beta_i^{(n)}}\delta_{\beta_i^{(n)}/n},\\
\label{SignedMeasure1}
&\nu_{n,j}(t)=\frac{1}{n}\sum_{\beta_i^{(n)}\leq n t_2^{(n)}}\delta_{\beta_i^{(n)}/n}-\frac{1}{n}\bigg(\sum_{k=nt_1^{(n)}+1}^{x_j^{(n)}}\delta_{k/n}-\sum_{nt_1^{(n)}<\beta_i^{(n)}\leq x_j^{(n)}}\delta_{\beta_i^{(n)}/n}\bigg)+\frac{1}{n}\sum_{x_c^{(n)}<\beta_i^{(n)}}\delta_{\beta_i^{(n)}/n}.
\end{align}
Furtheremore, define $f_n$, $g_{n,j}$ and $h_{n,j}$ by
\begin{align}
\label{FuncNonAsymp1}
f_n(w)&=\frac{1}{n}\log E_n(nw)=\int_{\R}\log(w-t)d\nu_n(t),\\
\label{FuncNonAsymp2}
g_{n,j}(w)&=\frac{1}{n}\log Q_n(nw;\Delta x_j^{(n)})+f_n(w):=\frac{1}{n}h_{n,j}(w)+f_n(w).
\end{align}
In particular, Re$[g_{n,j}(w)]=\int_{\R}\log\vert w-t\vert d\nu_{n,j}(t)=U^{\nu_{n,j}}(w)$. Here, we let $\log$ be the principal branch of the logaritm for $f_n(w)$. For $h_{n,j}(w)$, we let the branch cut of $\log$ lie along the the \emph{positive} real axis. With these choices of branch cuts, it follows that $f_n$ is analytic on $\C\backslash (-\infty, b_n]$, and $h_{n,j}$ is analytic in $\C\backslash [t_1^{(n)},+\infty)$. In particular we note that $f_n(w)$ has a jump discontinuity over the real line at $t_c$. However Re[$f_n$] is real analytic on $\C\backslash ([a_n,t_2^{(n)}]\cup [t_1^{(n)},b_n])$, and $f'_n$ is homolorphic on $\C\backslash ([a_n,t_2^{(n)}]\cup [t_1^{(n)},b_n])$. Moreover, $\displaystyle\lim_{t_c\pm i\eps\to t_c}f(t_c\pm i\eps):=f_n^{\pm}(t_c)=\text{Re}[f_n(t_c)]\pm i\pi k/n$, where $k$ is an integer. Therefore 
$\exp(n(f_n^{+}(t_c)-f_n^-(t_c)))=\exp(2\pi i k)=1$. Thus the jump discontinuity in the imaginary part of $f_n$ does not matter when we perform a steepest descent analysis in a neighborhood of $t_c$. From now on it will be understood that when we Taylor expand $f_n$ at $w=t_c$, we look at the branch
\begin{align*}
f_n^{+}(t_c)+\sum_{n=1}^{\infty}\frac{f^{(n)}(t_c)}{n!}(w-t_c)^n.
\end{align*}
It follows from our assumptions that $f_n$ and $g_{n,j}$ converge uniformly to $f$ on compact subsets of $\C\backslash ([a,t_2]\cup[t_1,b])$, see Lemma \ref{UniformConv}
\newline\newline \indent
We need one more assumption which will enable us to replace the non-asymptotic function $f_n(z)$ by the asymptotic function $f(z)$ in a neighbourhood of the critical point $t_c$. 
\begin{Assump}
\label{A4}
There is a neighbourhood $U$ of $t_c$ such that
\begin{align}
\label{assumpConvRate}
&\lim_{n\to \infty} n^{2/3}(f_n'(z)-f'(z))=0
\end{align}
uniformly in $U$.
\end{Assump}

\subsection{Main Theorem}

We can now give the main theorem about convergence to the Cusp-Airy kernel for our system of red particles in the interlacing model.

\begin{Thm}
\label{thm}
Assume that the sequence of empirical measures $\{\mu_n\}_n$ satisfies Assumptions \ref{A1}-\ref{A4}, and assume the scaling in Definition \ref{def:scaling}. Then
\begin{align} 
\lim_{n\to \infty}\frac{p_n(x_2,y_2)}{p_n(x_1,y_1)}\frac{c_0}{2}n^{1/3}K_{\mathcal{R}}^{(n)}((x_1,y_1),(x_2,y_2))=\mathcal{K}_{CA}((\xi,r),(\tau,s))
\end{align}
uniformly for $\xi$ and $\tau$ in some fixed compact subset of $\R$, where 
\begin{align}\label{Conjugationfactor}
p_n(x_1,y_1)=\bigg(d_0n^{\frac{2}{3}}\bigg)^{x_1+y_1-n-nt_c^{(n)}}Q_n(nt_c;x_1-nx_c^{(n)}). 
\end{align} 
Let $(\xi^r_j,r)$ be the rescaled coordinates for particles on line $r$. Fix $r_1,\dots,r_M\in\Z$ and let $\phi\,:\,\R\times\{r_1,\dots,r_M\}\to[0,1]$ be a bounded measurable function
with compact support. Let $\mathbb{E}_{\beta^{(n)}}$ denote the expectation with respect to the determinantal point process with kernel $K_{\mathcal{R}}^{(n)}$. Then,
\begin{align} \label{Weaklimit}
\lim_{n\to \infty}\mathbb{E}_{\beta^{(n)}}\bigg[\prod_{r\in\{r_1,\dots,r_m\}}\prod_{j}(1-\phi(\xi_j^r,r))\bigg]=\det(I-\phi \mathcal{K}_{CA})_{L^2(\R\times \{r_1,...,r_M\})}.
\end{align} 
\end{Thm}

\begin{ex}
\label{exPetrov}
Let the limiting measure $d\mu(t)=\chi_{[-1-a,-1]}(t)dt+\chi_{[0,a]}(t)dt+\chi_{[2a,1]}(t)dt$ where $a=\frac{-3+\sqrt{17}}{4}\approx 0.28$. In particular $\Vert \mu\Vert=1$ and $t_c=0\in R_2$. Let $t_1=a$. Then, with $I=(0,a)$,
%\begin{align*}
%C(t)=\int_{-1-a}^{-a}\frac{dy}{t-y}+\int_{0}^{a}\frac{dy}{t-y}+\int_{2a}^{1}\frac{dy}{t-y}
%=\log\bigg\vert \frac{t+1+a}{t+1}\bigg\vert +\log\bigg\vert \frac{t}{t-x}\bigg\vert+\log\bigg\vert \frac{t-2a}{t-1}\bigg\vert
%\end{align*}
%and 
\begin{align*}
C_I(t)&=\log\bigg\vert \frac{t+1+a}{t+1}\bigg\vert +\log\bigg\vert \frac{t-2a}{t-1}\bigg\vert.
\end{align*}
From equation (\ref{EqCusp2}) we then get
$
(\chi_{\mathcal{E}}(0),\eta_{\mathcal{E}}(0))=\bigg(\frac{1}{2(1+a)},1-\frac{1}{2(1+a)}\bigg).
$
Using that $a<\frac{1}{2(1+a)}<2a$ and (\ref{AsympFunc1}), we obtain
\begin{align*}
f''(0;\chi_{\mathcal{E}}(0),\eta_{\mathcal{E}}(0))=\int_{-1-a}^{-1}\frac{ds}{s^2}-\int_{a}^{\chi_{\mathcal{E}}(0)}\frac{ds}{s^2}+\int_{2a}^{1}\frac{ds}{s^2}
=\frac{1-a-8a^2-4a^3}{2a(1+a)}=0,
\end{align*}
since $a$ is a root of the polynomial $1-x-8x^2-4x^3$. Thus Assumption \ref{A2} is satisfied. We will now construct a sequence of empirical measures $\{\mu_n\}_n$ that satisfies Assumptions \ref{A1},\ref{A3} and \ref{A4}. 
Let
\begin{align*}
\mu_n=\frac{1}{n}\sum_{k=\lfloor -n(1-a)\rfloor}^{ -n+d_n}\delta_{k/n}+\frac{1}{n}\sum_{k=0}^{\lfloor na\rfloor}\delta_{k/n}+\frac{1}{n}\sum_{k=\lfloor 2na\rfloor}^{n}\delta_{k/n}
\end{align*}
where $d_n$ is chosen so that $ -n+d_n-\lfloor -n(1-a)\rfloor+\lfloor na\rfloor+n-\lfloor 2na\rfloor+3=n$. 
%In particular $\vert d_n\vert \leq 3$. Let $\varphi\in C_c(\R)$. Then
%\begin{align*}
%\int_{\R}\varphi(t)d\mu_n(t)&=\frac{1}{n}\sum_{k=\lfloor -n(1-a)\rfloor}^{ -n+d_n}\varphi(k/n)+\frac{1}{n}\sum_{k=0}^{\lfloor na\rfloor}\varphi(k/n)+\frac{1}{n}\sum_{k=\lfloor 2na\rfloor}^{n}\varphi(k/n)\\
%&\to \int_{-1-a}^{-1}\varphi(t)dt+\int_{0}^{a}\varphi(t)dt+\int_{2a}^{1}\varphi(t)dt=\int_{\R}\varphi(t)d\mu(t)
%\end{align*}
%as $n\to\infty$, since $\varphi$ is Riemann integrable. 
Clealy, $\mu_n\rightharpoonup \mu$, so Assumption \ref{A1} is satisfied. Now,
\begin{align*}
\text{supp}(\mu_n)=\frac{1}{n}\{\lfloor -n(1-a)\rfloor,...,-n+d_n\}\bigcup\frac{1}{n}\{0,...,\lfloor na\rfloor\}\bigcup\frac{1}{n}\{\lfloor 2na\rfloor,...,n\},
\end{align*}
so by construction $\{\mu_n\}_n$ satisfies Assumption \ref{A3}. Now, 
\begin{align*}
\nu_n=\frac{1}{n}\sum_{k=\lfloor -n(1-a)\rfloor}^{ -n+d_n}\delta_{k/n}-\frac{1}{n}\sum_{k=\lfloor na\rfloor+1}^{\big\lfloor \frac{n}{2(1+a)}\big\rfloor}\delta_{k/n}+\frac{1}{n}\sum_{k=\lfloor 2na\rfloor}^{n}\delta_{k/n},
\end{align*}
and thus
\begin{align}
\label{exConvf1}
f_n'(w)&=\int_{\R}\frac{d\nu_n(t)}{w-t}=\frac{1}{n}\sum_{k=\lfloor -n(1-a)\rfloor}^{ -n+d_n}\frac{1}{w-k/n}-\frac{1}{n}\sum_{k=\lfloor na\rfloor+1}^{\big\lfloor \frac{n}{2(1+a)}\big\rfloor}\frac{1}{w-k/n}+\frac{1}{n}\sum_{k=\lfloor 2na\rfloor}^{n}\frac{1}{w-k/n},
\end{align}
and
\begin{align}
\label{exConvf2}
f_n''(w)&=-\int_{\R}\frac{d\nu_n(t)}{(w-t)^2}=-\frac{1}{n}\sum_{k=\lfloor -n(1-a)\rfloor}^{ -n+d_n}\frac{1}{w-k/n}+\frac{1}{n}\sum_{k=\lfloor na\rfloor+1}^{\big\lfloor \frac{n}{2(1+a)}\big\rfloor}\frac{1}{(w-k/n)^2}-\frac{1}{n}\sum_{k=\lfloor 2na\rfloor}^{n}\frac{1}{(w-k/n)^2}.
\end{align}
Choose $r<a$. Then for $n$ large enough, $\frac{1}{w-t}$ and $-\frac{1}{(w-t)^2}$ are continuously differentiable in $B(0,r)$. Hence (\ref{exConvf1}) and (\ref{exConvf2}) are Riemann sums of smooth functions with equidistant partitions. Thus, there exists a constant $C>0$, such that
\begin{align*}
\vert f_n'(z)-f'(z)\vert \leq \frac{C}{n},%\quad \vert f_n''(w)-f''(w)\vert \leq \frac{C}{n},
\end{align*}
holds uniformly in $B(0,r)$ for $n$ large enough. In particular Assumption \ref{A4} is satisfied. This example shows that we indeed can get the Cusp-Airy limit in a natural model of the type considered in \cite{Pet12}.
\end{ex}

\begin{rem}
The regularity assumption on  the sequence of empirical measures made in Assumption \ref{A3} are necessary for Theorem 1 to hold, and cannot be substantially relaxed. If one in particular try to perform the cancellation of factors as in Lemma \ref{RedCancel} without Assumption \ref{A3}, one immediately sees that $q_n(w;r)\neq1_{r>0}\prod_{k=nt_c^{(n)}+1}^{nt_c^{(n)}+r}(w-k)^{-1}+1_{r=0}+1_{r<0}\prod_{k=nt_c^{(n)}+r+1}^{nt_c^{(n)}}(w-k)$, in general, and that $q_n(w;r)$ would also depend on the sequence $\{\beta_i^{(n)}\}_n$. In particular the limit in Theorem \ref{thm} need not exist.
\end{rem}

\subsection{Random Top Line Measure}

Up to now we have assumed that the top line configuration of \emph{yellow} particles is fixed. However, for many models it is natural not to assume that the top line configuration is fixed, but instead is a random particle process.

%Returning to the interlacing configurations of the yellow particles, the probability of finding a particular configuration of $m$ yellow (or red) particles $(X_1,Y_1),....(X_m,Y_m)$ is given by
%\begin{equation}
%\label{ProbConf}
%\mathbb{P}[(X_1,Y_1)=(x_1,y_1),...,(X_m,Y_m)=(x_m,y_m)]=\det[\mathcal{K}_{\mathcal{Y}/\mathcal{R}}^{(n)}((x_i,y_i),(x_j,y_j))]_{i,j=1}^m.
%\end{equation}
%In particular, the probability depends on the fixed top line configuration $\beta^{(n)}=(\beta_1^{(n)},...,\beta_n^{(n)})$, which furthermore is independent of the configurations %$y^{(n-1)},...,y^{(1)}$ of the lower lines. Then (\ref{ProbConf}) expresses the conditional probability of a given configuration of particles given the top line configuration $x^{(n)}$.

Let $\Sigma\subset \R$ be finite union of closed, bounded intervals and write $\Sigma_n=\Z\cap n\Sigma$. Let $\mathcal{X}^{(n)}$ denote the set 
\begin{equation*}
\mathcal{X}^{(n)}=\{\beta^{(n)}\in\Sigma_n^n\,;\,\beta_1^{(n)}<\dots <\beta_n^{(n)}\}.
\end{equation*}
We will call $\mathcal{X}^{(n)}$ the set of all admissible top line configurations. Note that $\mathcal{X}^{(n)}$ is a finite set. We will now assume that we have a probability distribution
$p^{(n)}(\beta_1^{(n)},...,\beta_n^{(n)})$ on $\mathcal{X}^{(n)}$. Extended to $\Sigma_n^n$, we assume that $p^{(n)}(\beta_1^{(n)},...,\beta_n^{(n)})$ is a symmetric function which vanishes if $\beta_i^{(n)}=\beta_j^{(n)}$  for some $i\neq j$. Let $\mathbf{P}^{(n)}$ denote the probability and $\mathbf{E}^{(n)}$ the corresponding expectation given by
$p^{(n)}$. Also, let $\mathbb{E}_{\beta^{(n)}}$ denote the expectation with respect to the red particles in the uniform interlacing with fixed top line $\beta^{(n)}$ that we studied above.

Let $\mu$ with $\text{supp\,}\mu\subseteq\Sigma$ be given and let $f$ be defined by (\ref{eqasymptoticFunction}) as previously. We assume that Assumption \ref{A2} holds with this $f$. In order to transfer the Main theorem to the case with a random top line we will use the following assumption.
\begin{Assump}
\label{A5}
For each $n\ge 1$ there is a set $\mathcal{X}_{reg}^{(n)}\subseteq \mathcal{X}^{(n)}$ of regular top line configurations such that
the following holds:
\begin{itemize}
\item[(i)] $\mathbf{P}^{(n)}[\mathcal{X}_{reg}^{(n)}]\to 1$ as $n\to\infty$,

\item[(ii)] If $\beta^{(n)}\in\mathcal{X}_{reg}^{(n)}$, $n\ge 1$ is any sequence of regular top line configurations, we define $\mu_n$ and $f_n$ as previously. Then $\mu_n$ and $f_n$ satisfy Assumptions \ref{A1}, \ref{A3} and \ref{A4} above.
\end{itemize}
\end{Assump}

We now can now give a version of Theorem \ref{thm} when we have a random top line measure.
\begin{Thm}
\label{thm2}
Consider uniform interlacing with a random top line given by $p^{(n)}$ as above, and consider the red particle point process. Let $(\xi^r_j,r)$ be the rescaled coordinates for particles on line $r$. Fix $r_1,\dots,r_M\in\Z$ and let $\phi\,:\,\R\times\{r_1,\dots,r_M\}\to[0,1]$ be a bounded measurable function with compact support. Assume that Assumption \ref{A5} holds. Then

\begin{align} 
\lim_{n\to \infty}\mathbf{E}_n\bigg[\mathbb{E}_{\beta^{(n)}}\bigg[\prod_{r\in\{r_1,\dots,r_M\}}\prod_{j}(1-\phi(\xi_j^r,r))\bigg]\bigg]=\det(I-\phi \mathcal{K}_{CA})_{L^2(\R\times
\{r_1,...,r_M\})}.
\end{align} 
\end{Thm}

\begin{proof}
We can write
\begin{align}\label{rantopline}
&\mathbf{E}_n\bigg[\mathbb{E}_{\beta^{(n)}}\bigg[\prod_{r\in\{r_1,\dots,r_M\}}\prod_{j}(1-\phi(\xi_j^r,r))\bigg]\bigg]\notag\\
&=\mathbf{E}_n\bigg[1_{\mathcal{X}_{reg}^{(n)}}\mathbb{E}_{\beta^{(n)}}\bigg[\prod_{r\in\{r_1,\dots,r_M\}}\prod_{j}(1-\phi(\xi_j^r,r))\bigg]\bigg]
+\mathbf{E}_n\bigg[1_{\mathcal{X}^{(n)}\setminus\mathcal{X}_{reg}^{(n)}}\mathbb{E}_{\beta^{(n)}}\bigg[\prod_{r\in\{r_1,\dots,r_M\}}\prod_{j}(1-\phi(\xi_j^r,r))\bigg]\bigg].
\end{align}
By Assumption \ref{A5} the second term in the right hand side of (\ref{rantopline}) goes to $0$ as $n\to\infty$. There is a $\tilde{\beta}^{(n)}$ such that 
\begin{equation*}
\max_{\beta^{(n)}\in\mathcal{X}_{reg}^{(n)}}\mathbb{E}_{\beta^{(n)}}\bigg[\prod_{r\in\{r_1,\dots,r_M\}}\prod_{j}(1-\phi(\xi_j^r,r))\bigg]
\end{equation*}
is assumed at $\beta^{(n)}=\tilde{\beta}^{(n)}$, since the maximum is over a finite set. By Assumption \ref{A5} we can apply the Main theorem to the sequence $\tilde{\beta}^{(n)}$
and thus
\begin{align*}
&\limsup_{n\to\infty}\mathbf{E}_n\bigg[1_{\mathcal{X}_{reg}^{(n)}}\mathbb{E}_{\beta^{(n)}}\bigg[\prod_{r\in\{r_1,\dots,r_M\}}\prod_{j}(1-\phi(\xi_j^r,r))\bigg]\bigg]
\le \limsup_{n\to\infty}\mathbb{E}_{\tilde{\beta}^{(n)}}\bigg[\prod_{r\in\{r_1,\dots,r_M\}}\prod_{j}(1-\phi(\xi_j^r,r))\bigg]
\\&=\det(I-\phi \mathcal{K}_{CA})_{L^2(\R\times
\{r_1,...,r_M\})}.
\end{align*}
We can do the analogous argument for the lower limes, and in this way we obtain the desired result.

\end{proof}

\begin{rem}
\label{remDOPE}
Interlacing particle systems with a random top line occur in certain types of lozenge tiling models. More precisely, we are interested in those tiling models that can be decomposed into two regions, such that after possibly adding \emph{virtual particles}, these the regions become interlacing regions of the type describe in section 1.2, glued together along a common line as depicted in figure \ref{figInterlacingDOPE}. 
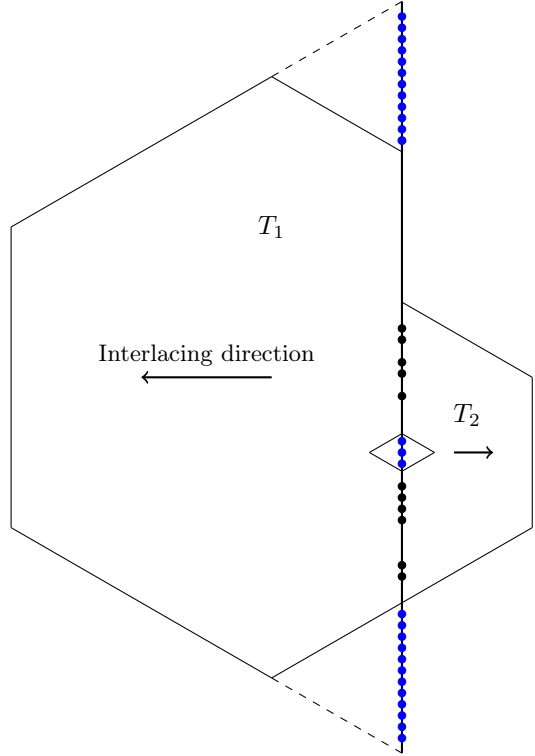
\begin{figure}[H]

\centering
\begin{tikzpicture}[xscale=1,yscale=1]
%outer polygon boundary
\draw (0,0)--(0,4);
\draw (0,4)--({2*sqrt(3)},6);
\draw ({2*sqrt(3)},6) -- ({3*sqrt(3)},5);
\draw ({3*sqrt(3)},5)-- ({3*sqrt(3)},3);
\draw ({3*sqrt(3)},3)-- ({4*sqrt(3)},2);
\draw ({4*sqrt(3)},2) -- ({4*sqrt(3)},0);
\draw ({4*sqrt(3)},0) -- ({2*sqrt(3)},-2);
\draw ({2*sqrt(3)},-2) -- (0,0);
%inner polygon boundary
\draw ({3*sqrt(3)},1.25)-- ({3.25*sqrt(3)},1);
\draw ({3.25*sqrt(3)},1)-- ({3*sqrt(3)},0.75);
\draw ({3*sqrt(3)},0.75)-- ({2.75*sqrt(3)},1);
\draw ({2.75*sqrt(3)},1) -- ({3*sqrt(3)},1.25);
% intersecting line
\draw[thick] ({3*sqrt(3)}, 7) -- ({3*sqrt(3)},-3);
% extended tangents
\draw[dashed] ({2*sqrt(3)},6) -- ({3*sqrt(3)},7);
\draw[dashed] ({2*sqrt(3)},-2) -- ({3*sqrt(3)},-3); 
% virtual particles
\filldraw[blue] ({3*sqrt(3)},1) circle (0.5mm);
\filldraw[blue] ({3*sqrt(3)},1.15) circle (0.5mm);
\filldraw[blue] ({3*sqrt(3)},0.85) circle (0.5mm);

\filldraw[blue] ({3*sqrt(3)},5.15) circle (0.5mm);
\filldraw[blue] ({3*sqrt(3)},5.15) circle (0.5mm);
\filldraw[blue] ({3*sqrt(3)},5.30) circle (0.5mm);
\filldraw[blue] ({3*sqrt(3)},5.45) circle (0.5mm);
\filldraw[blue] ({3*sqrt(3)},5.60) circle (0.5mm);
\filldraw[blue] ({3*sqrt(3)},5.75) circle (0.5mm);
\filldraw[blue] ({3*sqrt(3)},5.9) circle (0.5mm);
\filldraw[blue] ({3*sqrt(3)},6.05) circle (0.5mm);
\filldraw[blue] ({3*sqrt(3)},6.20) circle (0.5mm);
\filldraw[blue] ({3*sqrt(3)},6.35) circle (0.5mm);
\filldraw[blue] ({3*sqrt(3)},6.5) circle (0.5mm);
\filldraw[blue] ({3*sqrt(3)},6.65) circle (0.5mm);
\filldraw[blue] ({3*sqrt(3)},6.8) circle (0.5mm);

\filldraw[blue] ({3*sqrt(3)},-{1.15}) circle (0.5mm);
\filldraw[blue] ({3*sqrt(3)},-{1.30}) circle (0.5mm);
\filldraw[blue] ({3*sqrt(3)},-{1.45}) circle (0.5mm);
\filldraw[blue] ({3*sqrt(3)},-{1.60}) circle (0.5mm);
\filldraw[blue] ({3*sqrt(3)},-{1.75}) circle (0.5mm);
\filldraw[blue] ({3*sqrt(3)},-{1.9}) circle (0.5mm);
\filldraw[blue] ({3*sqrt(3)},-{2.05}) circle (0.5mm);
\filldraw[blue] ({3*sqrt(3)},-{2.20}) circle (0.5mm);
\filldraw[blue] ({3*sqrt(3)},-{2.35}) circle (0.5mm);
\filldraw[blue] ({3*sqrt(3)},-{2.5}) circle (0.5mm);
\filldraw[blue] ({3*sqrt(3)},-{2.65}) circle (0.5mm);
\filldraw[blue] ({3*sqrt(3)},-{2.8}) circle (0.5mm);

%free particles
\filldraw[black] ({3*sqrt(3)},1.75) circle (0.5mm);
\filldraw[black] ({3*sqrt(3)},2.05) circle (0.5mm);
\filldraw[black] ({3*sqrt(3)},2.2) circle (0.5mm);
\filldraw[black] ({3*sqrt(3)},2.5) circle (0.5mm);
\filldraw[black] ({3*sqrt(3)},2.65) circle (0.5mm);

\filldraw[black] ({3*sqrt(3)},0.55) circle (0.5mm);
\filldraw[black] ({3*sqrt(3)},0.4) circle (0.5mm);
\filldraw[black] ({3*sqrt(3)},0.25) circle (0.5mm);

\filldraw[black] ({3*sqrt(3)},0.1) circle (0.5mm);
\filldraw[black] ({3*sqrt(3)},-0.5) circle (0.5mm);
\filldraw[black] ({3*sqrt(3)},-0.65) circle (0.5mm);

% nodes

\draw[<-,thick] ({1*sqrt(3)},2) -- ({2*sqrt(3)},2);
\draw ({1.5*sqrt(3)},2.3) node {\small Interlacing direction};
\draw ({2*sqrt(3)},4) node { $T_1$};
\draw[->,thick] ({3.4*sqrt(3)},1) -- ({3.7*sqrt(3)},1);
\draw ({3.5*sqrt(3)},1.5) node { $T_2$};

\end{tikzpicture}
\caption{\label{figInterlacingDOPE} Decomposition of a polygon into two interlacing regions $T_1$ and $T_2$ glued together along the thick black line. The blue dots indicate the positions of \emph{virtual} particles/tiles and the black dots indicate the positions of ordinary particles/tiles.}
\end{figure}

Recall that the number of interlacing configurations with a given top line configuration $(y_1,y_2,....,y_N)\in \Z^N$ is given by Weyl's dimension formula in \cite{Weyl} for the irreducible characters of the unitary group $U(n)$,  
\begin{align}
\label{nofInterlacing}
\mathcal{N}^{\sharp}(y_1,...,y_N)&=\frac{\prod_{1\leq i<j\leq N}\vert y_i-y_j\vert}{\prod_{1\leq i<j\leq N}\vert i-j\vert}:=\frac{1}{2C_N}\prod_{\substack{1\leq i,j\leq N\\ i\neq j}}\vert y_i-y_j\vert.
\end{align}
Let $(y_1,y_2,....,y_N)$ be the positions of the particles/tiles on the intersecting thick black line as in figure \ref{figInterlacingDOPE}. Let $\mathcal{V}$ be the index set for the \emph{virtual} or \emph{frozen} particles and let $\mathcal{F}$ be the index set for the \emph{free} particles, so that $\vert \mathcal{V}\vert+\vert \mathcal{F}\vert=N$, and $y_i$ is a \emph{virtual} particle if $i\in \mathcal{V}$ and \emph{free} otherwise. Assume that $\vert \mathcal{F}\vert=n$, and let $g:\{1,....,n\}\to \mathcal{F}$ be a set bijection such that $x_i:=y_{g(i)}$, and $x_1<x_2<...<x_n$. Furthermore, the \emph{virtual} particles are densely packed, which implies that they vill form wedge shaped frozen regions. However, the fact that the two interlacing regions $T_1$ and $T_2$ need not be symmetrical implies that we need not have frozen regions on both sides of the intersecting black line. Let $\mathcal{V}_L\subseteq \mathcal{V}$  be the index set of those \emph{virtual} particles such that they form a frozen region to the left, and let $\mathcal{V}_R\subseteq \mathcal{V}$  be the index set of those \emph{virtual} particles such that they form a frozen region to the right. Then by (\ref{nofInterlacing}), the number of interlacing configuration with a given fixed configuration of free particles at positions $(x_1,...,x_n)$ is given by
\begin{align*}
&\mathcal{N}^{\sharp}(x_1,...,x_n)=\frac{1}{2C_N}\prod_{1\leq i<j\leq N}\vert y_i-y_j\vert\\
&=\frac{1}{2C_N}\prod_{\substack{i,j\in \mathcal{F}\\i\neq j}}\vert y_i-y_j\vert^2\prod_{\substack{i\in \mathcal{F}\\j\in  \mathcal{V}_L}}\vert y_i-y_j\vert\prod_{\substack{i\in \mathcal{F}\\ j\in \mathcal{V}_R}}\vert y_i-y_j\vert\prod_{\substack{i,j\in\mathcal{V}_L\\i\neq j}}\vert y_i-y_j\vert\prod_{\substack{i,j\in\mathcal{V}_R\\i\neq j}}\vert y_i-y_j\vert\\
&=\frac{1}{2C_N}\prod_{\substack{1\leq k,l\leq n\\k\neq l}}\vert x_k-x_l\vert^2\prod_{k=1}^n\prod_{j\in \mathcal{V}_L}\vert x_k-y_j\vert\prod_{k=1}^n\prod_{j\in \mathcal{V}_R}\vert x_k-y_j\vert\prod_{\substack{i,j\in\mathcal{V}_L\\i\neq j}}\vert y_i-y_j\vert\prod_{\substack{i,j\in\mathcal{V}_R\\i\neq j}}\vert y_i-y_j\vert.
\end{align*}
Now, consider the set of all possible lozenge tessellations of the the original polygon. One easily sees that each such tessellation is in a bijective correspondence with two interlacing configurations on $T_1$ and $T_2$ with the same configuration of \emph{free} particles $x^{(n)}=(x_1,x_2,...,x_n)$ on their common top line. In particular, 
$x_i\in \Sigma_n$ for each $i=1,...,n$, where $\Sigma$ is a finite union of intervals. The set $\mathcal{X}_n$ then denotes the set of all configurations of \emph{free} particles. Consider the set of all lozenge tessellations of the polygon with uniform distribution. Then this induces a probability distribution on $\mathcal{X}_n$, given by
\begin{align*}
p^{(n)}(x_1,....,x_n)&=\mathbb{P}[\text{particles at positions $x_1,x_2,...,x_n$}]=\frac{\mathcal{N}^{\sharp}(x_1,...,x_n)}{\displaystyle\sum_{(x_1,...,x_n)\in \mathcal{X}_n}\mathcal{N}^{\sharp}(x_1,...,x_n)}.
\end{align*}
Let 
\begin{align*}
w_n(x):=\prod_{j\in \mathcal{V}_L}\vert x-y_j\vert\prod_{j\in \mathcal{V}_R}\vert x-y_j\vert.
\end{align*}
Then
\begin{align}
\label{DOPEdist}
p^{(n)}(x_1,....,x_n)=\frac{1}{Z_n}\prod_{1\leq i<j\leq n}(x_i-x_j)^2\prod_{i=1}^{n}w_n(x_i)=\frac{1}{Z_n}\Delta_n(x)^2\prod_{i=0}^nw_n(x_i),
\end{align}
where $Z_n$ is a normalization constant, and $\Delta_n(x)$ is the Vandermonde determinnat. Associated with a particular weight function $w_n(x)$ is a class of discrete orthogonal polynomials $\{p_{n,k}(x)\}_k$ satisfying\begin{align}
\label{DOPEpoly}
\sum_{x\in \Sigma_n}p_{n,k}(x)p_{n,l}(x)w_n(x)=\delta_{kl}.
\end{align}
Such particle processes are called \emph{discrete orthogonal polynomial ensembles}, DOPE, and have been studied e.g. in \cite{J.Baik}. In particular if one considers the random empirical measure $\mu_n=\frac{1}{n}\sum_{i=1}^n\delta_{x_i}$, then $\mu_n\rightharpoonup \mu_V^{\lambda}$, where the measure $\mu_V^{\lambda}\in \mathcal{M}_1^{\lambda}(\Sigma)$ is the unique solution of the constrained variational problem 

\begin{align}
\label{VarProb}
\min_{\nu\in \mathcal{M}_1^{\lambda}(\Sigma)}\{I_V[\nu]\}=\min_{\nu\in \mathcal{M}_1^{\lambda}(\Sigma)}\bigg\{\int_{\Sigma\times \Sigma}\log\vert x-y\vert^{-1}d\nu(x)d\nu(y)+\int_{\Sigma}V(x)d\nu(x)\bigg\},
\end{align}
where $V(x)=\lim_{n\to \infty}-n^{-1}\log(w_n(x))$. In the case of those problems originating from a random tiling model as above, the potential will be of the form
\begin{align}
\label{TilingPotential}
V(x)=U^{\chi_{I^r}}(x)+U^{\chi_{I^l}}(x),
\end{align}
where $I^r$ and $I^l$ are finite unions of closed intervals of $\R$, such that $(I^r\cup I^l)^{\circ}\cap\Sigma^{\circ}=\varnothing$. In particular $\lim_{n\to \infty}-n^{-1}\log(w_n(x))=V(x)$ uniformly on compact subsets of $\Sigma$ not containing the subset $\dv \Sigma\cap\dv (I^r\cup I^l)$. It follows from Theorem 2.1 in \cite{Dragnev} that the  minimizer $\mu_V^{\lambda}$ of the variational problem (\ref{VarProb}) has a unique the characterization in terms of the following variational inequalities.

There exists a constant $F_V^{\lambda}$ such that 
\begin{align}
\label{VarIneq1}
&2U^{\mu_V^{\lambda}}(x)+V(x)\geq F_V^{\lambda}\quad \text{for all $x\in\Sigma\backslash$supp$(\mu_V^{\lambda}):=I_V$}\\
\label{VarIneq2}
&2U^{\mu_V^{\lambda}}(x)+V(x)= F_V^{\lambda}\quad \text{for all $x\in$supp$(\mu_V^{\lambda})\cap$supp$(\lambda-\mu_V^{\lambda}):=I_B$}\\
\label{VarIneq3}
&2U^{\mu_V^{\lambda}}(x)+V(x)\leq F_V^{\lambda}\quad \text{for all $x\in\Sigma\backslash$(supp$(\mu_V^{\lambda})\cap$supp$(\lambda-\mu_V^{\lambda})):=I_S$}.
\end{align}
It follows from general \emph{large deviation estimates} for DOPE that we can define $\mathcal{X}_{reg}^{(n)}$ so that Assumption \ref{A3} is satisfied for large classes of potentials, see \cite{Feral}. In particular, in the class of potentials of the form (\ref{TilingPotential}) and associated weights $w_n$ coming from certain tiling models, this is proved in the upcoming review article \cite{Duse15d}. 

It is shown in \cite{Duse15d} that if $R_1\cup R_2\neq \varnothing$, with $\mu=\mu_V^{\lambda}$, then for $t_c\in R_1\cup R_2$, we must have $t_c\in \dv\Sigma$. Thus, if in particular $t_c\in R_2$, then there exists an interval $[t_2,t_c]$, such that $[t_2,t_c]\cap \text{supp}(\mu_n)=\varnothing$. This implies that the sequence of empirical measures $\{\mu_n\}_n$ automatically satisfies the left-sided part of regularity assumption in Assumption \ref{A3}.
Let 
\begin{align}
R_m^{(n)}(x_1,...,x_m)=\mathbb{P}[\text{there are particles at each of the nodes $x_1,...,x_m$}]
\end{align}
be the $m$:th correlation kernel with $1\leq m\leq n$. If the potential $V(x)$ is analytic in a complex convex neighbourhood of $\Sigma$, then it is proven in Theorem 3.3 and Theorem 3.5 in \cite{J.Baik} that for all subsets $V\in I_V$ and $S\in I_S$ such that $d_H(V,I_B)>0$ and $d_H(S,I_B)>0$ there exists constants $K=K(d_H(V,I_B))>0$ and $L=L(d_H(S,I_B))>0$, such that 
\begin{align}
\max_{x_1,...,x_m\in V}\vert R_m^{(n)}(x_1,...,x_m)\vert \leq C_V\frac{e^{-mnK_V}}{n^m}
\end{align}
and 
\begin{align}
\max_{x_1,...,x_m\in S}\vert 1-R_m^{(n)}(x_1,...,x_m)\vert \leq C_S\frac{e^{-mnL_S}}{n^m}
\end{align}
for some positive constants $C_V$ and $C_S$. In particular this implies that we can define $\mathcal{X}_{reg}^{(n)}$ so that Assumption \ref{A3} holds. Unfortunately, the class of potentials given by (\ref{TilingPotential}) need not be analytic in a neighborhood of $\Sigma$ due to the fact that we may have $\dv \Sigma\cap\dv (I^r\cup I^l)\neq 0$. However, the we believe that by additional local arguments around such points, one may generalize the methods used in the book \cite{J.Baik} to prove that Theorem 3.3 and Theorem 3.5 in \cite{J.Baik} also hold in this case. In particular, the effect of non-analyticity should only matter on a very small neighborhood around the points of $\dv \Sigma\cap\dv (I^r\cup I^l)$. Since $t_c\notin \dv \Sigma\cap\dv (I^r\cup I^l)$, and therefore is a macroscopic distance away from $\dv \Sigma\cap\dv (I^r\cup I^l)$, the effects of non-analyticity should not matter. 

Finally, let $A\subset \R$ be an interval and let $K$ be a compact subset of $\{z\in \C: d_H(z,A\cap \text{supp}(\mu_n))>0\}$. Then, for a ``generic'' potential $V$, and a sequence of weight functions $\{w_n(x)\}_n$, such that \newline
$\lim_{n\to \infty}-n^{-1}\log(w_n(x))=V(x)$ uniformly on compact subsets of $\Sigma$, we should have for every $\eps>0$ and every $z\in K$ that
\begin{align}
\lim_{n\to\infty}\mathbb{P}\bigg[\bigg\vert\int_{\R}\frac{\chi_A(t)d(\mu_n-\mu)(t)}{z-t}\bigg\vert>\frac{1}{n^{1-\eps}}\bigg]=0.
\end{align}
In particular, this implies that we can define $\mathcal{X}_{reg}^{(n)}$ so that Assumption \ref{A4} holds. This result should also follow from Riemann-Hilbert methods adapted to the class of potentials derived from tiling models, with the caveat that some of the local arguments need to be modified due to the possible non-analyticity of $V(x)$ at some of the points of $\dv \Sigma$. This questions has also been studied in \cite{Borodin} for very similar models by means of \emph{discrete loop equations}, though the assumptions in \cite{Borodin} are not satisfied in our models, and can therefore not be applied directly. However, in the special case when $\Sigma$ is an interval, the results in \cite{Borodin} do apply. In particular they apply to the model in Example \ref{CardiodEx}.
\end{rem}
\begin{rem}\label{remEx}
Let $z^{(r)}=(y_1^{(r)},...,y_r^{(r)},x_n+n-r-1,x_n+n-r-2,...,x_n)$. 
%and let $Z(y^{(r)},y^{(r+1)})$ be the matrix defined through $Z(y^{(r)},y^{(r+1)})_{ij}=1_{z_j^{(r+1)}\geq z_i^{(r)}}$. 
Then, using (\ref{DOPEdist}) the total probability distribution is given by
\begin{align*}
\nu_{tot}[(y^{1},...,y^{(n-1)},x^{(n)})]=\frac{1}{Z_{n,tot}}\Delta_n(x)^2\prod_{i=0}^nw_n(x_i)\prod_{r=0}^{n-1}
\det\left(1_{z_j^{(r+1)}\geq z_i^{(r)}}\right).
\end{align*}
Hence, by the Eynard-Mehta theorem, the total process is also a determinantal process. One could therefore try to derive the correlation kernel for the total process instead of the conditional process. However, the Eynard-Mehta formula for the kernel of this process
seems very difficult to analyze.
 In particular, we can not expect to get such a simple formula for the kernel as (\ref{Red}), as it would necessarily need to contain all the 
information about the DOPE, which is highly non-trivial.
\end{rem}

\begin{ex}
\label{CardiodEx}
Consider now the example discussed briefly in section \ref{sec:intro} see fig~\ref{figCardioid1}.
This can be approached via the tiling model illustrated in figure \ref{figCardioid2}. Elementary geometry gives $\delta=\frac{2}{\sqrt{3}}-2\kappa$, where $\kappa\in(0,1/\sqrt{3})$. In the figure we have added densely packed virtual particles in the intervals $[0,\kappa]$ and $[2\kappa+2\delta,3\kappa+2\delta]$. Moreover, the particles contained in the interval $[\kappa+\delta,2\kappa+2\delta]$ are distributed according to a \emph{discrete orthogonal polynomial ensemble}, DOPE, as is shown above. The fraction of particles contained in this interval as $n\to \infty$ equals $1-\sqrt{3}\kappa$. We will now make the symmetric parameter choice $\kappa=\delta=\frac{2}{3\sqrt{3}}$. Associated to the DOPE, there is an equilibrium measure with respect to an external field as above. See also Proposition 2.2 in \cite{Borodin}. Solving the minimization problem gives the density of the measure $\mu$,
\begin{align}
\rho(t)=\chi_{[0,\frac{2}{3\sqrt{3}}]}(t)+\chi_{[\frac{8}{3\sqrt{3}},\frac{10}{3\sqrt{3}}]}(t)+\phi(3t)\chi_{[\frac{4}{3\sqrt{3}},\frac{8}{3\sqrt{3}}]}(t),
\end{align}
where 
\begin{align}
\label{exEqDensity}
\phi(x)&=-\frac{\sqrt{(b-x)(x-a)}}{2\pi}\int_{0}^{2/\sqrt{3}}\frac{dt}{\sqrt{(t-a)(t-b)}(t-x)}-\frac{\sqrt{(b-x)(x-a)}}{\pi}\int_{4/\sqrt{3}}^{a}\frac{dt}{\sqrt{(t-a)(t-b)}(t-x)}\nonumber\\&+\frac{\sqrt{(b-x)(x-a)}}{2\pi}\int_{8/\sqrt{3}}^{10/\sqrt{3}}\frac{dt}{\sqrt{(t-a)(t-b)}}.
\end{align}
Here $(a,b)$ is the unique solution of 
\begin{align}
\label{exEqSupp1}
&-\frac{1}{2}\int_{0}^{2/\sqrt{3}}\frac{dt}{\sqrt{(t-a)(t-b)}}-\int_{4/\sqrt{3}}^{a}\frac{dt}{\sqrt{(t-a)(t-b)}}+\frac{1}{2}\int_{8/\sqrt{3}}^{10/\sqrt{3}}\frac{dt}{\sqrt{(t-a)(t-b)}}=0
\end{align}
and 
\begin{align}
\label{exEqSupp2}
&-\frac{1}{2}\int_{0}^{2/\sqrt{3}}\frac{tdt}{\sqrt{(t-a)(t-b)}}-\int_{4/\sqrt{3}}^{a}\frac{tdt}{\sqrt{(t-a)(t-b)}}+\frac{1}{2}\int_{8/\sqrt{3}}^{10/\sqrt{3}}\frac{tdt}{\sqrt{(t-a)(t-b)}}=1
\end{align}
satisfying $\frac{4}{\sqrt{3}}<a<b<\frac{8}{\sqrt{3}}$.
In particular it follows that $\rho\big\vert_{[\frac{4}{3\sqrt{3}},a]}\equiv1$. A direct verification that the Cauchy transform of $\rho$ satisfies Assumption \ref{A2} for $t_c=\frac{4}{3\sqrt{3}}$ is very difficult using (\ref{exEqDensity})-(\ref{exEqSupp2}) even in the case when the original polygon has an apparent symmetry. This is due to the artificial decomposition of the polygon into two interlacing regions which are glued together along a common boundary. After this decomposition has been done, the original symmetry is no longer apparent in the parametrization of the \emph{edge} $\mathcal{E}$. We will therefore prove that the asymptotic cusp condition holds by an indirect symmetry argument. Instead of considering figure \ref{figCardioid2}, we consider figure \ref{figCardioid3}. We see the yellow dashed line corresponds to the decomposition in figure \ref{figCardioid2} into two interlacing regions glued together along the yellow dashed line. However, we may equally well decompose our hexagon into two interlacing regions in blue particles glued together along the dashed blue line instead. Moreover there is bijective correspondence between each configuration of blue and yellow particles given by a reflection in the dashed black symmetry line. This implies in particular that the \emph{edge} $\EE$ must posses a reflection symmetry in the dashed black line. Now, the parametrization of the \emph{edge} $\mathcal{E}$, given by the density (\ref{exEqDensity}) and (\ref{Chi0})-(\ref{Eta1}) is smooth by Remark 2.1 in \cite{Duse14a}, in fact real analytic. Since the density $\rho$ of the limit measure $\mu$ satisfies $\rho\big\vert_{[\frac{4}{3\sqrt{3}},a]}=1$, it follows from Lemma 2.7 in \cite{Duse14a}, that $\mathcal{E}$, is tangent to the line $\chi+\eta-1=\frac{4}{3\sqrt{3}}$. However, by Theorem 2.3 in \cite{Duse14a}, the parametrization is injective. This together with the fact that the only singularities of the \emph{edge} $\EE$ are cusps implies the \emph{edge} $\mathcal{E}$ has a cusp at the tangent point with the line $\chi+\eta-1=\frac{4}{3\sqrt{3}}$. Lemma 2.8 and Lemma 2.9 in \cite{Duse14a}, implies that $f''(\frac{4}{3\sqrt{3}};\chi_\EE(\frac{4}{3\sqrt{3}}), \eta_\EE(\frac{4}{3\sqrt{3}}))=0$.
Hence, modulo the technical issues discussed above about how to construct $\mathcal{X}_{reg}^{(n)}$, we get a cusp-point where the scaling limit is given by the Cusp-Airy process.
\begin{figure}[H]
\centering
\begin{tikzpicture}[xscale=3/4,yscale=3/4]

%lozenge tiles
\draw (-1.2,7.35) \lozy;
\draw (-0.7,7.35) \lozy;
\draw (-0.2,7.35) \lozy;
\draw (0.6,7.35) \lozy;
\draw (1.6,7.35) \lozy;
\draw (3.1,7.35) \lozy;
\draw (4.74,7.35) \lozy;
\draw ({4.74+0.5},7.35) \lozy;
\draw ({4.74+1},7.35) \lozy;
\draw ({4.74+1.5},7.35) \lozy;
\draw ({4.74+2},7.35) \lozy;
\draw ({4.74+2.5},7.35) \lozy;
\draw (-7.28,7.35) \lozy;
\draw ({-7.28+0.5},7.35) \lozy;
\draw ({-7.28+1},7.35) \lozy;
\draw ({-7.28+1.5},7.35) \lozy;
\draw ({-7.28+2},7.35) \lozy;
\draw ({-7.28+2.5},7.35) \lozy;
\filldraw[gray] ({4.74},7.794) circle (3pt);
\filldraw[gray] ({4.74+0.5},7.794) circle (3pt);
\filldraw[gray] ({4.74+1},7.794) circle (3pt);
\filldraw[gray] ({4.74+1.5},7.794) circle (3pt);
\filldraw[gray] ({4.74+2},7.794) circle (3pt);
\filldraw[gray] ({4.74+2.5},7.794) circle (3pt);
\filldraw[gray] ({-7.28},7.794) circle (3pt);
\filldraw[gray] ({-7.28+0.5},7.794) circle (3pt);
\filldraw[gray] ({-7.28+1},7.794) circle (3pt);
\filldraw[gray] ({-7.28+1.5},7.794) circle (3pt);
\filldraw[gray] ({-7.28+2},7.794) circle (3pt);
\filldraw[gray] ({-7.28+2.5},7.794) circle (3pt);
\draw (-1.2,7.794) circle (3pt);
\draw (-0.7,7.794) circle (3pt);
\draw (-0.2,7.794) circle (3pt);
\draw (0.6,7.794) circle (3pt);
\draw (1.6,7.794) circle (3pt);
\draw (3.1,7.794) circle (3pt);
%hexagon
\draw[thick] (-3,0)--(3,0)--(6,5.196)--(3,10.392)--(0,10.392)--(-1.5,7.794)--(-4.5,7.794)--(-6,5.196)--(-3,0);
\draw[thick,dashed] (-6,5.196)--(-7.5,7.794);
\draw[thick,dashed] (-7.5,7.794)--(-4.5,7.794);
\draw[thick,dashed] (6,5.196)--(7.5,7.794);
\draw[thick,dashed](-1.5,7.794)--(7.5,7.794);
\draw[thick](8,7.794) --(8.5,7.794);
\draw[thick](8,0) --(8.5,0);
\draw (-2.7,0) arc (0:120:3mm);
\draw (-2.6,0.5) node {$\frac{2\pi}{3}$};
\draw[<->] (8.25,0)--(8.25,7.794);
\draw (9.5,3.8) node {1 ($n$ lines)};
\draw (-7.2,7.794) arc (0:-60:3mm);
\draw (-7,7.5) node {$\frac{\pi}{3}$};
\draw (-4.9,6.65) node {$\kappa$};
\draw[thick,dashed] (-4.5,7.794) -- (-3,10.392)--(0,10.392);
\draw[thick,dashed,blue] (-3,10.392)--(3,0);
\draw (-3,7.5) node {$\delta$};
\draw (2,8.4) node {$\delta+\kappa$};

\end{tikzpicture}
\caption{\label{figCardioid2}}
\end{figure}

\begin{figure}[H]
\centering
\begin{tikzpicture}[xscale=0.5,yscale=0.5]

%Hexagon
\draw[yellow,dashed,thick] (-{5*sqrt(3)},7) --(-{3*sqrt(3)},5);
\draw[yellow,dashed,thick] (-{3*sqrt(3)},5) -- ({sqrt(3)},1);
\draw[blue,dashed,thick] (-{2*sqrt(3)},4)  --(-{5*sqrt(3)},1);
\draw[blue,dashed,thick] (-{2*sqrt(3)},4)  --({sqrt(3)},7);
\draw[thick] (0,0) -- (0,4) --(-{sqrt(3)},5) -- (-{2*sqrt(3)},4) -- (-{3*sqrt(3)},5) -- (-{4*sqrt(3)},4) -- (-{4*sqrt(3)},0) -- (-{2*sqrt(3)},-2) --(0,0);
\draw[thick,dashed] (-1.73,5)--(-3.46,6) -- (-5.20,5);
\draw[thick,dashed] (-3.46,-3)--(-3.46,7);

\end{tikzpicture}
\caption{\label{figCardioid3}}
\end{figure}

\end{ex}

\subsection{Outline of the Paper}
We now give a brief outline of the paper. In section \ref{sec:cancellation} we prove Lemma \ref{RedCancel}. Due to the fact that $f'''(t_c,\chi_c,\eta_c)>0$, it will be necessary to change the integration contours of the correlation kernel given in Proposition \ref{RedCancel} to be able to deform the contours to the local steepest ascent/descent contours. This is done section \ref{sec:cancellation}.  In section \ref{sec:globalcontours} we prove the existence of global ascent/descent contours and finally in sections \ref{sec:estimates} and \ref{sec:local} we perform the local asymptotic analysis to arrive at the final result

In section \ref{sec:symmetry} we prove a certain reflection symmetry of the Cusp-Airy kernel in the axis $r=0$. In section \ref{sec:rAiry} we proceed to prove an alternative representation of the the Cusp-Airy kernel in terms of r-Airy integrals and certain polynomials. In section \ref{sec:kernel} we derive the integral representation of the kernel for the interlacing particle system, or the yellow particles and prove Proposition \ref{Transform}.

\section{Proof of Main Result}

\subsection{Discrete Cancellation in the Correlation Kernel}\label{sec:cancellation}
In this section we will perform the discrete cancellation of factors in the correlation kernel (\ref{Red}). In the continuum limit, this corresponds to the cancellation between the measures $\mu$ and $\lambda$ on the interval $[t_c, t_1]$, where $\mu\big\vert_{[t_c,t_1]}=\lambda$. The difference is that on the discrete level, the cancellation need no longer be exact. 

We now prove Lemma \ref{RedCancel}
\begin{proof}
Using Assumption \ref{A3} and the definition of $t_c^{(n)},t_1^{(n)}$ and $t_2^{(n)}$ we have
\begin{align*} 
&\frac{\prod_{i=1}^n(w-\beta_i^{(n)})}{\prod_{\substack{k=x_1+y_1-n}}^{x_1}(w-k)}=\frac{\prod_{\beta_i^{(n)}\leq nt_2^{(n)}}(w-\beta_i^{(n)})\prod_{k=nt_c^{(n)}}^{nt_1^{(n)}}(w-k)\prod_{nt_1^{(n)}<\beta_i^{(n)}}(w-\beta_i^{(n)})}{\prod_{\substack{k=x_1+y_1-n}}^{x_1}(w-k)}\\
&=\frac{\prod_{k=nt_c^{(n)}}^{nx_c^{(n)}}(w-k)}{\prod_{\substack{k=x_1+y_1-n}}^{x_1}(w-k)}\frac{\prod_{\beta_i^{(n)}\leq nt_2^{(n)}}(w-\beta_i^{(n)})\prod_{nt_1^{(n)}<\beta_i^{(n)}}(w-\beta_i^{(n)})}{\prod_{k=nt_1^{(n)}+1}^{nx_c^{(n)}}(w-k)}\\
&=\frac{\prod_{k=nt_c^{(n)}}^{x_1}(w-k)}{\prod_{\substack{k=x_1+y_1-n}}^{x_1}(w-k)}\frac{\prod_{k=nt_c^{(n)}}^{nx_c^{(n)}}(w-k)}{\prod_{\substack{k=nt_c^{(n)}}}^{x_1}(w-k)}\frac{\prod_{\beta_i^{(n)}\leq nt_2^{(n)}}(w-\beta_i^{(n)})\prod_{nt_1^{(n)}<\beta_i^{(n)}}(w-\beta_i^{(n)})}{\prod_{k=nt_1^{(n)}+1}^{nx_c^{(n)}}(w-k)}\\
&=q_n(w;r)Q_n(w;\Delta x_1^{(n)})E_n(w),
\end{align*}
where we have used the definitions of $q_n,Q_n$ and $E_n$. Similarly, we get
\begin{align*} 
&\frac{\prod_{i=1}^n(z-\beta_i^{(n)})}{\prod_{\substack{k=x_2+y_2-n}}^{x_2-1}(z-k)}=(z-x_2)q_n(z;r)Q_n(z;\Delta x_1^{(n)})E_n(z).
\end{align*}

Finally we rescale the integration variables according to $w\to \frac{1}{n}w$ and $z\to \frac{1}{n}z$. Set $\G_n=\frac{1}{n}\mathscr{Z}_n$, $\tilde{\G}_n=\frac{1}{n}\tilde{\mathscr{Z}}_n$, and $\g_n=\tilde{\g}_n=\frac{1}{n}\mathscr{W}_n$. 
Since we have cancelled out poles we may deform $\G_n$ to be a contour that contains the set $\{n^{-1}\beta_i^{(n)}\geq t_1^{(n)}\}$ but not the set $\{n^{-1}\beta_i^{(n)}\leq t_c^{(n)}+s\}$. Similarly, we may deform $\G'_n$ to contain the set $\{n^{-1}\beta_i^{(n)}\leq t_c^{(n)}+s\}$ but not the set $\{n^{-1}\beta_i^{(n)}\geq n^{-1}t_1^{(n)}\}$. 
\end{proof}

\subsection{Change of Integration Contours}\label{sec:contours}
In order to perform a steepest descent analysis in a later section, it will be necessary to change the integration contours so that they may be suitably deformed around the critical point.

\begin{Prop}
\label{IntKernel3}
The correlation kernel can be rewritten as
\begin{align} 
\label{ChangedContourKernel1}
K_{\mathcal{R}}^{(n)}((\xi_n,r)(\mu_n,s))&=-1_{x_1\geq x_2}B_n((x_1,y_1),(x_2,y_2))+\tilde{K}_{\mathcal{R}}^{(n)}((x_1,y_1),(x_2,y_2)),
\end{align}
where 
\begin{align}
\label{ChangedContourKernel2}
\tilde{K}_{\mathcal{R}}^{(n)}((x_1,y_1),(x_2,y_2))=\frac{1}{(2\pi i)^2}\oint_{\G_n^1+\G_n^2}dz \oint_{\g_n^1+\g_n^2}dw\frac{1}{z-n^{-1}x_2}\frac{q(nw;r)}{q(nz;s)}\frac{e^{n(f_n(w)-f_n(z))+n(h_{n,1}(w)-h_{n,2}(z)})}{w-z}
\end{align}
and
\begin{align}
\label{ChangedContourKernel3}
B_n((x_1,y_1),(x_2,y_2))=\frac{1}{2\pi i}\oint_{\Gamma_n^2}dz \frac{1}{z-n^{-1}x_2}\frac{q(nz;r)}{q(nz;s)}e^{n(h_{n,1}(z)-h_{n,2}(z))}.
\end{align}
Here the contours are as in figure \ref {figContour1}; more precisely:
\begin{itemize}
\item
$\G_n^1$ is a counter-clockwise oriented contour that contains the interval $[t_c^{(n)}-\max\{\vert r\vert, \vert s\vert\},b_n]$ and nothing else of the support of $\nu_n$. Furthermore, $\G_n^1$ contains the contours $\G_n^2$, $\g_n^2$ and $\g_n^3$\\
\item
$\G_n^2$ is a clockwise oriented contour that contains  $\g_n^2$ and the interval $[t_c^{(n)}-\max\{\vert r\vert, \vert s\vert\},t_c^{(n)}+\max\{\vert r\vert, \vert s\vert\}]$ and nothing else of the support of $\nu_n$.\\
\item
$\g_n^1$ is a clockwise oriented contour that contains the interval $[t_1^{(n)},b_n]$ and nothing else of the support of $\nu_n$.\\
\item
$\g_n^2$ is a clockwise oriented contour that contains the interval $[t_c^{(n)}-\max\{\vert r\vert, \vert s\vert\},t_c^{(n)}+\max\{\vert r\vert, \vert s\vert\}]$ and nothing else of the support of $\nu_n$.\\
\end{itemize}

\end{Prop}
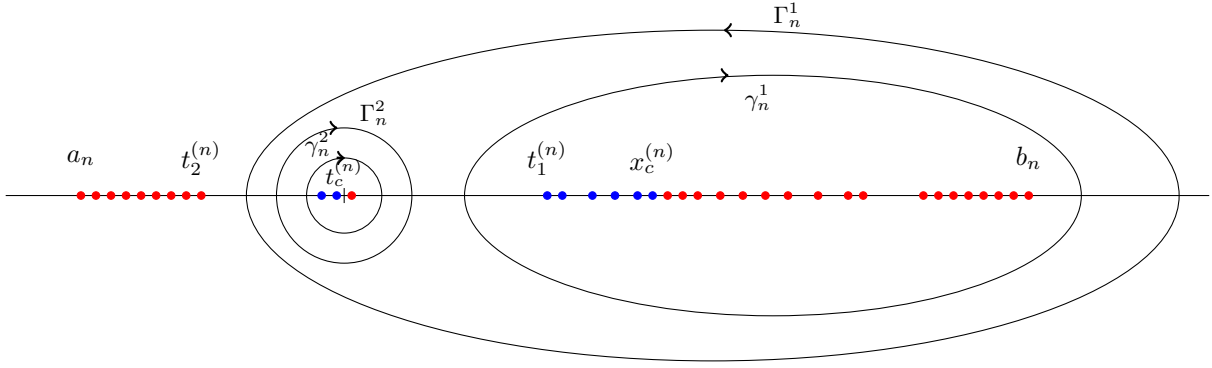
\begin{figure}[H]
\centering
\begin{tikzpicture}

\draw (0.4,0) ellipse (6.2cm and 2.2cm);
\draw (1.2,0) ellipse (4.1cm and 1.6cm);
\draw(-9,0) --++(16,0);
\draw (-4.5,0.0) circle (0.9cm);
\draw (-4.5,0.0) circle (0.5cm);

\draw[arrows=->,line width=1pt](0.56,2.2)--(0.55,2.2);
\draw (1.4,2.4) node {\small$\G_n^1$};
\draw[arrows=->,line width=1pt](0.6,1.6)--(0.61,1.6);
\draw (1.0,1.3) node {\small$\g_n^1$};
\draw[arrows=->,line width=1pt](-4.6,0.9)--(-4.59,0.9);
\draw (-4.85,0.68) node {\small$\g_n^2$};
\draw[arrows=<-,line width=1pt](-4.49,0.5)--(-4.6,0.5);
\draw (-4.1,1.1) node {\small$\G_n^2$};

\draw (-8,0.5) node {$a_n$};
\filldraw[fill=red, draw=red] (-8,0) circle (0.05 cm);
\filldraw[fill=red, draw=red] (-7.8,0) circle (0.05 cm);
\filldraw[fill=red, draw=red] (-7.6,0) circle (0.05 cm);
\filldraw[fill=red, draw=red] (-7.4,0) circle (0.05 cm);
\filldraw[fill=red, draw=red] (-7.2,0) circle (0.05 cm);
\filldraw[fill=red, draw=red] (-7.0,0) circle (0.05 cm);
\filldraw[fill=red, draw=red] (-6.8,0) circle (0.05 cm);
\filldraw[fill=red, draw=red] (-6.6,0) circle (0.05 cm);
\filldraw[fill=red, draw=red] (-6.4,0) circle (0.05 cm);
\draw (-6.4,0.5) node {$t_2^{(n)}$};
\draw (-4.5,0.3) node {\small$t_c^{(n)}$};
\draw(-4.5,-.1) --++(0,.2);
\filldraw[fill=red, draw=red] (-4.4,0) circle (0.05 cm);
\filldraw[fill=blue, draw=blue] (-4.6,0) circle (0.05 cm);
\filldraw[fill=blue, draw=blue] (-4.8,0) circle (0.05 cm);
\draw (-1.8,0.5) node {$t_1^{(n)}$};
\filldraw[fill=blue, draw=blue] (-1.8,0) circle (0.05 cm);
\filldraw[fill=blue, draw=blue] (-1.6,0) circle (0.05 cm);
\filldraw[fill=blue, draw=blue] (-1.2,0) circle (0.05 cm);
\filldraw[fill=blue, draw=blue] (-0.9,0) circle (0.05 cm);
\filldraw[fill=blue, draw=blue] (-0.6,0) circle (0.05 cm);
\filldraw[fill=blue, draw=blue] (-0.4,0) circle (0.05 cm);
\draw (-0.4,0.5) node {$x_c^{(n)}$};
\filldraw[fill=red, draw=red] (-0.2,0) circle (0.05 cm);
\filldraw[fill=red, draw=red] (0,0) circle (0.05 cm);
\filldraw[fill=red, draw=red] (0.2,0) circle (0.05 cm);
\filldraw[fill=red, draw=red] (0.5,0) circle (0.05 cm);
\filldraw[fill=red, draw=red] (0.8,0) circle (0.05 cm);
\filldraw[fill=red, draw=red] (1.1,0) circle (0.05 cm);
\filldraw[fill=red, draw=red] (1.4,0) circle (0.05 cm);
\filldraw[fill=red, draw=red] (1.8,0) circle (0.05 cm);
\filldraw[fill=red, draw=red] (2.2,0) circle (0.05 cm);
\filldraw[fill=red, draw=red] (2.4,0) circle (0.05 cm);

\filldraw[fill=red, draw=red] (3.2,0) circle (0.05 cm);
\filldraw[fill=red, draw=red] (3.4,0) circle (0.05 cm);
\filldraw[fill=red, draw=red] (3.6,0) circle (0.05 cm);
\filldraw[fill=red, draw=red] (3.8,0) circle (0.05 cm);
\filldraw[fill=red, draw=red] (4,0) circle (0.05 cm);
\filldraw[fill=red, draw=red] (4.2,0) circle (0.05 cm);
\filldraw[fill=red, draw=red] (4.4,0) circle (0.05 cm);
\filldraw[fill=red, draw=red] (4.6,0) circle (0.05 cm);
\draw (4.6,0.5) node {$b_n$};
\end{tikzpicture}
\caption{\label{figContour1} Integration contours}
\end{figure}

\begin{proof}
The starting point is lemma \ref{RedCancel}. The proof will consist of a series of deformations of the contours. It will be convenient to introduce the notation
\begin{align*}
J_{\G,\g}^{(n)}:=\frac{1}{(2\pi i)^2}\oint_{\G}dz \oint_{\g}dw\frac{1}{z-n^{-1}x_2}\frac{q_n(nw;r)}{q_n(nz;s)}\frac{Q_n(nw;\Delta x_1^{(n)})}{Q_n(nz;\Delta x_2^{(n)})}\frac{E_n(nw)E_n(nz)^{-1}}{w-z}
\end{align*}
for some contours $\G,\g$.
First assume that $x_1<x_2$ and consider the deformation of $\gamma_n$ in figure \ref{figCancelledContour1} into two contours $\g_n^2$ and $\g_n^0$ as shown in figure \ref{figDeform1}. Next consider adding the contour $\gamma_n^1$. This is shown in figure \ref{figDeform2}. Consider the contour $\g_n^0+\g_n^1$. Then the integrand of $K_{\mathcal{R}}^{(n)}$ in (\ref{lemCancelKernel}) has one residue inside the domain bounded by $\g_n^1$ and $\g_n^0$. Computing this residue gives
\begin{align*}
&\frac{1}{(2\pi i)^2}\oint_{\Gamma_n}dz \oint_{\g_n^0\cup \g_n^1}dw\frac{1}{z-n^{-1}x_2}\frac{q_n(nw;r)}{q_n(nz;s)}\frac{Q_n(nw;\Delta x_1^{(n)})}{Q_n(nz;\Delta x_2^{(n)})}\frac{E_n(nw)E_n(nz)^{-1}}{w-z}\\
&=\frac{1}{2\pi i}\oint_{\Gamma_n}\frac{\prod_{k=x_2+y_2-n}^{x_2-1}(nz-k)}{\prod_{\substack{k=x_1+y_1-n}}^{x_1}(nz-k)}dz
=\frac{1}{2\pi i}\oint_{\Gamma_n}\frac{\prod_{k=nt_c^{(n)}+s}^{x_2-1}(nz-k)}{\prod_{\substack{k=nt_c^{(n)}+r}}^{x_1}(nz-k)}dz\\
&=\frac{1}{2\pi i}\oint_{\Gamma_n}\frac{q_n(nz;r)}{q_n(nz;s)}\prod_{k=x_1+1}^{x_2-1}(nz-k)dz=0,
\end{align*}
since $\G_n$ contains no poles. Hence
\begin{align}
\label{Res}
J_{\G_n\g_n^0}^{(n)}=-J_{\G_n\g_n^1}^{(n)}
\end{align}
Next, we deform the contour $\G_n$ into the contours $\G_n^1$ and $\G_n^2$ according to Figure \ref{figDeform3}.  
This gives us
\begin{align*}
J_{\G_n\g_n}&=J_{\G_n\g_n^0}+J_{\G_n\g_n^2}\\
&=-J_{\G_n\g_n^1}+J_{\G_n\g_n^2}\quad \text{by (\ref{Res})}\\
&=-(J_{\G_n^1\g_n^1}+J_{\G_n^2\g_n^1})+(J_{\G_n^1,-\g_n^2}+J_{\G_n^2,-\g_n^2}).
\end{align*}
We now instead assume that $x_1\geq x_2$. We then deform deform the contour $\tilde{\G} _n$ into the contours $\tilde{\G} _n^1$ and $\tilde{\G} _n^2$ according to figure \ref{figDeform4}. Finally, we deform the contour $\tilde{\g}_n$ into the contours $\tilde{\g}_n^0$, $\tilde{\g}_n^2$ and $\tilde{\g}_n^1$, according to figure \ref{figDeform5}. However, using that the only residue contained in $\tilde{\g}_n^0$ is the pole $(w-z)^{-1}$, we get
\begin{align*}
J_{\tilde{\G}_n^1\tilde{\g}_n^0}^{(n)}=\frac{1}{2\pi i}\oint_{\tilde{\G}_n^1}\frac{q_n(nz;r)}{q_n(nz;s)}\prod_{k=x_1+1}^{x_2-1}(nz-k)dz=0,
\end{align*}
since the contour $\tilde{\G}_n^1$ contains no poles in $z$. Similarly, we get that 
\begin{align*}
J_{\tilde{\G}_n^2\tilde{\g}_n^0}^{(n)}=0
\end{align*}
since $\tilde{\g}_n^0$ contains no poles in $w$. This gives 
\begin{align*}
J_{\tilde{\G}_n\tilde{\g}_n}^{(n)}&=J_{\tilde{\G}_n^1\tilde{\g}_n^2}^{(n)}+J_{\tilde{\G}_n^1\tilde{\g}_n^1}^{(n)}+J_{\tilde{\G}_n^2\tilde{\g}_n^2}^{(n)}+J_{\tilde{\G}_n^2\tilde{\g}_n^1}^{(n)},
\end{align*}
where the contours are shown in figure \ref{figDeform6}. We now deform the contour $\tilde{\G}_n^1$ into $\tilde{\G}_n^3+C_R$
according to figure \ref{figDeform7}.  Clearly the contribution along $\ell$ vanish and to prove that the contribution from $C_R$ vanishes as $R\to\infty$ we observe that that $g_{n,2}(z)=\nu_{n,2}(\R)\log\vert z\vert+ O(\vert z\vert^{-1})$ and $\lim_{n\to\infty}\nu_{n,2}(\R)=\nu(\R)>0$. From this it is not difficult to see that
$\lim_{R\to \infty}\vert J_{C_R\tilde{\g}_n^i}^{(n)}\vert=0$.
 
We now have the contours as shown in figure \ref{figDeform8}. Using that  $\tilde{\G}_n^3=-\G_n^1$, $\tilde{\g}_n^1=-\g_n^1$and  $\tilde{\G}_n^2=-\G_n^2$, where the minus sign means orientation reversion, we get
\begin{align*}
J_{\tilde{\G}_n\tilde{\g}_n}^{(n)}&=J_{\tilde{\G}_n^3\tilde{\g}_n^1}^{(n)}+J_{\tilde{\G}_n^3\tilde{\g}_n^2}^{(n)}+J_{\tilde{\G}_n^2\tilde{\g}_n^1}^{(n)}+J_{\tilde{\G}_n^2\tilde{\g}_n^2}^{(n)}
%\\&=-J_{\G_n^1\tilde{\g}_n^1}^{(n)}-J_{\G_n^1\tilde{\g}_n^2}^{(n)}+J_{\tilde{\G}_n^2\tilde{\g}_n^1}^{(n)}+J_{\tilde{\G}_n^2\tilde{\g}_n^2}^{(n)}\\
%&=-J_{\G_n^1\tilde{\g}_n^1}^{(n)}-J_{\G_n^1\tilde{\g}_n^2}^{(n)}-J_{\G_n^2\tilde{\g}_n^1}^{(n)}-J_{\G_n^2\tilde{\g}_n^2}^{(n)}
\\&=J_{\G_n^1\g_n^1}^{(n)}-J_{\G_n^1\tilde{\g}_n^2}^{(n)}+J_{\G_n^2\g_n^1}^{(n)}-J_{\G_n^2\tilde{\g}_n^2}^{(n)}.
\end{align*}
 We see that $\g_n^2$ is inside $\G_n^2$ which is inside $\tilde{\g}_n^2$. Using the residue theorem we find
\begin{align*}
J_{\G_n^2\tilde{\g}_n^2}^{(n)}-J_{\G_n^2,-\g_n^2}^{(n)}&=\frac{1}{(2\pi i)^2}\oint_{\Gamma_n^2}dz \frac{1}{z-n^{-1}x_2}\frac{q_n(nz;r)}{q_n(nz;s)}\frac{Q_n(nz;\Delta x_1^n)}{Q_n(nz;\Delta x_2^{(n)})}\\
&=\frac{1}{(2\pi i)^2}\oint_{\Gamma_n^2}dz \frac{1}{z-n^{-1}x_2}\frac{q_n(nz;r)}{q_n(nz;s)}e^{n(h_{n,1}(w)-h_{n,2}(z))}\\
&:=B_n((x_1,y_1),(x_2,y_2)).
\end{align*}
Furthermore, $J_{\G_n^1\tilde{\g}_n^2}^{(n)}=J_{\G_n^1,-\g_n^2}^{(n)}$.
Together, this gives
\begin{align*}
J_{\tilde{\G}_n\tilde{\g}_n}^{(n)}&=J_{\G_n^1\g_n^1}^{(n)}-J_{\G_n^1,-\g_n^2}^{(n)}+J_{\G_n^2\g_n^1}^{(n)}-J_{\G_n^2,-\g_n^2}^{(n)}-B_n.
\end{align*}
Hence, by lemma \ref{RedCancel}, we get
\begin{align*}
&K_{\mathcal{R}}^{(n)}=-1_{x_1<x_2}J_{\G_n\g_n}^{(n)}+1_{x_1\geq x_2}J_{\tilde{\G}_n\tilde{\g}_n}^{(n)}\\
&=-1_{x_1<x_2}(-J_{\G_n^1\g_n^1}^{(n)}-J_{\G_n^2\g_n^1}^{(n)}+J_{\G_n^1,-\g_n^2}^{(n)}+J_{\G_n^2,-\g_n^2}^{(n)})+1_{x_1\geq x_2}(J_{\G_n^1\g_n^1}^{(n)}-J_{\G_n^1,-\g_n^2}^{(n)}+J_{\G_n^2\g_n^1}^{(n)}-J_{\G_n^2,-\g_n^2}^{(n)}-B_n)\\
&=-1_{x_1\geq x_2}B_n+J_{\G_n^1\g_n^1}^{(n)}+J_{\G_n^2\g_n^1}^{(n)}-J_{\G_n^1,-\g_n^2}^{(n)}-J_{\G_n^2,-\g_n^2}^{(n)}\\
&=-1_{x_1\geq x_2}B_n+J_{\G_n^1\g_n^1}^{(n)}+J_{\G_n^2\g_n^1}^{(n)}+J_{\G_n^1\g_n^2}^{(n)}+J_{\G_n^2\g_n^2}^{(n)}
\end{align*}
This gives us finally,
\begin{align*}
&K_{\mathcal{R}}^{(n)}((x_1,y_1),(x_2,y_2))=-1_{x_1\geq x_2}B_n((x_1,y_1),(x_2,y_2))+J_{(\G_n^1+\G_n^2)(\g_n^1+\g_n^2)}^{(n)}\\
%&:=-1_{x_1\geq x_2}B_n((x_1,y_1),(x_2,y_2))+\tilde{K}_{\mathcal{R}}^{(n)}((x_1,y_1),(x_2,y_2))\\
%&=-1_{x_1\geq x_2}\frac{1}{2\pi i^2}\oint_{\Gamma_n^2}dz \frac{1}{z-n^{-1}x_2}\frac{q_n(nz;r)}{q_n(nz;s)}e^{n(h_{n,1}(w)-h_{n,2}(z))}\\
%&+\frac{1}{(2\pi i)^2}\oint_{\G_n^1\cup\G_n^2}dz \oint_{\g_n^1\cup \g_n^2}dw\frac{1}{z-n^{-1}x_2}\frac{q_n(nw;r)}{q_n(nz;s)}\frac{Q_n(nw;\Delta x_1^{(n)})}{Q_n(nz;\Delta x_2^{(n)})}\frac{E_n(nw)E_n(nz)^{-1}}{w-z}\\
&=-1_{x_1\geq x_2}\frac{1}{2\pi i}\oint_{\Gamma_n^2}dz \frac{1}{z-n^{-1}x_2}\frac{q_n(nz;r)}{q_n(nz;s)}e^{n(h_{n,1}(w)-h_{n,2}(z))}\\&+\frac{1}{(2\pi i)^2}\oint_{\G_n^1+\G_n^2}dz \oint_{\g_n^1+\g_n^2}dw\frac{1}{z-n^{-1}x_2}\frac{q_n(nw;r)}{q_n(nz;s)}\frac{e^{n(f_n(w)-nf_n(z))+n(h_{n,1}(w)-h_{n,2}(z))}}{w-z}
\end{align*}
by the definitions of $f_n$, $h_{n,1}$ and $h_{n,2}$.
\end{proof}

\begin{figure}[H]
\centering
\begin{tikzpicture}

\draw (1.2,0) ellipse (4.8cm and 2.2cm);
\draw (1.0,0) ellipse (4.2cm and 1.8cm);
\draw(-9,0) --++(16,0);
\draw (-4.5,0.0) circle (0.7cm);

\draw[arrows=->,line width=1pt](0.56,2.2)--(0.55,2.2);
\draw (1.4,2.4) node {$\g_n^0$};
\draw[arrows=->,line width=1pt](0.61,1.8)--(0.6,1.8);
\draw (1.0,1.3) node {$\G_n$};
\draw[arrows=->,line width=1pt](-4.59,0.7)--(-4.6,0.7);
\draw (-4.7,1.0) node {$-\g_n^2$};

\draw (-8,0.5) node {$a_n$};
\filldraw[fill=red, draw=red] (-8,0) circle (0.05 cm);
\filldraw[fill=red, draw=red] (-7.8,0) circle (0.05 cm);
\filldraw[fill=red, draw=red] (-7.6,0) circle (0.05 cm);
\filldraw[fill=red, draw=red] (-7.4,0) circle (0.05 cm);
\filldraw[fill=red, draw=red] (-7.2,0) circle (0.05 cm);
\filldraw[fill=red, draw=red] (-7.0,0) circle (0.05 cm);
\filldraw[fill=red, draw=red] (-6.8,0) circle (0.05 cm);
\filldraw[fill=red, draw=red] (-6.6,0) circle (0.05 cm);
\filldraw[fill=red, draw=red] (-6.4,0) circle (0.05 cm);
\draw (-6.4,0.5) node {$t_2^{(n)}$};
\draw (-4.5,0.3) node {$t_c^{(n)}$};
\draw(-4.5,-.1) --++(0,.2);
\filldraw[fill=red, draw=red] (-4.4,0) circle (0.05 cm);
\filldraw[fill=blue, draw=blue] (-4.6,0) circle (0.05 cm);
\filldraw[fill=blue, draw=blue] (-4.8,0) circle (0.05 cm);
\draw (-1.8,0.5) node {$t_1^{(n)}$};
\filldraw[fill=blue, draw=blue] (-1.8,0) circle (0.05 cm);
\filldraw[fill=blue, draw=blue] (-1.6,0) circle (0.05 cm);
\filldraw[fill=blue, draw=blue] (-1.2,0) circle (0.05 cm);
\filldraw[fill=blue, draw=blue] (-0.9,0) circle (0.05 cm);
\filldraw[fill=blue, draw=blue] (-0.6,0) circle (0.05 cm);
\filldraw[fill=blue, draw=blue] (-0.4,0) circle (0.05 cm);
\draw (-0.4,0.5) node {$x_c^{(n)}$};
\filldraw[fill=red, draw=red] (-0.2,0) circle (0.05 cm);
\filldraw[fill=red, draw=red] (0,0) circle (0.05 cm);
\filldraw[fill=red, draw=red] (0.2,0) circle (0.05 cm);
\filldraw[fill=red, draw=red] (0.5,0) circle (0.05 cm);
\filldraw[fill=red, draw=red] (0.8,0) circle (0.05 cm);
\filldraw[fill=red, draw=red] (1.1,0) circle (0.05 cm);
\filldraw[fill=red, draw=red] (1.4,0) circle (0.05 cm);
\filldraw[fill=red, draw=red] (1.8,0) circle (0.05 cm);
\filldraw[fill=red, draw=red] (2.2,0) circle (0.05 cm);
\filldraw[fill=red, draw=red] (2.4,0) circle (0.05 cm);

\filldraw[fill=red, draw=red] (3.2,0) circle (0.05 cm);
\filldraw[fill=red, draw=red] (3.4,0) circle (0.05 cm);
\filldraw[fill=red, draw=red] (3.6,0) circle (0.05 cm);
\filldraw[fill=red, draw=red] (3.8,0) circle (0.05 cm);
\filldraw[fill=red, draw=red] (4,0) circle (0.05 cm);
\filldraw[fill=red, draw=red] (4.2,0) circle (0.05 cm);
\filldraw[fill=red, draw=red] (4.4,0) circle (0.05 cm);
\filldraw[fill=red, draw=red] (4.6,0) circle (0.05 cm);
\draw (4.6,0.5) node {$b_n$};

\end{tikzpicture}
\caption{\protect \label{figDeform1} Integration contours}
\end{figure}

\begin{figure}[H]

\centering
\begin{tikzpicture}

\draw (1.2,0) ellipse (5.4cm and 2.6cm);
\draw (1.2,0) ellipse (4.8cm and 2.2cm);
\draw (1.2,0) ellipse (4.4cm and 1.6cm);
\draw(-9,0) --++(16,0);

\draw[arrows=->,line width=1pt](0.56,2.6)--(0.55,2.6);
\draw (1.4,2.8) node {$\g_n^0$};
\draw[arrows=->,line width=1pt](0.61,2.2)--(0.6,2.2);
\draw (1.0,2.0) node {$\G_n$};
\draw[arrows=->,line width=1pt](0.61,1.6)--(0.62,1.6);
\draw (1.0,1.2) node {$\g_n^{1}$};

\draw (-1.8,0.5) node {$t_1^{(n)}$};
\filldraw[fill=blue, draw=blue] (-1.8,0) circle (0.05 cm);
\filldraw[fill=blue, draw=blue] (-1.6,0) circle (0.05 cm);
\filldraw[fill=blue, draw=blue] (-1.2,0) circle (0.05 cm);
\filldraw[fill=blue, draw=blue] (-0.9,0) circle (0.05 cm);
\filldraw[fill=blue, draw=blue] (-0.6,0) circle (0.05 cm);
\filldraw[fill=blue, draw=blue] (-0.4,0) circle (0.05 cm);
\draw (-0.4,0.5) node {$x_c^{(n)}$};
\filldraw[fill=red, draw=red] (-0.2,0) circle (0.05 cm);
\filldraw[fill=red, draw=red] (0,0) circle (0.05 cm);
\filldraw[fill=red, draw=red] (0.2,0) circle (0.05 cm);
\filldraw[fill=red, draw=red] (0.5,0) circle (0.05 cm);
\filldraw[fill=red, draw=red] (0.8,0) circle (0.05 cm);
\filldraw[fill=red, draw=red] (1.1,0) circle (0.05 cm);
\filldraw[fill=red, draw=red] (1.4,0) circle (0.05 cm);
\filldraw[fill=red, draw=red] (1.8,0) circle (0.05 cm);
\filldraw[fill=red, draw=red] (2.2,0) circle (0.05 cm);
\filldraw[fill=red, draw=red] (2.4,0) circle (0.05 cm);

\filldraw[fill=red, draw=red] (3.2,0) circle (0.05 cm);
\filldraw[fill=red, draw=red] (3.4,0) circle (0.05 cm);
\filldraw[fill=red, draw=red] (3.6,0) circle (0.05 cm);
\filldraw[fill=red, draw=red] (3.8,0) circle (0.05 cm);
\filldraw[fill=red, draw=red] (4,0) circle (0.05 cm);
\filldraw[fill=red, draw=red] (4.2,0) circle (0.05 cm);
\filldraw[fill=red, draw=red] (4.4,0) circle (0.05 cm);
\filldraw[fill=red, draw=red] (4.6,0) circle (0.05 cm);
\draw (4.6,0.5) node {$b_n$};

\end{tikzpicture}
\caption{\label{figDeform2} Integration contours}
\end{figure}

\begin{figure}[H]

\centering
\begin{tikzpicture}

\draw (0.4,0) ellipse (6.2cm and 2.2cm);
\draw (1.2,0) ellipse (4.1cm and 1.6cm);
\draw(-9,0) --++(16,0);
\draw (-4.5,0.0) circle (0.9cm);
\draw (-4.5,0.0) circle (0.5cm);

\draw[arrows=->,line width=1pt](0.56,2.2)--(0.55,2.2);
\draw (1.4,2.4) node {\small$\G_n^1$};
\draw[arrows=->,line width=1pt](0.6,1.6)--(0.61,1.6);
\draw (1.0,1.3) node {\small$\g_n^1$};
\draw[arrows=<-,line width=1pt](-4.59,0.9)--(-4.6,0.9);
\draw (-4.85,0.68) node {\small$\g_n^2$};
\draw[arrows=->,line width=1pt](-4.59,0.5)--(-4.6,0.5);
\draw (-4.1,1.1) node {\small$\G_n^2$};

\draw (-8,0.5) node {$a_n$};
\filldraw[fill=red, draw=red] (-8,0) circle (0.05 cm);
\filldraw[fill=red, draw=red] (-7.8,0) circle (0.05 cm);
\filldraw[fill=red, draw=red] (-7.6,0) circle (0.05 cm);
\filldraw[fill=red, draw=red] (-7.4,0) circle (0.05 cm);
\filldraw[fill=red, draw=red] (-7.2,0) circle (0.05 cm);
\filldraw[fill=red, draw=red] (-7.0,0) circle (0.05 cm);
\filldraw[fill=red, draw=red] (-6.8,0) circle (0.05 cm);
\filldraw[fill=red, draw=red] (-6.6,0) circle (0.05 cm);
\filldraw[fill=red, draw=red] (-6.4,0) circle (0.05 cm);
\draw (-6.4,0.5) node {$t_2^{(n)}$};
\draw (-4.5,0.3) node {\small$t_c^{(n)}$};
\draw(-4.5,-.1) --++(0,.2);
\filldraw[fill=red, draw=red] (-4.4,0) circle (0.05 cm);
\filldraw[fill=blue, draw=blue] (-4.6,0) circle (0.05 cm);
\filldraw[fill=blue, draw=blue] (-4.8,0) circle (0.05 cm);
\draw (-1.8,0.5) node {$t_1^{(n)}$};
\filldraw[fill=blue, draw=blue] (-1.8,0) circle (0.05 cm);
\filldraw[fill=blue, draw=blue] (-1.6,0) circle (0.05 cm);
\filldraw[fill=blue, draw=blue] (-1.2,0) circle (0.05 cm);
\filldraw[fill=blue, draw=blue] (-0.9,0) circle (0.05 cm);
\filldraw[fill=blue, draw=blue] (-0.6,0) circle (0.05 cm);
\filldraw[fill=blue, draw=blue] (-0.4,0) circle (0.05 cm);
\draw (-0.4,0.5) node {$x_c^{(n)}$};
\filldraw[fill=red, draw=red] (-0.2,0) circle (0.05 cm);
\filldraw[fill=red, draw=red] (0,0) circle (0.05 cm);
\filldraw[fill=red, draw=red] (0.2,0) circle (0.05 cm);
\filldraw[fill=red, draw=red] (0.5,0) circle (0.05 cm);
\filldraw[fill=red, draw=red] (0.8,0) circle (0.05 cm);
\filldraw[fill=red, draw=red] (1.1,0) circle (0.05 cm);
\filldraw[fill=red, draw=red] (1.4,0) circle (0.05 cm);
\filldraw[fill=red, draw=red] (1.8,0) circle (0.05 cm);
\filldraw[fill=red, draw=red] (2.2,0) circle (0.05 cm);
\filldraw[fill=red, draw=red] (2.4,0) circle (0.05 cm);

\filldraw[fill=red, draw=red] (3.2,0) circle (0.05 cm);
\filldraw[fill=red, draw=red] (3.4,0) circle (0.05 cm);
\filldraw[fill=red, draw=red] (3.6,0) circle (0.05 cm);
\filldraw[fill=red, draw=red] (3.8,0) circle (0.05 cm);
\filldraw[fill=red, draw=red] (4,0) circle (0.05 cm);
\filldraw[fill=red, draw=red] (4.2,0) circle (0.05 cm);
\filldraw[fill=red, draw=red] (4.4,0) circle (0.05 cm);
\filldraw[fill=red, draw=red] (4.6,0) circle (0.05 cm);
\draw (4.6,0.5) node {$b_n$};

\end{tikzpicture}
\caption{\label{figDeform3} Integration contours}
\end{figure}

\begin{figure}[H]

\centering
\begin{tikzpicture}

\draw (-2,0) ellipse (7.6cm and 2.6cm);
\draw (-7.1,0) ellipse (1.2cm and 1.2cm);
%\draw (-7.1,0) ellipse (1.4cm and 1.4cm);
\draw(-10,0) --++(16,0);
\draw (-4.5,0) circle (0.8cm);
%\draw (-4.5,0) circle (1.0cm);

\draw[arrows=->,line width=1pt](-2,2.6)--(-2.02,2.6);
\draw (2,2.5) node {$\tilde{\g}_n$};
%\draw[arrows=->,line width=1pt](-7.09,1.4)--(-7.1,1.4);
\draw (-7.1,0.8) node {\small$\tilde{\G}_n^1$};
\draw[arrows=->,line width=1pt](-7.09,1.2)--(-7.1,1.2);
%\draw (-6.9,1.7) node {\small$\tilde{\g}_n^1$};
\draw[arrows=->,line width=1pt](-4.49,0.8)--(-4.5,0.8);
\draw (-4.4,0.55) node {\small$\tilde{\G}_n^2$};
%\draw[arrows=->,line width=1pt](-4.49,1)--(-4.5,1);
%\draw (-4.3,1.4) node {\small$\tilde{\g}_n^2$};
\draw (-8,0.3) node {$a_n$};

\filldraw[fill=red, draw=red] (-8,0) circle (0.05 cm);
\filldraw[fill=red, draw=red] (-7.8,0) circle (0.05 cm);
\filldraw[fill=red, draw=red] (-7.6,0) circle (0.05 cm);
\filldraw[fill=red, draw=red] (-7.4,0) circle (0.05 cm);
\filldraw[fill=red, draw=red] (-7.2,0) circle (0.05 cm);
\filldraw[fill=red, draw=red] (-7.0,0) circle (0.05 cm);
\filldraw[fill=red, draw=red] (-6.8,0) circle (0.05 cm);
\filldraw[fill=red, draw=red] (-6.6,0) circle (0.05 cm);
\filldraw[fill=red, draw=red] (-6.4,0) circle (0.05 cm);
\draw (-6.4,0.5) node {$t_2^{(n)}$};
\draw (-4.5,0.3) node {\small$t_c^{(n)}$};
\draw(-4.5,-.1) --++(0,.2);
\filldraw[fill=red, draw=red] (-4.4,0) circle (0.05 cm);
\filldraw[fill=blue, draw=blue] (-4.6,0) circle (0.05 cm);
\filldraw[fill=blue, draw=blue] (-4.8,0) circle (0.05 cm);
\draw (-1.6,0.5) node {$t_1^{(n)}$};
\filldraw[fill=blue, draw=blue] (-1.8,0) circle (0.05 cm);
\filldraw[fill=blue, draw=blue] (-1.6,0) circle (0.05 cm);
\filldraw[fill=blue, draw=blue] (-1.2,0) circle (0.05 cm);
\filldraw[fill=blue, draw=blue] (-0.9,0) circle (0.05 cm);
\filldraw[fill=blue, draw=blue] (-0.6,0) circle (0.05 cm);
\filldraw[fill=blue, draw=blue] (-0.4,0) circle (0.05 cm);
\draw (-0.4,0.5) node {$\chi_c^{(n)}$};
\filldraw[fill=red, draw=red] (-0.2,0) circle (0.05 cm);
\filldraw[fill=red, draw=red] (0,0) circle (0.05 cm);
\filldraw[fill=red, draw=red] (0.2,0) circle (0.05 cm);
\filldraw[fill=red, draw=red] (0.5,0) circle (0.05 cm);
\filldraw[fill=red, draw=red] (0.8,0) circle (0.05 cm);
\filldraw[fill=red, draw=red] (1.1,0) circle (0.05 cm);
\filldraw[fill=red, draw=red] (1.4,0) circle (0.05 cm);
\filldraw[fill=red, draw=red] (1.8,0) circle (0.05 cm);
\filldraw[fill=red, draw=red] (2.2,0) circle (0.05 cm);
\filldraw[fill=red, draw=red] (2.4,0) circle (0.05 cm);

\filldraw[fill=red, draw=red] (3.2,0) circle (0.05 cm);
\filldraw[fill=red, draw=red] (3.4,0) circle (0.05 cm);
\filldraw[fill=red, draw=red] (3.6,0) circle (0.05 cm);
\filldraw[fill=red, draw=red] (3.8,0) circle (0.05 cm);
\filldraw[fill=red, draw=red] (4,0) circle (0.05 cm);
\filldraw[fill=red, draw=red] (4.2,0) circle (0.05 cm);
\filldraw[fill=red, draw=red] (4.4,0) circle (0.05 cm);
\filldraw[fill=red, draw=red] (4.6,0) circle (0.05 cm);
\draw (4.45,0.3) node {$b_n$};

\end{tikzpicture}
\caption{\protect\label{figDeform4} Integration contours}
\end{figure}

\begin{figure}[H]

\centering
\begin{tikzpicture}

\draw (1.4,0) ellipse (3.6cm and 1.6cm);
\draw (-7.1,0) ellipse (1.2cm and 1.2cm);
\draw (-7.1,0) ellipse (1.4cm and 1.4cm);
\draw(-10,0) --++(16,0);
\draw (-4.5,0) circle (0.8cm);
\draw (-4.5,0) circle (1.0cm);

\draw[arrows=<-,line width=1pt](1.39,1.6)--(1.4,1.6);
\draw (2,2) node {$\tilde{\g}_n^1$};
\draw[arrows=->,line width=1pt](-7.09,1.4)--(-7.1,1.4);
\draw (-7.1,0.8) node {\small$\tilde{\G}_n^1$};
\draw[arrows=->,line width=1pt](-7.09,1.2)--(-7.1,1.2);
\draw (-6.9,1.7) node {\small$\tilde{\g}_n^0$};
\draw[arrows=->,line width=1pt](-4.49,0.8)--(-4.5,0.8);
\draw (-4.4,0.55) node {\small$\tilde{\G}_n^2$};
\draw[arrows=->,line width=1pt](-4.49,1)--(-4.5,1);
\draw (-4.3,1.4) node {\small$\tilde{\g}_n^2$};
\draw (-8,0.3) node {$a_n$};

\filldraw[fill=red, draw=red] (-8,0) circle (0.05 cm);
\filldraw[fill=red, draw=red] (-7.8,0) circle (0.05 cm);
\filldraw[fill=red, draw=red] (-7.6,0) circle (0.05 cm);
\filldraw[fill=red, draw=red] (-7.4,0) circle (0.05 cm);
\filldraw[fill=red, draw=red] (-7.2,0) circle (0.05 cm);
\filldraw[fill=red, draw=red] (-7.0,0) circle (0.05 cm);
\filldraw[fill=red, draw=red] (-6.8,0) circle (0.05 cm);
\filldraw[fill=red, draw=red] (-6.6,0) circle (0.05 cm);
\filldraw[fill=red, draw=red] (-6.4,0) circle (0.05 cm);
\draw (-6.4,0.5) node {$t_2^{(n)}$};
\draw (-4.5,0.3) node {\small$t_c^{(n)}$};
\draw(-4.5,-.1) --++(0,.2);
\filldraw[fill=red, draw=red] (-4.4,0) circle (0.05 cm);
\filldraw[fill=blue, draw=blue] (-4.6,0) circle (0.05 cm);
\filldraw[fill=blue, draw=blue] (-4.8,0) circle (0.05 cm);
\draw (-1.6,0.5) node {$t_1^{(n)}$};
\filldraw[fill=blue, draw=blue] (-1.8,0) circle (0.05 cm);
\filldraw[fill=blue, draw=blue] (-1.6,0) circle (0.05 cm);
\filldraw[fill=blue, draw=blue] (-1.2,0) circle (0.05 cm);
\filldraw[fill=blue, draw=blue] (-0.9,0) circle (0.05 cm);
\filldraw[fill=blue, draw=blue] (-0.6,0) circle (0.05 cm);
\filldraw[fill=blue, draw=blue] (-0.4,0) circle (0.05 cm);
\draw (-0.4,0.5) node {$x_c^{(n)}$};
\filldraw[fill=red, draw=red] (-0.2,0) circle (0.05 cm);
\filldraw[fill=red, draw=red] (0,0) circle (0.05 cm);
\filldraw[fill=red, draw=red] (0.2,0) circle (0.05 cm);
\filldraw[fill=red, draw=red] (0.5,0) circle (0.05 cm);
\filldraw[fill=red, draw=red] (0.8,0) circle (0.05 cm);
\filldraw[fill=red, draw=red] (1.1,0) circle (0.05 cm);
\filldraw[fill=red, draw=red] (1.4,0) circle (0.05 cm);
\filldraw[fill=red, draw=red] (1.8,0) circle (0.05 cm);
\filldraw[fill=red, draw=red] (2.2,0) circle (0.05 cm);
\filldraw[fill=red, draw=red] (2.4,0) circle (0.05 cm);

\filldraw[fill=red, draw=red] (3.2,0) circle (0.05 cm);
\filldraw[fill=red, draw=red] (3.4,0) circle (0.05 cm);
\filldraw[fill=red, draw=red] (3.6,0) circle (0.05 cm);
\filldraw[fill=red, draw=red] (3.8,0) circle (0.05 cm);
\filldraw[fill=red, draw=red] (4,0) circle (0.05 cm);
\filldraw[fill=red, draw=red] (4.2,0) circle (0.05 cm);
\filldraw[fill=red, draw=red] (4.4,0) circle (0.05 cm);
\filldraw[fill=red, draw=red] (4.6,0) circle (0.05 cm);
\draw (4.45,0.3) node {$b_n$};

\end{tikzpicture}
\caption{\protect\label{figDeform5} Integration contours}
\end{figure}

\begin{figure}[H]
\centering
\begin{tikzpicture}

\draw (1.4,0) ellipse (3.6cm and 1.6cm);
\draw (-7.5,0) ellipse (1.6cm and 1.4cm);
\draw(-10,0) --++(16,0);
\draw (-4.5,0) circle (0.8cm);
\draw (-4.5,0) circle (1.0cm);

\draw[arrows=<-,line width=1pt](1.39,1.6)--(1.4,1.6);
\draw (2,2) node {$\tilde{\g}_n^1$};
\draw[arrows=->,line width=1pt](-7.49,1.4)--(-7.5,1.4);
\draw (-7.5,1.8) node {\small$\tilde{\G}_n^1$};
\draw[arrows=->,line width=1pt](-4.49,0.8)--(-4.5,0.8);
\draw (-4.4,0.55) node {\small$\tilde{\G}_n^2$};
\draw[arrows=->,line width=1pt](-4.49,1)--(-4.5,1);
\draw (-4.3,1.4) node {\small$\tilde{\g}_n^2$};
\draw (-8,0.3) node {$a_n$};

\filldraw[fill=red, draw=red] (-8,0) circle (0.05 cm);
\filldraw[fill=red, draw=red] (-7.8,0) circle (0.05 cm);
\filldraw[fill=red, draw=red] (-7.6,0) circle (0.05 cm);
\filldraw[fill=red, draw=red] (-7.4,0) circle (0.05 cm);
\filldraw[fill=red, draw=red] (-7.2,0) circle (0.05 cm);
\filldraw[fill=red, draw=red] (-7.0,0) circle (0.05 cm);
\filldraw[fill=red, draw=red] (-6.8,0) circle (0.05 cm);
\filldraw[fill=red, draw=red] (-6.6,0) circle (0.05 cm);
\filldraw[fill=red, draw=red] (-6.4,0) circle (0.05 cm);
\draw (-6.4,0.5) node {$t_2^{(n)}$};
\draw (-4.5,0.3) node {\small$t_c^{(n)}$};
\draw(-4.5,-.1) --++(0,.2);
\filldraw[fill=red, draw=red] (-4.4,0) circle (0.05 cm);
\filldraw[fill=blue, draw=blue] (-4.6,0) circle (0.05 cm);
\filldraw[fill=blue, draw=blue] (-4.8,0) circle (0.05 cm);
\draw (-1.6,0.5) node {$t_1^{(n)}$};
\filldraw[fill=blue, draw=blue] (-1.8,0) circle (0.05 cm);
\filldraw[fill=blue, draw=blue] (-1.6,0) circle (0.05 cm);
\filldraw[fill=blue, draw=blue] (-1.2,0) circle (0.05 cm);
\filldraw[fill=blue, draw=blue] (-0.9,0) circle (0.05 cm);
\filldraw[fill=blue, draw=blue] (-0.6,0) circle (0.05 cm);
\filldraw[fill=blue, draw=blue] (-0.4,0) circle (0.05 cm);
\draw (-0.4,0.5) node {$x_c^{(n)}$};
\filldraw[fill=red, draw=red] (-0.2,0) circle (0.05 cm);
\filldraw[fill=red, draw=red] (0,0) circle (0.05 cm);
\filldraw[fill=red, draw=red] (0.2,0) circle (0.05 cm);
\filldraw[fill=red, draw=red] (0.5,0) circle (0.05 cm);
\filldraw[fill=red, draw=red] (0.8,0) circle (0.05 cm);
\filldraw[fill=red, draw=red] (1.1,0) circle (0.05 cm);
\filldraw[fill=red, draw=red] (1.4,0) circle (0.05 cm);
\filldraw[fill=red, draw=red] (1.8,0) circle (0.05 cm);
\filldraw[fill=red, draw=red] (2.2,0) circle (0.05 cm);
\filldraw[fill=red, draw=red] (2.4,0) circle (0.05 cm);

\filldraw[fill=red, draw=red] (3.2,0) circle (0.05 cm);
\filldraw[fill=red, draw=red] (3.4,0) circle (0.05 cm);
\filldraw[fill=red, draw=red] (3.6,0) circle (0.05 cm);
\filldraw[fill=red, draw=red] (3.8,0) circle (0.05 cm);
\filldraw[fill=red, draw=red] (4,0) circle (0.05 cm);
\filldraw[fill=red, draw=red] (4.2,0) circle (0.05 cm);
\filldraw[fill=red, draw=red] (4.4,0) circle (0.05 cm);
\filldraw[fill=red, draw=red] (4.6,0) circle (0.05 cm);
\draw (4.45,0.3) node {$b_n$};

\end{tikzpicture}

\caption{\label{figDeform6} Integration contours}
\end{figure}

\begin{figure}[H]
\centering
\begin{tikzpicture}

\draw (-8,0) --(8,0);
\draw[thick] (0,0) circle (4cm);
\draw[thick,->] (0,4) -- (-0.04,4);
\draw[thick] (1,0) ellipse (2cm and 1cm);
\draw[thick,<-] (1,1) -- (0.96,1);
\draw[thick] (3,0) -- (4,0);
\draw[->] (3.25,0.1) -- (3.75,0.1);
\draw[<-] (3.25,-0.1) -- (3.75,-0.1);
\draw (3.5,0.3) node {$\ell$};
\draw (1,1.5) node {$\tilde{\G}_n^3$};
\draw (0.7,3.7) node {$C_R$};
\filldraw (-1.5,0) circle (0.04cm);
\filldraw (-0.5,0) circle (0.04cm);
\filldraw (2.5,0) circle (0.04cm);
\filldraw (-3,0) circle (0.04cm);

\draw (-1.5,0.4) node {\small$t_2^{(n)}$};
\draw (-0.5,0.4) node {\small$t_c^{(n)}$};
\draw (2.5,0.4) node {\small$b_n$};
\draw (-3,0.4) node {\small$a_n$};

\end{tikzpicture}

\caption{\label{figDeform7} Integration contours}
\end{figure}

\begin{figure}[H]

\centering
\begin{tikzpicture}

\draw (0.4,0) ellipse (6.2cm and 2.2cm);
\draw (1.2,0) ellipse (4.1cm and 1.6cm);
\draw(-9,0) --++(16,0);
\draw (-4.5,0.0) circle (0.9cm);
\draw (-4.5,0.0) circle (0.5cm);

\draw[arrows=<-,line width=1pt](0.56,2.2)--(0.55,2.2);
\draw (1.4,2.4) node {\small$\tilde{\G}_n^3$};
\draw[arrows=<-,line width=1pt](0.6,1.6)--(0.61,1.6);
\draw (1.0,1.3) node {\small$\tilde{\g}_n^1$};
\draw[arrows=->,line width=1pt](-4.59,0.9)--(-4.6,0.9);
\draw (-4.85,0.68) node {\small$\tilde{\G}_n^2$};
\draw[arrows=->,line width=1pt](-4.59,0.5)--(-4.6,0.5);
\draw (-4.1,1.1) node {\small$\tilde{\g}_n^2$};

\draw (-8,0.5) node {$a_n$};
\filldraw[fill=red, draw=red] (-8,0) circle (0.05 cm);
\filldraw[fill=red, draw=red] (-7.8,0) circle (0.05 cm);
\filldraw[fill=red, draw=red] (-7.6,0) circle (0.05 cm);
\filldraw[fill=red, draw=red] (-7.4,0) circle (0.05 cm);
\filldraw[fill=red, draw=red] (-7.2,0) circle (0.05 cm);
\filldraw[fill=red, draw=red] (-7.0,0) circle (0.05 cm);
\filldraw[fill=red, draw=red] (-6.8,0) circle (0.05 cm);
\filldraw[fill=red, draw=red] (-6.6,0) circle (0.05 cm);
\filldraw[fill=red, draw=red] (-6.4,0) circle (0.05 cm);
\draw (-6.4,0.5) node {$t_2^{(n)}$};
\draw (-4.5,0.3) node {\small$t_c^{(n)}$};
\draw(-4.5,-.1) --++(0,.2);
\filldraw[fill=red, draw=red] (-4.4,0) circle (0.05 cm);
\filldraw[fill=blue, draw=blue] (-4.6,0) circle (0.05 cm);
\filldraw[fill=blue, draw=blue] (-4.8,0) circle (0.05 cm);
\draw (-1.8,0.5) node {$t_1^{(n)}$};
\filldraw[fill=blue, draw=blue] (-1.8,0) circle (0.05 cm);
\filldraw[fill=blue, draw=blue] (-1.6,0) circle (0.05 cm);
\filldraw[fill=blue, draw=blue] (-1.2,0) circle (0.05 cm);
\filldraw[fill=blue, draw=blue] (-0.9,0) circle (0.05 cm);
\filldraw[fill=blue, draw=blue] (-0.6,0) circle (0.05 cm);
\filldraw[fill=blue, draw=blue] (-0.4,0) circle (0.05 cm);
\draw (-0.4,0.5) node {$x_c^{(n)}$};
\filldraw[fill=red, draw=red] (-0.2,0) circle (0.05 cm);
\filldraw[fill=red, draw=red] (0,0) circle (0.05 cm);
\filldraw[fill=red, draw=red] (0.2,0) circle (0.05 cm);
\filldraw[fill=red, draw=red] (0.5,0) circle (0.05 cm);
\filldraw[fill=red, draw=red] (0.8,0) circle (0.05 cm);
\filldraw[fill=red, draw=red] (1.1,0) circle (0.05 cm);
\filldraw[fill=red, draw=red] (1.4,0) circle (0.05 cm);
\filldraw[fill=red, draw=red] (1.8,0) circle (0.05 cm);
\filldraw[fill=red, draw=red] (2.2,0) circle (0.05 cm);
\filldraw[fill=red, draw=red] (2.4,0) circle (0.05 cm);

\filldraw[fill=red, draw=red] (3.2,0) circle (0.05 cm);
\filldraw[fill=red, draw=red] (3.4,0) circle (0.05 cm);
\filldraw[fill=red, draw=red] (3.6,0) circle (0.05 cm);
\filldraw[fill=red, draw=red] (3.8,0) circle (0.05 cm);
\filldraw[fill=red, draw=red] (4,0) circle (0.05 cm);
\filldraw[fill=red, draw=red] (4.2,0) circle (0.05 cm);
\filldraw[fill=red, draw=red] (4.4,0) circle (0.05 cm);
\filldraw[fill=red, draw=red] (4.6,0) circle (0.05 cm);
\draw (4.6,0.5) node {$b_n$};

\end{tikzpicture}
\caption{\label{figDeform8} Integration contours}
\end{figure}

\subsection{Global Choice of Contours}\label{sec:globalcontours}

Recall the asymptotic function $f(w;\chi_c)$ given by (\ref{AsympFuncCusp}),
\begin{align*}
f(w;\chi_c)=\int_{\R}\log(w-t)d\nu(t)=U^{\nu}(w)+i\int_{\R}\text{arg}(w-t)d\nu(t).
\end{align*}
%where $\log(w-t)$ is the principal branch of the logarithm, and where $\nu=\mu\big\vert_{[a,t_2]\cup[\chi_c,b]}-(\lambda-\mu)\big\vert_{[t_1,\chi_c]}.$
For every $\delta>0$ we let
$
\Omega_{\delta}:=\{w\in\C: d(w,\text{supp}(\nu))>\delta \}
$.

\begin{Lem}
\label{UniformConv}
The functions $f_n(w),g_{n,1}(w)$ and $g_{n,2}(w)$ converge uniformly to $f(w;\chi_c)$ on $\overline{\Omega}_{\delta}\bigcap \overline{B(0,R)}$, for any $\delta>0$ and any $R>0$.
\end{Lem}
\begin{proof}
This is a standard consequence of the weak convergence of $\mu_n$ and
Vitali's theorem, (see \cite{Remmert} page 157). 
\end{proof}

\begin{Lem}
\label{lemArgFunc}
Consider the function $\nu([x,+\infty))$, where the measure $\nu$ is as defined in (\ref{AsympMeasureCusp}). Then $\nu([x,+\infty))$ is
monotonically decreasing on $(-\infty,t_2)$, constant in $(t_2,t_1)$,
monotonically increasing on $(t_1,\chi_c)$, and
monotonically decreasing on $(\chi_c,+\infty)$.
More precisely, 
\begin{align}
\label{Arg}
\nu([x,+\infty))=\left\{
 \begin{array}{ll}
 \eta_c & \text{if } x\le a\\
 \mu([x,+\infty))-(\chi_c-t_c)  & a<x<t_2\\
 \mu([t_c,+\infty))- (\chi_c-t_c) & t_2<x<t_1\\
 \mu([\chi_c,+\infty))- \int_{x}^{\chi_c}(1-\varphi(t))dt& t_1\leq x\leq \chi_c\\
 \mu([x,+\infty)) & \chi_c<x<b\\
  0 & x\in (b,+\infty)
 \end{array} \right. .
\end{align}
\end{Lem}

\begin{proof}
By definition of $\nu$ we have
\begin{align*}
\nu([x,+\infty))&=\int_{x}^{\infty}(\chi_{[a,t_2]}(t)+\chi_{[t_1,\chi_c]}(t)+\chi_{[\chi_c,b]}(t))\varphi(t)dt-\int_{x}^{\infty}\chi_{[t_1,\chi_c]}dt\\
&=\int_{x}^{\infty}(\chi_{[a,t_2]}(t)+\chi_{[t_c,\chi_c]}(t)+\chi_{[\chi_c,b]}(t))\varphi(t)dt-\int_{x}^{\infty}\chi_{[t_c,\chi_c]}dt\\
&=\mu([x,+\infty))-\int_{x}^{\infty}\chi_{[t_c,\chi_c]}(t)dt,
\end{align*}
from which (\ref{Arg}) follows. The monotonicity properties are immediate from these formulas.
\end{proof}

Since $f(\overline{w};\chi_c)=\overline{f(w;\chi_c)}$ it is sufficient to prove the existence of the contours in the upper-half plane $\Hp$.
\begin{Lem}
\label{GlobalDescentAsympFunc}
The asymptotic function $f(w;\chi_c)$ has global steepest ascent/descent contours in the upper half plane as shown in figure \ref{figAscentAsymp}. 
%In particular the contour $\G_n^1$ in Proposition \ref{IntKernel3} homotopy equivalent to the contours $\mathcal{A}$ on $\hat{\C}\backslash \supp(\nu_{n,2}^+)$, where $\hat{\C}$ is the Riemann sphere, and the contour $\g_n^1$ Proposition \ref{IntKernel3} is homotopy equivalent to $\mathcal{D}$ on $\C\backslash \supp(\nu_{n,1}^-)$. 
\end{Lem}

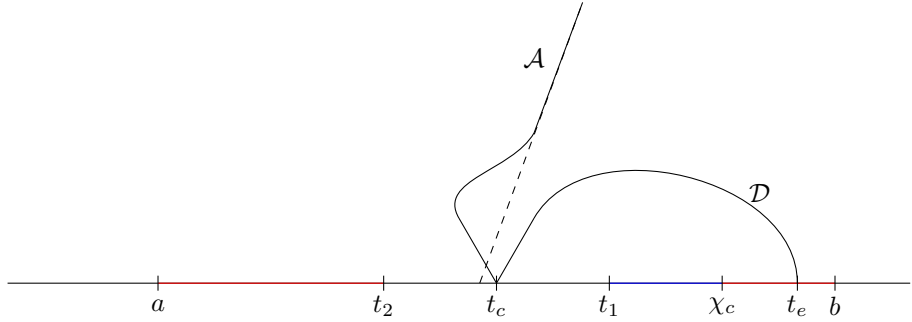
\begin{figure}[H]
\centering
\begin{tikzpicture}
\draw (-6,0) --(6,0);
\draw[red] (-4,0) -- (-1,0);
\draw[blue] (2,0) -- (3.5,0);
\draw[red] (3.5,0) -- (5,0);
\draw (-4,0.1)--(-4,-0.1);
\draw (-1,0.1)--(-1,-0.1);
\draw (0.5,0.1)--(0.5,-0.1);
\draw (2,0.1)--(2,-0.1);
\draw (3.5,0.1)--(3.5,-0.1);
\draw (4.5,0.1)--(4.5,-0.1);
\draw (5,0.1)--(5,-0.1);
\draw (-4,-0.3) node {$a$};
\draw (-1,-0.3) node {$t_2$};
\draw (0.5,-0.3) node {$t_c$};
\draw (2,-0.3) node {$t_1$};
\draw (3.5,-0.3) node {$\chi_c$};
\draw (4.5,-0.3) node {$t_e$};
\draw (5,-0.3) node {$b$};
\draw (0.5,0) -- (1,{0.5*sqrt (3)});
\draw (0.5,0) -- (0,{0.5*sqrt (3)});
\draw (1,{0.5*sqrt (3)}) to [out=60,in=90] (4.5,0);
\draw (0,{0.5*sqrt (3)}) to [out=120,in=240] (1,2);
\draw (1,2) --({0.277+4*0.342},{4*0.934});
\draw[dashed] (0.277,0) -- ({0.277+4*0.342},{4*0.934});
\draw (4,1.2) node {$\mathcal{D}$};
\draw (1,3) node {$\mathcal{A}$};
\end{tikzpicture}
\caption{\label{figAscentAsymp} Steepest ascent and descent paths for the asymptotic function $f(w;\chi_c)$. The support of $\nu^{+}$ is indicated by a red line and the support of $\nu^{-}$ by a blue line.}

\end{figure}

\begin{proof}
Since $f(\overline{w};\chi_c)=\overline{f(w;\chi_c)}$ it is sufficient to prove the existence of the contours in the upper-half plane $\Hp$. We note that $U^{\nu}$ is real analytic in $\C\backslash \text{supp}(\nu)$ and that $\int_{\R}\arg(w-t)d\nu(t)$ is real analytic in $\C\backslash (-\infty,b]$. Moreover, the boundary values on the real axis are given by
\begin{align*}
\lim_{\substack{w\to x\in \R\\ w\in \Hp}} U^{\nu}(w)=U^{\nu}(x)
\end{align*}
and
\begin{align*}
\lim_{\substack{w\to x\in \R\\ w\in \Hp}} \int_{\R}\arg(x-t)d\nu(t)=\pi\nu([x,+\infty)).
\end{align*}
In particular, $U^{\nu}(w)$ has a continuous extension to all of $\C$. 

Since $f'''(t_c;\chi_c,\eta_c)>0$ the local steepest ascent/descent structure around $t_c$ is as in figure \ref{figLocalContourAsymp}.

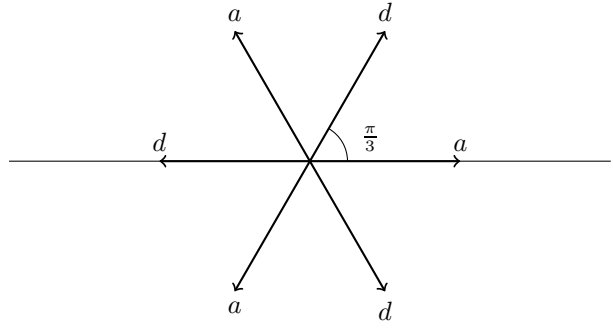
\begin{figure}[H]
\centering
\begin{tikzpicture}
\draw (-4,0) --(4,0);
\draw[thick,->,above] (0,0)--(2,0) node {$a$};
\draw[thick,->,above] (0,0)--(1,{sqrt(3)}) node {$d$};
\draw[thick,->,above] (0,0)--(-1,{sqrt(3)}) node {$a$};
\draw[thick,->,above] (0,0)--(-2,0) node {$d$};
\draw[thick,->,below] (0,0)--(1,-{sqrt(3)}) node {$d$};
\draw[thick,->,below] (0,0)--(-1,-{sqrt(3)}) node {$a$};
\draw (0.5,0) arc (0:60:0.5);
\draw (0.8,0.3) node {\small$\frac{\pi}{3}$};
\end{tikzpicture}
\caption{\protect\label{figLocalContourAsymp}Local steepest ascent/descent contours of the asymptotic function $f(\omega,\chi_c)$, where $a$ denotes ascent contour and $d$ denotes descent contour.}

\end{figure}

Recall that the contours of steepest ascent/descent are those for which $\text{Im\,}f(w)=\text{Im\,}f(t_c)=\pi\mu([t_c,b])$. Note that $U^{\nu}(w)=\nu(\R)\log\vert w\vert+O(\vert w\vert^{-1})$ as $\vert w\vert\to \infty$. Therefore $\lim_{\substack{w\to \infty\\ w\in \Hp}}U^{\nu}(w)=+\infty$. This implies that the descent contour have to be contained in some ball $B(0,R)$, $R$ sufficiently large. On the other hand since the function $\text{Im\,}f(w)$ is real analytic, the curve $\text{Im\,}f(w)=\text{Im\,}f(t_c)$ has to be either a closed curve in $B(0,R)\backslash \text{supp}(\nu)$ or end somewhere in $\text{supp}(\nu)$. Assume the first case. Then, clearly there has to be a point $t_p\neq t_c$ on the curve such that $f'(t_p)=0$. By Theorem 3.1 in \cite{Duse14a}, we must for such a $t_p$ have $t_p\in \R\backslash \supp(\nu)$ and $f''(t_p)\neq 0$. At such a point we have descent contour exiting at angle $\pm \pi/2$ and ascent contours exiting at $0$ and $\pi$. This gives a contradiction. We may therefore assume the second case holds

It follows from Lemma \ref{lemArgFunc} that if $\pi\nu([t_c,\infty))<0$ then the equation $\pi\nu([x,\infty))=\pi\mu([t_c,+\infty))$ has no solution and the steepest
descent contour from $t_c$ has to go to infinity, which is impossible. If $\pi\nu([t_c,\infty))=0$, then we have to be in the first case above which is impossible.
Thus, $\pi\nu([t_c,\infty))>0$ and Lemma \ref{lemArgFunc} implies that the equation $\pi\nu([x,\infty))=\pi\mu([t_c,+\infty))$ 
has at least one solution $t_e$ for $x\in (\chi_c,b)$ and no solution for $x\in(-\infty,t_2)\cup (t_1,\chi_c)\cup[b,\infty)$. In figure \ref{figArgument} we give a plot of what the function $\pi\nu([x,+\infty))$ may look like. If $\nu([x,\infty))$ is strictly monotonically decreasing at $t_e$, then $t_e$ is the unique solution to the equation above. Otherwise, by the monotonicity of $\nu([x,\infty))$, there exists an interval $[t_e^-,t_e^+]$ such that $\pi\nu([x,\infty))=\pi\mu([t_c,+\infty))$ for all $x\in[t_e^-,t_e^+]$. In particular, $(t_e^-,t_e^+)\cap\text{supp}(\mu)=\varnothing$.

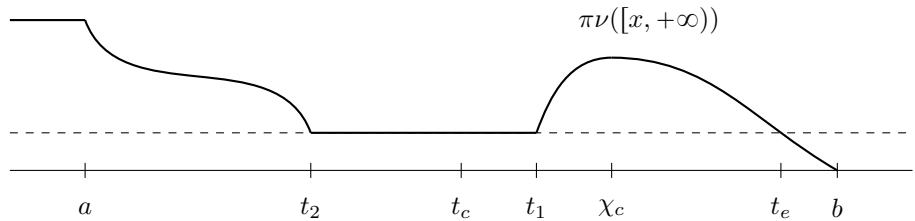
\begin{figure}[H]
\centering
\begin{tikzpicture}
\draw (-6,0) --(6,0);
\draw[dashed] (-6,0.5)--(6,0.5);
\draw (-5,-0.5) node {$a$};
\draw (-2,-0.5) node {$t_2$};
\draw (-5,-0.1) --(-5,0.1);
\draw (-2,-0.1) --(-2,0.1);
\draw (0,-0.5) node {$t_c$};
\draw (0,-0.1) --(0,0.1);
\draw (1,-0.5) node {$t_1$};
\draw (1,-0.1) --(1,0.1);
\draw (2,-0.5) node {$\chi_c$};
\draw (2,-0.1) --(2,0.1);
\draw (4.25,-0.5) node {$t_e$};
\draw (4.25,-0.1) --(4.25,0.1);
\draw (5,-0.5) node {$b$};
\draw (5,-0.1) --(5,0.1);
\draw[thick] (-6,2) --(-5,2);
\draw[thick] (-2,0.5) --(1,0.5);
\draw[thick] (-5,2) to [out=290,in=110] (-2,0.5);
\draw[thick] (1,0.5) to [out=70,in=180] (2,1.5);
\draw[thick] (2,1.5) to [out=0,in=150] (5,0);
\draw (2.5,2) node {$\pi \nu([x,+\infty))$};
\end{tikzpicture}
\caption{An example of a plot of the function $\int \arg(x-t)d\nu(t)$ for some possible $\nu$.}
\label{figArgument}
\end{figure}

By Lemma \ref{lemArgFunc} and the discussion above, the only possible end points are $t_2,t_1$ and $t_e$ or $t_2,t_1,t_e^-$ and $t_e^+$. Assume that it ends at $t_1$. Then we get a closed contour in $\overline{\Hp}$ containing the interval $[t_c,t_1]$, and such that the boundary value of $\text{Im}[f(w)]$ equals $\text{Im}[f(t_c)]=\pi\mu([t_c,b])$ everywhere on the curve. However, since $\text{Im}[f(w)]$ is harmonic inside the domain bounded by the curve, this implies that $\text{Im}[f(w)]$ is constant, a contradiction. Similarly, the descent contour cannot end in $t_2$. Thus it has to end either at $t_e$ or one of $t_e^-$ and $t_e^+$. 

%However, $\nabla \text{Im}[f(w)]=(-P_v\rho(u),H_v\rho(u))$, where
%\begin{align*}
%P_v\rho(u)&=\int_{\R}\frac{v\rho(t)dt}{(u-t)^2+v^2}\\
%H_v\rho(u)&=\int_{\R}\frac{(u-t)\rho(t)dt}{(u-t)^2+v^2}
%\end{align*}
%and where $\rho$ is the density of $\mu$. Therefore the limit $\lim_{w\to t_e}\nabla \text{Im}[f(w)]$ need not exists (and similarly for $t_e^-$ or $t_e^+$). In particular we need not have $\mathcal{H}^1(B(t_e,\delta)\bigcap\g_d)<+\infty$ for some $\delta$, and where $\mathcal{H}^1$ denotes the 1-dimensional Hausdorff measure. However for $\delta$ sufficiently small there exists an $\eps>0$ such that $U^{\nu}(w)<U^{\nu}(t_e)+\eps<U^{\nu}(t_c)$ for all $w\in B(t_e,\delta)$. Therefore any smooth curve in $B(t_e,\delta)$ which ends at $t_e$ (or in some neighborhood of $t_e$) and connects with the curve $\text{Im}[f(w)]=\text{Im}[f(t_c)]$ suffice. 

This proves the existence of the global steepest descent path of $f(w,\chi_c)$. We now consider the ascent path. Recall that the ascent and descent paths cannot intersect. By considering the local ascent and descent contours we know it cannot end at $t_1$. Suppose that it ends a $t_2$. However, by a similar argument as before, this is not possible. Moreover, by the continuity of $U^{\nu}(w)$ it cannot end at $t_e$. Thus, the ascent contour will become an asymptote of the line $\{te^{i\theta}:t\in[0,+\infty), \theta=\pi\mu([t_c,b])\}$ since $\int_{\R}\arg(w-t)d\nu(t)=\nu(\R)\arg(w)+O(\vert w\vert^{-1})$ as $\vert w\vert\to \infty$.  Now assume that the decent path ends at $t_e^-$ and that the ascent path ends at $t_e^+$. Again by forming a closed contour containing the interval $[t_e^-,t_e^+]$ and exploiting the harmonicity of $\text{Im}[f(w)]$, we get a contradiction. Thus as before, contour will become an asymptote of the line $\{te^{i\theta}:t\in[0,+\infty), \theta=\pi\mu([t_c,b])\}$. 

%Finally, we immediately see that the contour $\G_n^1$ in Proposition \ref{IntKernel3} homotopy equivalent to the contours $\mathcal{A}$ on $\hat{\C}\backslash \supp(\nu_{n,2}^+)$, where $\hat{\C}$ is the Riemann sphere, and the contour $\g_n^1$ Proposition \ref{IntKernel3} is homotopy equivalent to $\mathcal{D}$ on $\C\backslash \supp(\nu_{n,1}^-)$. This implies that we can deform $\G_n^1$ to $\mathcal{A}$ and $\g_n^1$ to $\mathcal{D}$ in $\hat{\C}\backslash (\supp(\nu_{n,2}^+)\cup\supp(\nu_{n,1}^-)\cup B(t_c,\delta_c)\cup B(t_e,\delta_e))$ for some $\delta_c,\delta_e>0$ sufficiently small. 
\end{proof}

%\begin{rem}
%Due to the uniform convergence of $g_{n,i}\to f$ on compact subsets of $\Hp$, the steepest ascent/descent contours of $f$ are also ascent/descent contours of $g_{n,i}$ for $i=1,2$. This is shown in Lemma \ref{lemGlobalContourfn} in the Appendix. 
%\end{rem}

%\begin{rem}
%In general, for a $\mu\in \mathcal{M}_{c,1}^{\lambda}(\R)$, it is unclear whether $\lim_{t\to T}P_v(t)\rho(u(t))$ and $\lim_{t\to T}H_v(t)\rho(u(t))$ exists for a monotone parametrization $\gamma(t)=u(t)+iv(t)$ of the steepest descent path of $f(w)$ in a neighborhood of $t_e$, such that $\lim_{t\to T}\gamma(t)=t_e$. It is therefore unclear if $\mathcal{H}^1(B(t_e,\delta)\bigcap\g_d)<+\infty$. However, if $\mu$ is the solution of a variational problem as in (\ref{VarProb}) in Remark \ref{remDOPE}, then the density $\rho$ of $\mu$ in a vicinity of $t_e$ is locally H\"older continuous. This implies that the Hilbert transform $\mathcal{H}\rho$ is also locally H\"older continuous in a neighborhood of $t_e$. In turn, this implies that the steepest descent curve has an extension as a $C^1$-curve, containing the end point $t_e$, which implies that $\mathcal{H}^1(B(t_e,\delta)\bigcap\g_d)<+\infty$. We will bypass this problem in the argument below by slightly modifying the steepest descent contour close to $t_e$ in such a way that it is still a descent
%contour.
%\end{rem}

\subsection{Estimates and localization}\label{sec:estimates}

We start with some preliminary results that we will need. By lemma \ref{UniformConv} and Assumption \ref{A4}, if we take $\delta_0$ small enough, then $|g_{n,i}(w)|\le C$
for $|w-t_c|\le \delta_0$, where $C$ is a constant. We have the Taylor expansion
\begin{align}
\label{LocalGfunc}
g_{n,i}(z)=g_{n,i}(t_c)+g_{n,i}'(t_c)(z-t_c)+\frac{1}{2}g_{n,i}''(t_c)^2(z-t_c)^2+\frac{1}{6}g_{n,i}'''(t_c)^3(z-t_c)^3+r_{n,i}(z)(z-t_c)^4,
\end{align}
where
\begin{align}
\label{LocalGerror}
r_{n,i}(z)=\frac{1}{2\pi i}\int_{\vert w-t_c\vert=\delta_3}\frac{g_{n,i}(w)}{(w-t_c)^5(1-\frac{z-t_c}{w-t_c})}dw.
\end{align}
From (\ref{LocalGerror}) we see that, if we take $\delta_1\leq \delta_0/2$, then there is a constant $C(\delta_0)$ so that
\begin{align}
\label{LocalAbsError}
\vert r_{n,i}(z)\vert \leq C(\delta_0),
\end{align}
for all $z\in B(t_c,\delta_1)$.  Consider a curve
\begin{equation}\label{Zedcurve}
z(t)=t_c+\zeta(t),
\end{equation}
$t\in I$, such that $|\zeta(t)|\le Ct\le\delta_1$ for all $t\in I$, $I$ an interval.
From (\ref{LocalGfunc}) we obtain

\begin{align}
\label{RealTaylorExp}
\text{Re}[g_{n,i}(z(t))-g_{n,i}(t_c)]&=g_{n,i}'(t_c)\text{Re}[\zeta(t)]+\frac{1}{2}g_{n,i}''(t_c)\text{Re}[\zeta(t)^2]+\frac{1}{6}g_{n,i}'''(t_c)\text{Re}[\zeta(t)^3]&\\&+\text{Re}[r_{n,i}(z(t))\zeta(t)^4].\nonumber
\end{align}
Note that, by the assumption on $\zeta(t)$,

\begin{align}
\label{RealPart}
\vert \text{Re}[\zeta(t)]\vert\leq Ct,\quad \vert \text{Re}[\zeta(t)^2]\vert\leq Ct^2,\quad \vert \text{Re}[r_{n,i}(z(t))\zeta(t)^4]\vert\leq Ct^4,
\end{align}
for all $t\in I$, where the constant $C$ is independent of $\delta_1$. 

By the definition (\ref{FuncNonAsymp2}), we have that
\begin{align*}
h_{n,i}(z)=\sum_{k=\Delta x_i^{(n)}}^{-1}\log\bigg(z-\bigg(\frac{x_c^{(n)}}{n}+\frac{k}{n}\bigg)\bigg)
\end{align*}
if $\Delta x_i^{(n)}<0$, $h_{n,i}(z)=0$ if $\Delta x_i^{(n)}=0$ and 
\begin{align*}
h_{n,i}(z)=-\sum_{k=1}^{\Delta x_i^{(n)}}\log\bigg(z-\bigg(\frac{x_c^{(n)}}{n}+\frac{k}{n}\bigg)\bigg),
\end{align*}
if $\Delta x_i^{(n)}>0$. Write $\kappa_i=\text{sgn}(\Delta x_i^{(n)})$ $(=0$ if $\Delta x_i^{(n)}=0)$. Then,

\begin{align}
\label{hDerivaive}
h_{n,i}'(z)=-\kappa_i\sum_{k=1}^{\vert \Delta x_i^{(n)}\vert}\frac{1}{z-\big(\frac{x_c^{(n)}}{n}+\kappa_i\frac{k}{n}\big)}.
\end{align}
From this, and (\ref{ScalingDeltaX}), it follows that, if $\vert z-t_c\vert\leq \delta_1$, with $\delta_1$ small enough, then
\begin{align}
\label{hDerivaiveIneq}
\vert h_{n,i}'(z)\vert=Cn^{1/3}, \quad \vert h_{n,i}''(z)\vert\leq Cn^{1/3}.
\end{align} 

Similarly to (\ref{LocalGfunc}), we get 

\begin{align}
\label{hTaylorExp}
 h_{n,i}(z(t))= h_{n,i}(t_c)+h_{n,i}'(t_c)\zeta(t)+s_{n,i}(z(t))\zeta(t)^2,
\end{align} 
where
\begin{align}
\label{hError}
\vert s_{n,i}(z(t))\vert\leq Cn^{1/3},
\end{align} 
and $z(t)$ is given by (\ref{Zedcurve}).

We will now discuss the localization of the asymptotic analysis of the kernel to a neighbourhood of $t_c$. Let $\delta_1>0$ and let $B_2=B(t_c,\delta_1)$ be a (small)
ball around $t_c$. The descent contour $\mathcal{D}$ from Lemma \ref{GlobalDescentAsympFunc} intersects $\dv B_2\cap \Hp$ at the point $D_2$, se figure 
\ref{figLocalContourAsympfn}

\begin{figure}[H]

\centering
\begin{tikzpicture}
\draw (-4,0) --(4,0);
\filldraw (0,0) circle (0.4mm);
\draw (0,-0.2) node {$t_c$};
\draw (0,0) circle (1cm);
\draw (1.2,-0.6) node {$B_2$};
\draw (0,0) to [out=80,in=210] (0.574,0.812); 
\filldraw (0.574,0.812) circle (0.4mm); 
\draw[->] (0.574,0.812) to [out=30,in=180] (1.8,1.2); 
\draw (1.8,1.2) to [out=0,in=180] (2.1,1.2); 
\draw (1.8,1.5) node {$\mathcal{D}$};
\draw (0.574,1.1) node {$D_2$}; 
\end{tikzpicture}
\caption{\protect\label{figLocalContourAsympfn} The global descent path $\mathcal{D}$ close to $t_c$}

\end{figure}
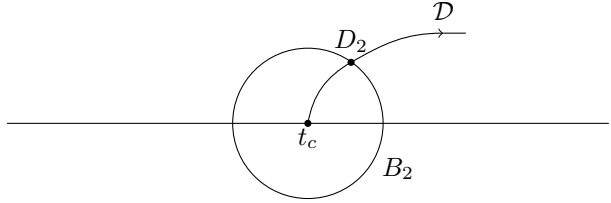

We now formulate a lemma that will allow us to neglect the contribution from $\mathcal{D}$ outside $B_2$.
\begin{Lem}
\label{lemDisc2}
If we choose $\delta_1$ sufficiently small, there is a constant $b_0(\delta_1)>0$ such that for
$n$ large enough
\begin{equation}
\label{IneqDisc1RealPartg}
\text{Re\,}[g_{n,i}(D_2)-g_{n,i}(t_c)]\leq -b_0(\delta_1).
\end{equation}
\end{Lem}

\begin{proof}
We take $\zeta(t)=te^{i\theta(\delta_1)}$, $0\le t\le\delta_1$, where $\theta(\delta_1)$ is chosen so that $\zeta(\delta_1)=D_2$. It follows
from Lemma \ref{UniformConv}, $f'(t_c)=f''(t_c)=0$, and $f'''(t_c)>0$ that, given $\eps_1>0$, we have
\begin{align}
\label{GfuncDerivativeEst}
\vert g_{n,i}'(t_c)\vert \leq \eps_1,\quad \vert g_{n,i}''(t_c)\vert \leq \eps_1 
\end{align}
for large $n$, and there is a $c_1>0$ so that
\begin{align}
\label{GfuncDerivative3}
g_{n,i}'''(t_c)\geq c_1 
\end{align}
for all sufficiently large $n$. Also, since $f'(t_c)=f''(t_c)=0$ and $f'''(t_c)>0$, a local Taylor expansion of $f$ shows that 
$\theta(\delta_1)\to\pi/3$ as $\delta_1\to 0$. Consequently,
\begin{equation}
\label{zedcube}
\text{Re\,}\zeta(t)^3=t^3\cos 3\theta(\delta_1)\le -\frac 12 t^3
\end{equation}
for $0\le t\le\delta_1$ if $\delta_1$ is sufficiently small. From (\ref{RealTaylorExp}), (\ref{RealPart}), (\ref{GfuncDerivativeEst}), and
(\ref{GfuncDerivative3}) we see that
 \begin{align}
\label{GfuncEndpoint}
 \text{Re\,}[g_{n,i}(D_2)- g_{n,i}(t_c)]\geq -\frac{1}{12}c_1\delta_1^3+C\eps_1(\delta_1+\delta_1^2)+C\delta_1^4\nonumber\\
 = c_1\delta_1^3\bigg[-\frac{1}{12}-C'\eps_1(\delta_1^{-2}+\delta_1^{-1})+C'\delta_1\bigg].
\end{align}
We can now choose $\delta_1$ so that $C'\delta_1\le1/48$ and then $\epsilon_1$ so that $C'\eps_1(\delta_1^{-2}+\delta_1^{-1})\le 1/48$. We then get 
(\ref{IneqDisc1RealPartg}) with $b_0(\delta_1)=c_1\delta_1^3/24$ and the lemma is proved.
\end{proof}

Let $B_3=B(t_e,\delta_2)$ be a small ball around $t_e$. The descent contour $\mathcal{D}$ intersects $\dv B_3\cap \Hp$ at the point $D_3$. Let $\mathcal{C}$ be as in figure \ref{figLocalContourEndfn} and $\mathcal{D}'$ be the part of $\mathcal{D}$ outside $B_2$ and $B_3$, so that  $\mathcal{D}'$ lies strictly in $\Hp$. Hence, from 
Lemma \ref{UniformConv}, it follows that $g_{n,i}(z)\to f(z)$ uniformly on $\mathcal{D}'$ as $n\to\infty$.
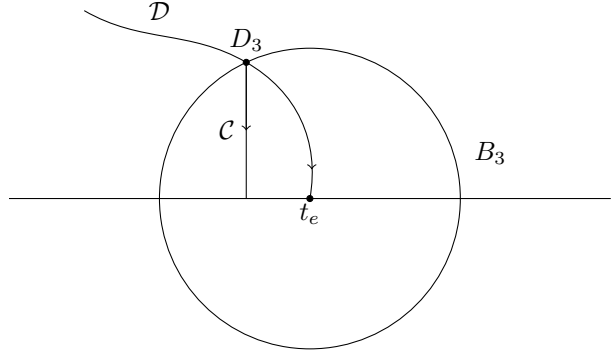
\begin{figure}[H]

\centering
\begin{tikzpicture}
\draw (-4,0) --(4,0);
\filldraw (0,0) circle (0.4mm);
\draw (0,-0.2) node {$t_e$};
\draw (0,0) circle (2cm);
\draw (-2,2.5) node {$\mathcal{D}$};
\draw (0,0) to [out=80,in=330] (-{2*0.423},{2*0.906}); 
\draw (-{2*0.423},{2*0.906}) to [out=150,in=330] (-3,2.5); 
\draw[->] (0.025,0.4) --(0.025,0.39);
\draw (-{2*0.423},{2*0.906}) -- (-{2*0.423},0); 
\draw[->] (-{2*0.423},{2*0.906}) -- (-{2*0.423},{0.906}); 
\draw(-1.1,{0.906}) node {$\mathcal{C}$}; 
\filldraw (-{2*0.423},{2*0.906}) circle (0.4mm); 
\draw (2.4,0.6) node {$B_3$};
\draw (-{2*0.423},{0.3+2*0.906}) node {$D_3$}; 
\end{tikzpicture}
\caption{\protect\label{figLocalContourEndfn}The global descent path $\mathcal{D}$ and $\mathcal{C}$ close to $t_e$}

\end{figure}

The next lemma gives the estimate we need on $\mathcal{C}$.
\begin{Lem}
\label{lemDisc3}
There is a constant $b_1(\delta_2)$ such that $b_1(\delta_2)\to 0$ as $\delta_2\to 0$ and
\begin{align}
\label{IneqDisc3}
\text{Re\,}[g_{n,i}(z)-g_{n,i}(D_3)]\leq b_1(\delta_2)
\end{align}
for all $z\in \mathcal{C}$ if $n$ is sufficiently large.
\end{Lem}

\begin{proof}
Let $D_3=x_0+iy_0$ and set 
\begin{align*}
z(t)=x_0+i(y_0-t), \quad0\leq t \leq y_0\leq \delta_2.
\end{align*}

We have chosen $\delta_2$ so small that $x_0>\chi_c$. From (\ref{FuncNonAsymp1}) we see that 
\begin{align*}
\frac{d}{dt}\text{Re}[f_n(z(t))]&=\frac{d}{dt}\bigg(\frac{1}{n}\sum_{\beta_i^{(n)}\leq t_2^{(n)}}\log\vert z(t)-n^{-1}\beta_i^{(n)}\vert+\frac{1}{n}\sum_{\beta_i^{(n)}>t_1^{(n)}}\log\vert z(t)-n^{-1}\beta_i^{(n)}\vert\bigg)\\&-\frac{d}{dt}\frac{1}{n}\sum_{k=t_1^{(n)}+1}^{x_c^{(n)}}\log\vert z(t)-n^{-1}k\vert=u_1^{(n)}(t)+u_2^{(n)}(t). 
\end{align*}
Note that $u_1^{(n)}(t)\leq 0$ and 

\begin{align*}
u_2^{(n)}(t)&=-\frac{d}{dt}\frac{1}{2n}\sum_{k=t_1^{(n)}+1}^{x_c^{(n)}}\log\bigg(\bigg(x_0-\frac{k}{n}\bigg)^2-(y_0-t)^2\bigg)\\
&=\frac{1}{n}\sum_{k=t_1^{(n)}+1}^{x_c^{(n)}}\frac{y_0-t}{\big(x_0-\frac{k}{n}\big)^2+(y_0-t)^2}\leq \frac{\delta_2}{n}\sum_{k=t_1^{(n)}+1}^{x_c^{(n)}}\frac{1}{\big(x_0-\frac{k}{n}\big)^2}\leq C\delta_2,
\end{align*}
since $x_c^{(n)}/n\to\chi_c<x_0$ as $n\to \infty$. Thus, 
\begin{equation*}
\text{Re\,}[f_n(z)-f_n(D_3)]\le C\delta_2
\end{equation*}
for all $z\in\mathcal{C}$ if $n$ is sufficiently large. 
It remains to estimate  $\frac{1}{n}\text{Re}[h_{n,i}(z(t))]$. Consider $h_{n,1}$ and assume $\Delta x_1^{(n)}>0$, the other cases are similar. From
(\ref{FuncNonAsymp2}) and (\ref{FuncCont}) we obtain
\begin{align*}
\frac{1}{n}\bigg\vert\text{Re\,}[h_{n,1}(z(t))] \bigg\vert \leq \frac{\log n}{n}\Delta x_1^{(n)}+\frac 1{2n}\sum_{k=nx_c^{(n)}+1}^{nx_c^{(n)}+\Delta x_1^{(n)}}
\log\bigg\vert\bigg(x_0-\frac kn\bigg)^2+(y_0-t)^2\bigg\vert\le \frac{C\log n}{n^{2/3}}\le \delta_2
\end{align*}
if $n$ is sufficiently large. Here we used again the fact that $x_c^{(n)}/n\to\chi_c<x_0$. This proves the lemma.
\end{proof}

From Proposition \ref{IntKernel3} and (\ref{Conjugationfactor}) we see that

\begin{align}
\label{CIAsym}
&\frac{p_n(x_2,y_2)}{p_n(x_1,y_1)}\frac{c_0n^{1/3}}{2}\tilde{K}_{\mathcal{R}}^{(n)}((x_1,y_1),(x_2,y_2))\notag
\\&=\frac{1}{(2\pi i)^2}\oint_{\G_n^1+\G_n^2}dz \oint_{\g_n^1+\g_n^2}dw\frac{c_0n^{1/3}}{z-n^{-1}x_2}\frac{(d_0n^{2/3})^{r}q(nw;r)}{(d_0n^{2/3})^{s}q(nz;s)}
\frac{e^{n(g_{n,1}(w)-g_{n,1}(t_c))-n(g_{n,2}(z)-g_{n,2}(t_c))}}{w-z}.
\end{align}
Note also that
\begin{equation}\label{qnAsym}
(d_0n^{2/3})^{r}q_n(nw;r)=\prod_{k=0}^{r-1}\left(\frac{n^{1/3}(w-t_c)-k/n^{1/3}}{d_0}\right)
\end{equation}
if $r>0$; there is an analogous formula if $r<0$ and if $r=0$ the expression is $=1$.

By uniform convergence we have that $\vert g_{n,i}(z)-f(z)\vert\leq \frac{1}{4}b_0(\delta_1)$ for all $z\in \mathcal{D}'$ if $n$ is large enough.
Thus, if $z\in \mathcal{D}'$ we have 

\begin{align*}
\text{Re}[g_{n,i}(z)-g_{n,i}(D_2)]\leq \text{Re}[g_{n,i}(z)-f(z)]+ \text{Re}[f(D_2)-g_{n,i}(D_2)]+ \underbrace{\text{Re}[f(z)-f(D_2)]}_{\leq0}
\leq \frac{1}{2}b_0(\delta_1),
\end{align*}
where we have used that $\mathcal{D}'$ is a descent curve. Combining this with Lemma \ref{lemDisc2}, we obtain
\begin{align}\label{gniest}
\text{Re}[g_{n,i}(z)-g_{n,i}(t_c)]&\leq -\frac{1}{2}b_{0}(\delta_1).
\end{align}
for all $z\in \mathcal{D}'$. Given $\delta_1$ we choose $\delta_2$ so small that $b_1(\delta_2)\le b_0(\delta_1)/4$. Since $D_3\in\mathcal{D}'$ we find,
using Lemma \ref{lemDisc3} and (\ref{gniest}) that
\begin{align}\label{gniest2}
\text{Re}[g_{n,i}(z)-g_{n,i}(t_c)]&\leq -\frac{1}{4}b_{0}(\delta_1).
\end{align}
for all $z\in\mathcal{C}$ if $n$ is sufficiently large.

Note that
\begin{equation}\label{Infest}
\text{Re\,}g_{n,2}(z)=\nu_{n,2}(\R)\log|z|+O(|z|^{-1})
\end{equation}
as $|z|\to\infty$, and $\nu_{n,2}(\R)\to\nu(\R)>0$ as $n\to\infty$. In (\ref{CIAsym}) we can let $\G_n^2$ and $\g_n^2$ be small circles around $t_c$ inside
$B_2$ and deform $\g_n^1$ to the descent contour $\mathcal{D}'+\mathcal{C}$ (and its reflection image in the lower half plane).  Using (\ref{Infest}) we see that we can
deform $\G_n^1$ to the ascent contour $\mathcal{A}$. We can now use the estimates (\ref{gniest}), (\ref{gniest2}) and (\ref{Infest}) to see that in the integral
(\ref{CIAsym}) we can ignore $\mathcal{D}'+\mathcal{C}$ and the part of $\mathcal{A}$ outside $B_2$ in the limit. More precisely, we also have to
combine this estimate with the estimates and computations inside $B_2$ that we will do in the next section. We leave out the details.

\subsection{Local Analysis}\label{sec:local}

Let $\{M_n\}_{n\ge 1}$ be a sequence satisfying
\begin{equation}\label{MnAss}
M_n\to\infty,\quad \frac{M_n}{n^{1/12}}\to 0,\quad n^{2/3}M_n|f_n'(t_c)|\to 0
\end{equation}
as $n\to\infty$, which exists by Assumption \ref{A4}. Consider the ball $B_1=B(t_c,M_n n^{-1/3})$. Again we will only discuss the descent contour; the ascent case is analogous.
As above, we take $\zeta(t)=te^{i\theta(\delta_1)}$ in (\ref{Zedcurve}) with $M_n/n^{1/3}\le t\le \delta_1$, so that $z(\delta_1)=t_c+\zeta(\delta_1)=D_2$, 
see figure \ref{figLocalContourB2}.

\begin{figure}[H]

\centering
\begin{tikzpicture}[xscale=0.5,yscale=0.5]
\draw (-5,0) --(5,0);
\filldraw[white] (0,0) circle (1cm);
\draw (0,0) circle (1cm);
\draw (0,0) circle (4cm);
\filldraw (0,0) circle (0.4mm);
\filldraw (0,0.4) node {\small$t_c$};
\filldraw ({cos(55)},{sin(55)}) circle (0.4mm);
\filldraw ({cos(55)},{-sin(55)}) circle (0.4mm);
\filldraw ({cos(125)},{sin(125)}) circle (0.4mm);
\filldraw ({cos(235)},{sin(235)}) circle (0.4mm);
\filldraw ({4*cos(60)},{4*sin(60)}) circle (0.4mm);
\filldraw ({4*cos(60)},{-4*sin(60)}) circle (0.4mm);
\filldraw ({4*cos(120)},{4*sin(120)}) circle (0.4mm);
\filldraw ({4*cos(240)},{4*sin(240)}) circle (0.4mm);
\draw[thick] ({cos(55)},{sin(55)}) -- ({4*cos(60)},{4*sin(60)});
\draw[thick,->] ({1.97*cos(59)},{1.97*sin(59)}) -- ({2*cos(59)},{2*sin(59)});
\draw[thick]  ({cos(55)},{-sin(55)}) -- ({4*cos(60)},{-4*sin(60)});
%\draw[thick,->]  ({cos(55)},{-sin(55)}) -- ({2*cos(60)},{-2*sin(60)});
\draw[thick]  ({cos(125)},{sin(125)}) -- ({4*cos(120)},{4*sin(120)}); 
\draw[thick,->]  ({1.97*cos(121)},{1.97*sin(121)}) -- ({2*cos(121)},{2*sin(121)}); 
\draw[thick]  ({cos(235)},{sin(235)}) --({4*cos(240)},{4*sin(240)});
\draw ({4.5*cos(30)},{-4.5*sin(30)}) node {$B_2$};
\draw ({1.5*cos(30)},{-1.5*sin(30)}) node {$B_1$};
\draw ({4.4*cos(60)+0.8},{4.4*sin(60)}) node {$z(\delta_1)=D_2$};

\end{tikzpicture}
\caption{\label{figLocalContourB2} Local contours in the region $B_2\backslash B_1$}
\end{figure}
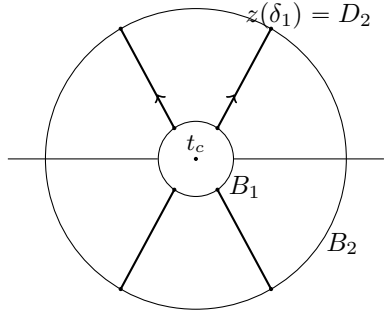

It follows from (\ref{RealTaylorExp}),(\ref{RealPart}) and (\ref{GfuncDerivative3}) that there is a constant $C$ so that
\begin{align}
\label{RealPartGfuncEst}
\text{Re}[g_{n,i}(z(t))-g_{n,i}(t_c)]&\leq -t^3\bigg[\frac{1}{12}c_1-\frac{C\vert g_{n,i}'(t_c)\vert}{t^2}-\frac{C\vert g_{n,i}''(t_c)\vert}{t}-Ct\bigg].
\end{align} 
It follows from Assumption \ref{A4} that 
$n^{2/3}f_n'(t_c)\to 0$ and $n^{2/3}f_n''(t_c)\to 0$
as $n\to \infty$. Combining this with (\ref{hDerivaiveIneq}) we see that, for $t\in[M_n/n^{1/3},\delta_1]$,
\begin{align}
\label{gDerivativesConv}
\frac{\vert g_{n,i}'(t_c)\vert}{t^2}\leq C\frac{n^{2/3}}{M_n^{2}}(\vert f_n'(t_c)\vert+Cn^{-2/3})
\end{align}  
and
\begin{align}
\label{gDerivativesConv2}
\frac{\vert g_{n,i}''(t_c)\vert}{t}\leq C\frac{n^{2/3}}{M_n}(\vert f_n''(t_c)\vert+Cn^{-2/3}).
\end{align}  
Using (\ref{MnAss}) we see that the right hand sides of (\ref{gDerivativesConv}) and (\ref{gDerivativesConv2}) are $\le c_1/72$ if $n$ is large enough. Also, we can assume that
$\delta_1$ has been chosen so small that $Ct$ in (\ref{RealPartGfuncEst}) satisfies $Ct\le C\delta_1\le c_1/72$. It then follows from (\ref{RealPartGfuncEst}) that
\begin{align}
\label{RealPartgFunc1}
\text{Re\,}[g_{n,i}(z(t))-g_{n,i}(t_c)]\leq -\frac{1}{24}c_1t^3
\end{align}  
for $t\in [M_n/n^{1/3},\delta_1]$. If we let
\begin{equation*}
z(t)=t_c+d_0n^{-1/3}te^{i\theta(\delta_1)},
\end{equation*}
$t\in [M_n,\delta_1n^{1/3}]$, then (\ref{RealPartgFunc1}) gives the estimate
\begin{equation}\label{B2B1est}
n\text{Re\,}[g_{n,i}(t_c+d_0n^{-1/3}te^{i\theta(\delta_1)})-g_{n,i}(t_c)]\le -\frac{d_0^3c_1}{24}t^3.
\end{equation}
This estimate can be used together with the corresponding estimate for the ascent contour to control, in the limit, the contribution from the parts of the descent and ascent
contours that lie in $B_2\backslash B_1$. We find that we can neglect the contributions from $B_2\backslash B_1$.

We now define the local contours that will be used in $B_1$. Let $D_1=t_c+M_nn^{-1/3}e^{i\theta(\delta_1)}$. Fix $x>0$ and $y<0$. We define $\overline{\g}_n^1$ in the upper
half plane by
\begin{equation}\label{zt}
z(t)=t_c+d_0n^{-1/3}u_n(t):=t_c+d_0n^{-1/3}(x+te^{i\theta_n(\delta_1)}),
\end{equation}
where $0\le t\le M_n'$. Here $\theta_n(\delta_1)$ and $M_n'$ are such that $z(M_n')=D_1$. Note that $M_n'/M_n\to 1$ as $n\to\infty$. I the lower half plane
we just take the mirror image. We define $\overline{\G}_n^1$ analogously corresponding to the ascent contour, see figure \ref{figLocalContourB_1}. Also,
we define $\overline{\G}_n^2$ by $z(t)=t_c+d_0n^{-1/3}r_1e^{it}$, $0\le t\le 2\pi$ and $\overline{\g}_n^2$ by $z(t)=t_c+d_0n^{-1/3}r_2e^{-it}$, $0\le t\le 2\pi$,
where $0<r_2<r_1<\min(x,-y)$, see figure \ref{figLocalContourB_1}.

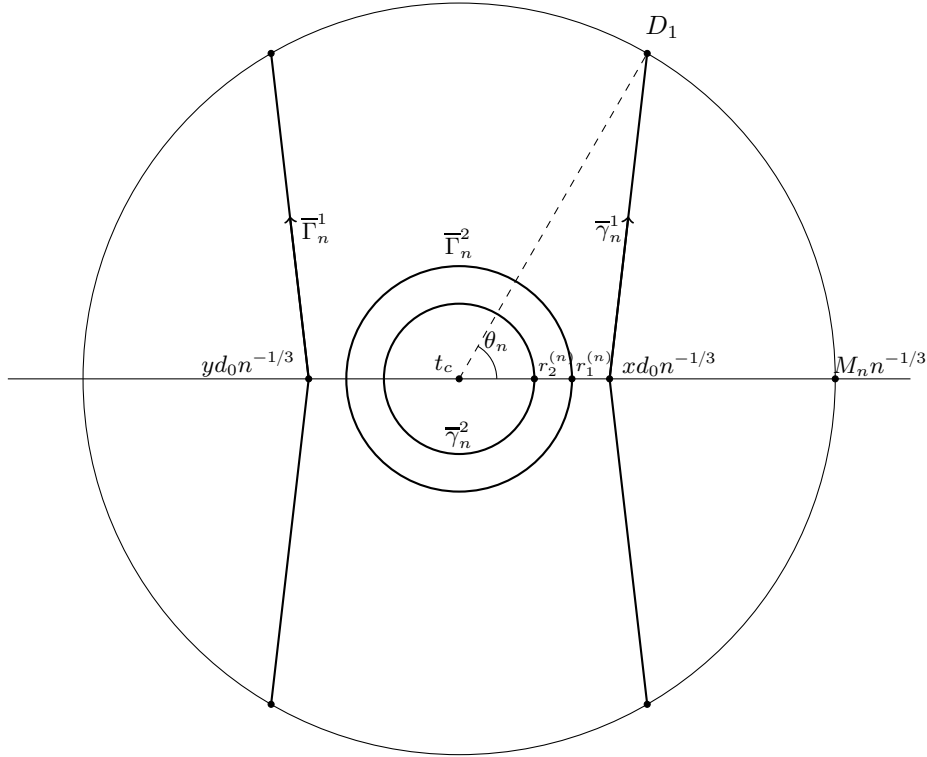
\begin{figure}[H]

\centering
\begin{tikzpicture}
\draw (-6,0) --(6,0);
\draw (0,0) circle (5cm);
\draw[thick] (0,0) circle (1cm);
\draw[thick] (0,0) circle (1.5cm);
\filldraw (0,0) circle (0.4mm);
\filldraw (-2,0) circle (0.4mm);
\filldraw (2,0) circle (0.4mm);
\filldraw (1,0) circle (0.4mm);
\filldraw (1.5,0) circle (0.4mm);
\filldraw (5,0) circle (0.4mm);
\filldraw ({5*cos(60)},{5*sin(60)}) circle (0.4mm);
\filldraw ({5*cos(120)},{5*sin(120)}) circle (0.4mm);
\filldraw ({5*cos(240)},{5*sin(240)}) circle (0.4mm);
\filldraw ({5*cos(300)},{5*sin(300)}) circle (0.4mm);
\draw[thick]  (-2,0) -- ({5*cos(120)},{5*sin(120)});
\draw[thick,->]  (2,0) -- ({1+2.5*cos(60)},{2.5*sin(60)});
\draw[thick,->]  (-2,0) -- ({-1-2.5*cos(60)},{2.5*sin(60)});
\draw[thick]  (-2,0) -- ({5*cos(240)},{5*sin(240)});
\draw[thick]  (2,0) -- ({5*cos(60)},{5*sin(60)});
\draw[thick]  (2,0) -- ({5*cos(300)},{5*sin(300)});
\draw (1.3,0.2) node {\scriptsize$r_2^{(n)}$};
\draw (1.8,0.2) node {\scriptsize$r_1^{(n)}$};
\draw (2.8,0.2) node {\small$xd_0n^{-1/3}$};
\draw (-2.8,0.2) node {\small$yd_0n^{-1/3}$};
\draw (5.6,0.2) node {\small$M_nn^{-1/3}$};
\draw (-0.2,0.2) node {\small$t_c$};
\draw (0,-0.75) node {\small$\overline{\g}_n^2$};
\draw (0,1.8) node {\small$\overline{\G}_n^2$};
\draw (2,2) node {\small$\overline{\g}_n^1$};
\draw (-1.9,2) node {\small$\overline{\G}_n^1$};
\draw[dashed] (0,0) -- ({5*cos(60)},{5*sin(60)});
\draw (0.5,0) arc (0:60:0.5cm);
\draw (0.5,0.5) node {\small$\theta_n$};
\draw ({5.4*cos(60)},{5.4*sin(60)}) node {$D_1$};
\end{tikzpicture}
\caption{\protect\label{figLocalContourB_1} Local contours in $B_1$. Here $r_i^{(n)}=n^{-1/3}r_i$ for $i=1,2$.}

\end{figure}

Write
\begin{align}
\label{Lkernel}
&\tilde{L}_{\mathcal{R}}^{(n)}((x_1,y_1),(x_2,y_2))\nonumber\\
&=\frac{1}{(2\pi i)^2}\oint_{-\overline{\G}_n^1-\overline{\G}_n^2}dz \oint_{\g_n^1+\g_n^2}dw\frac{c_0n^{1/3}}{z-n^{-1}x_2}\frac{(d_0n^{2/3})^{r}q(nw;r)}{(d_0n^{2/3})^{s}q(nz;s)}
\frac{e^{n(g_{n,1}(w)-g_{n,1}(t_c))-n(g_{n,2}(z)-g_{n,2}(t_c))}}{w-z}.
\end{align}
From (\ref{CIAsym}) and the discussion above it follows that
\begin{equation*}
\lim_{n\to\infty}\frac{p_n(x_2,x_1)}{p_n(x_1,y_1)}\frac{c_0n^{1/3}}{2}\tilde{L}_{\mathcal{R}}^{(n)}((x_1,y_1),(x_2,y_2))
=\lim_{n\to\infty}\tilde{K}_{\mathcal{R}}^{(n)}((x_1,y_1),(x_2,y_2)).
\end{equation*}

Consider the contribution from $\overline{\g}_n^1$. Write
\begin{align*}
n(g_{n,1}(w)-g_{n,1}(t_c))=n(f_n(w)-f_n(t_c))+h_{n,1}(w)-h_{n,1}(t_c).
\end{align*}

In analogy with (\ref{LocalGfunc}) we have, with $z(t)$ given by ({\ref{zt}),
\begin{align}
\label{FTaylorExpand}
n(f_{n}(z(t))-f(t_c))=n^{2/3}f_n'(t_c)d_0u_n(t)+\frac{1}{2}n^{1/3}f_n''(t_c)d_0^2u_n(t)^2+\frac{1}{6}f_n'''(t_c)d_0^3u_n(t)^3+n^{-1/3}\tilde{r}_n(z(t))d_0^4u_n(t)^4,
\end{align}
where $\vert \tilde{r}_n(z(t))\vert\leq C$ for $0\le t\le M_n'$. Since $\vert u_n(t)\vert\leq CM_n$,
\begin{align}
\label{RfuncEst}
\vert n^{-1/3}\tilde{r}_n(z(t))d_0^4u_n(t)^4\vert \le Cn^{-1/3}M_n^4\to 0,
\end{align}
as $n\to \infty$ for $0\leq t\leq M_n,$ by (\ref{MnAss}). Also,
\begin{align}
\label{FderivativeTaylorExpand}
&\vert n^{2/3}f_n'(t_c)d_0u_n(t)\vert \leq CM_nn^{2/3}\vert f_n'(t_c)\vert\\
&\vert n^{1/3}f_n''(t_c)d_0^2u_n(t)^2\vert \leq CM_n^2n^{1/3}\vert f_n''(t_c)\vert.\nonumber
\end{align}
By (\ref{MnAss}) and Assumption \ref{A4} it follows that the right hand sides of (\ref{FderivativeTaylorExpand}) go to $0$ as $n\to\infty$. Furthermore it follows
from Assumption \ref{A4} that $n^{2/3}\vert f_n'''(t_c)-f'''(t_c) \vert\to 0$ as $n\to \infty$. Hence, we have shown that
\begin{align}
\label{fConv2}
n(f_n(t_c)-f(t_c))=\frac{d_0^3}{6}f'''(t_c)u_n(t)^3+o(1)=\frac 13u_n(t)^3+o(1)
\end{align}
uniformly for $t\in[0,M_n']$ as $n\to \infty$, by our choice (\ref{DefD0}).
From the definition of $\theta_n(\delta_1)$ it follows that $n^{1/3}\vert \theta_n(\delta_1)-\theta(\delta_1)\vert$ is bounded, and thus $M_n^{3}\vert \theta_n(\delta_1)-\theta(\delta_1)\vert\to 0$ as $n\to \infty$ by (\ref{MnAss}). From this we see that
\begin{align}
\label{fConv3}
n(f_n(z(t))-f(t_c))=\frac{1}{3}(x+te^{i\theta(\delta_1)})^3+o(1),
\end{align}
uniformly for $t\in[0,M_n']$ as $n\to \infty$. From (\ref{hTaylorExp}) we obtain
\begin{align}
\label{hConvExp}
h_{n,1}(z(t))-h_{n,1}(t_c)=h_{n,1}'(t_c)d_0n^{-1/3}u_n(t)+s_{n,1}(z(t))d_0^2n^{-2/3}u_n(t)^2,
\end{align}
and by (\ref{hError}),
\begin{align}
\label{hErrorEst}
\vert n^{-2/3}d_0^2s_{n,1}(z(t))u_n(t)^2\vert\leq Cn^{-1/3}M_n^2\to 0,
\end{align}
uniformly for $0\leq t\leq M_n'$ as $n\to \infty$ by (\ref{MnAss}). Recall (\ref{hDerivaive}) and consider the case when $\Delta x_1^{(n)}>0$. From (\ref{hDerivaive}) we see that
\begin{align*}
\left| n^{-1/3}h_{n,1}'(t_c)+\frac{\Delta x_1^{(n)}}{(t_c-\chi_c)n^{1/3}}\right|&\le \frac{1}{n^{1/3}}\left|\sum_{k=1}^{\Delta x_1^{(n)}}-\frac{1}{t_c-\frac{x_c^{(n)}}{n}-\frac kn}+
\frac{1}{t_c-\chi_c}\right| \\
&\le\frac{C}{n^{1/3}}\sum_{k=1}^{\Delta x_1^{(n)}}\left|\frac{\frac kn+\frac{x_c^{(n)}}{n}-\chi_c}{t_c-\frac{t_c-x_c^{(n)}}{n}-\frac kn}\right|\le\frac C{n^{2/3}}
\end{align*}
since $|\Delta x_1^{(n)}|\le |r/2-c_0n^{1/3}\xi_n|\le Cn^{1/3}$, $|x_c^{(n)}/n-\chi_c|\le 1/n$, $\xi_n\to\xi$ as $n\to\infty$ and $\chi_c>t_c$. It follows that
\begin{align}
\label{hFuncConv}
n^{-1/3}h_{n,1}'(t_c)-\frac{c_0\xi}{2(t_c-\chi_c)} \to 0
\end{align}
as $n\to\infty$.
Using this and the control on $\theta_n(\delta_1)$, it follows from (\ref{hConvExp}), (\ref{hErrorEst}) and (\ref{hFuncConv}) that
\begin{align}
\label{hDifference}
h_{n,1}(z(t))-h_{n,1}(t_c)=-\frac{c_0d_0\xi}{2(t_c-\chi_c)}(x+te^{i\theta(\delta_1)})+o(1),
\end{align}
uniformly for $0\le t\le M_n$ as $n\to \infty$. We have chosen $c_0$, (\ref{DefC0}), so that
\begin{align}
\label{c0Constant}
\frac{c_0d_0}{2(t_c-\chi_c)}=-1.
\end{align}
Thus
\begin{align}
h_{n,1}(z(t))-h_{n,1}(t_c)=-\xi(x+te^{i\theta(\delta_1)})+o(1),
\end{align}
uniformly for $0\le t\le M_n$ as $n\to \infty$. Now, by (\ref{qnAsym}), for $r>0$,
\begin{align*}
(d_0n^{2/3})^rq_n(nz(t);r)=\prod_{k=0}^{r-1}\frac{n^{1/3}(d_0n^{-1/3}(x+te^{i\theta_n(\delta_1)})+n^{1/3}(t_c-t_c^{(n)})-k/n^{2/3}}{d_0}
=\left(x+te^{i\theta(\delta_1)}\right)^re^{o(1)} 
\end{align*}
uniformly for $0\le t\le M_n$ as $n\to \infty$.

We can do the same type of computations for $\overline{\G}_n^1$, $\overline{\G}_n^2$ and $\overline{\g}_n^2$. This gives us
\begin{align}
\label{LKernelConv}
&\lim_{n\to \infty}\frac{c_0d_0}{2(t_c-\chi_c)(2\pi i)^2}\tilde{L}_{\mathcal{R}}^{(n)}((x_1,y_1),(x_2,y_2))\nonumber\\
&=\frac{c_0d_0}{2(t_c-\chi_c)(2\pi i)^2}\int_{-\mathscr{L}_L-\mathscr{C}_{out}}dz\int_{\mathscr{L}_R+\mathscr{C}_{in}}dw\frac{1}{w-z}\frac{w^r}{z^s}e^{\frac{1}{3}w^3-\frac{1}{3}z^3-\xi w+\tau z}\nonumber\\
&=\frac{1}{(2\pi i)^2}\int_{\mathscr{L}_L+\mathscr{C}_{out}}dz\int_{\mathscr{L}_R+\mathscr{C}_{in}}dw\frac{1}{w-z}\frac{w^r}{z^s}e^{\frac{1}{3}w^3-\frac{1}{3}z^3-\xi w+\tau z}.
\end{align}

It remains to consider the asymptotics of $B_n$ in (\ref{ChangedContourKernel2}).
\begin{align}
\label{BKernel2}
&\lim_{n\to\infty}\frac{p_n(x_2,y_2)}{p_n(x_1,y_1)}\frac{c_0n^{1/3}}{2}B_n((x_1,y_1),(x_2,y_2))\nonumber\\
&=\lim_{n\to\infty}1_{x_1\geq x_2}\frac{1}{2\pi i}\int_{-\overline{\G}_n^2}dz\frac{c_0n^{1/3}}{z-x_2/n}\frac{(d_0n^{2/3})^rq_n(nz;r)}{(d_0n^{2/3})^sq_n(nz;s)}
e^{h_{n,1}(z)-h_{n,1}(t_c)-(h_{n,2}(z)-h_{n,2}(t_c)}
\end{align} 
since $\overline{\G}_n^2$ has the opposite orientation to $\G_n^2$.
Using the same asymptotics for $q_n$ and $h_n$ as above we get
\begin{equation}
\label{BKernelLimit}
-1_{-\xi\geq-\tau}\frac{1}{2\pi i}\int_{-\mathscr{C}_{out}}z^{r-s}e^{(\tau-\xi)z}\,dz\nonumber =1_{\tau\geq\xi}1_{s>r}\frac{(\tau-\xi)^{s-r-1}}{(s-r-1)!}.
\end{equation} 
This gives us finally the complete \emph{Cusp-Airy} kernel

\begin{align}
\label{CAKernel1}
\mathcal{K}_{CA}((\xi,r),(\tau,s))=-1_{\tau\geq \xi}1_{s>r}\frac{(\tau-\xi)^{s-r-1}}{(s-r-1)!}+\frac{1}{(2\pi i)^2}\int_{\mathscr{L}_L+\mathscr{C}_{out}}dz\int_{\mathscr{L}_R+\mathscr{C}_{in}}dw\frac{1}{w-z}\frac{w^r}{z^s}e^{\frac{1}{3}w^3-\frac{1}{3}z^3-\xi w+\tau z}.
\end{align} 

It is not difficult to check that the asymptotics above is uniform for $\xi,\tau$ in compact subsets of $\R$. A standard argument using the Hadamard inequality then proves 
(\ref{Weaklimit}). This completes the proof of the Main Theorem.

\section{Representations and Properties of The Cusp-Airy Kernel}

\subsection{Symmetry Property of The Cusp-Airy Kernel}\label{sec:symmetry}

From the geometry of the problem we expect that the Cusp-Airy process should be symmetric around the line $r=0$.
This is indeed seen in the kernel as the next proposition shows.

\begin{Prop}
\label{Reflection}
The Cusp-Airy kernel satisfies 
\begin{align} 
\mathcal{K}_{CA}((\xi,-r),(\tau,-s))=(-1)^{s-r}\mathcal{K}_{CA}((\tau,s),(\xi,r)).
\end{align}
In particular, this implies that the correlation functions satisfies the reflection symmetry
\begin{align} 
\rho_n((\xi_1,-r_1),(\xi_2,-r_2),...,(\xi_n,-r_n))=\rho_n((\xi_1,r_1),(\xi_2,r_2),...,(\xi_n,r_n))
\end{align}
for all $n$.
\end{Prop}

\begin{proof}
First note that under the change of variables $z\to -z$ and $w\to -w$ the contours transform according to:
$\mathscr{L}_L\to -\mathscr{L}_R$,
$\mathscr{L}_R\to -\mathscr{L}_L$,   
$\mathscr{C}_{in}\to \mathscr{C}_{in}$, and
$\mathscr{C}_{out}\to \mathscr{C}_{out}$.
Thus we see that
\begin{align*} 
&\mathcal{K}_{CA}((\xi,-r),(\tau,-s))=-1_{\tau\geq \xi}1_{-s>-r}\frac{(\tau-\xi)^{-s+r-1}}{(-s+r-1)!}\\&+\frac{1}{(2\pi i)^2}\int_{\mathscr{L}_L\cup\mathscr{C}_{out}}dz\int_{\mathscr{L}_R\cup\mathscr{C}_{in}}dw\frac{1}{w-z}\frac{w^{-r}}{z^{-s}}e^{\frac{1}{3}w^3-\frac{1}{3}z^3-\xi w+\tau z}=\{z\to-z,w\to-w\}\\
&=1_{\tau\geq \xi}1_{r>s}(-1)^{r-s}\frac{(\xi-\tau)^{r-s-1}}{(r-s-1)!}+\frac{(-1)^{s-r}}{(2\pi i)^2}\int_{-\mathscr{L}_R\cup\mathscr{C}_{out}}dz\int_{-\mathscr{L}_L\cup\mathscr{C}_{in}}dw\frac{1}{z-w}\frac{z^s}{w^{r}}e^{\frac{1}{3}z^3-\frac{1}{3}w^3-\tau z+\xi w}=\{z\leftrightarrow w\}\\
&=1_{\tau\geq \xi}1_{r>s}(-1)^{r-s}\frac{(\xi-\tau)^{r-s-1}}{(r-s-1)!}+\frac{(-1)^{s-r}}{(2\pi i)^2}\int_{-\mathscr{L}_L\cup\mathscr{C}_{in}}dz\int_{-\mathscr{L}_R\cup\mathscr{C}_{out}}dw\frac{1}{w-z}\frac{w^s}{z^{r}}e^{\frac{1}{3}w^3-\frac{1}{3}z^3-\tau w+\xi z}.
\end{align*}
Let $\mathscr{C}$ be a negatively oriented circular contour around the origin which is contained inside $\mathscr{C}_{in}$. The residue theorem implies that 
\begin{align*}
&\frac{1}{(2\pi i)^2}\int_{-\mathscr{L}_L\cup\mathscr{C}_{in}}dz\int_{-\mathscr{L}_R\cup\mathscr{C}_{out}}dw\frac{1}{w-z}\frac{w^s}{z^{r}}e^{\frac{1}{3}w^3-\frac{1}{3}z^3-\tau w+\xi z}
\\
&=\frac{1}{2\pi i}\int_{\mathscr{C}_{in}}dzz^{s-r}e^{(\xi-\tau)z}+\frac{1}{(2\pi i)^2}\int_{-\mathscr{L}_L+\mathscr{C}_{in}}dz\int_{-\mathscr{L}_R-\mathscr{C}}dw\frac{1}{w-z}\frac{w^s}{z^{r}}e^{\frac{1}{3}w^3-\frac{1}{3}z^3-\tau w+\xi z}
\\
&=-1_{r>s}\frac{(\xi-\tau)^{r-s-1}}{(r-s-1)!}+\frac{1}{(2\pi i)^2}\int_{-\mathscr{L}_L+\mathscr{C}_{in}}dz\int_{-\mathscr{L}_R-\mathscr{C}}dw\frac{1}{w-z}\frac{w^s}{z^{r}}e^{\frac{1}{3}w^3-\frac{1}{3}z^3-\tau w+\xi z}.
\end{align*} 
This gives
\begin{align*} 
&\mathcal{K}_{CA}((\xi,-r),(\tau,-s))=1_{\tau\geq \xi}1_{r>s}(-1)^{s-r}\frac{(\xi-\tau)^{r-s-1}}{(r-s-1)!}-(-1)^{s-r}1_{r>s}\frac{(\xi-\tau)^{r-s-1}}{(r-s-1)!}\\
&+\frac{(-1)^{r-s}}{(2\pi i)^2}\int_{-\mathscr{L}_L+\mathscr{C}_{in}}dz\int_{-\mathscr{L}_R-\mathscr{C}_{in}}dw\frac{1}{w-z}\frac{w^s}{z^{r}}e^{\frac{1}{3}w^3-\frac{1}{3}z^3-\tau w+\xi z},
\end{align*}
We can now change $\mathscr{C}\to \mathscr{C}_{out}$ and $\mathscr{C}_{in}\to -\mathscr{C}_{out}$. This finally gives
\begin{align*} 
&\mathcal{K}_{CA}((\xi,-r),(\tau,-s))\\
&=(-1)^{s-r}1_{r>s}\frac{(\xi-\tau)^{r-s-1}}{(r-s-1)!}\big[1_{\tau\geq \xi}-1\big]
+\frac{(-1)^{s-r}}{(2\pi i)^2}\int_{\mathscr{L}_L\cup\mathscr{C}_{out}}dz\int_{\mathscr{L}_R\cup\mathscr{C}_{in}}dw\frac{1}{w-z}\frac{w^s}{z^{r}}e^{\frac{1}{3}w^3-\frac{1}{3}z^3-\tau w+\xi z}\\
&=-(-1)^{s-r}1_{\xi>\tau}1_{r>s}\frac{(\xi-\tau)^{r-s-1}}{(r-s-1)!}
+\frac{(-1)^{s-r}}{(2\pi i)^2}\int_{\mathscr{L}_L\cup\mathscr{C}_{out}}dz\int_{\mathscr{L}_R\cup\mathscr{C}_{in}}dw\frac{1}{w-z}\frac{w^s}{z^{r}}e^{\frac{1}{3}w^3-\frac{1}{3}z^3-\tau w+\xi z}\\
&=(-1)^{s-r}\mathcal{K}_{CA}((\tau,s)(\xi,r)).
\end{align*}
\end{proof}

\subsection{Representation of The Cusp-Airy Kernel}\label{sec:rAiry}
In this section we will derive an alternative representation of the Cusp-Airy kernel involving the so called r-Airy integrals and certain polynomials. 
Define the \emph{r-Airy integrals}, 
\begin{equation*}
A^{\pm}_r(u)=\frac 1{2\pi}\int_{\ell} e^{\frac 13 ia^3+iua}(\mp ia)^{\pm r}\,da,
\end{equation*}
where $r\ge 0$ and $\ell$ is a contour from $\infty e^{5\pi i/6}$ to $\infty e^{\pi i/6}$ such that $0$ lies above the contour, see \cite{ADvM}; compare also with the functions
$s^{(m)}$ and $t^{(m)}$ in \cite{Peche}. Note that $A_0^{\pm}(u)=\text{Ai\,}(u)$, the standard Airy function.
Let $\mathscr{L}_L^*$ be the contour $\mathscr{L}_L$ shifted to the right so that $0$ lies to the left of it, see figure 
\ref{figContourSpecial2}; define $\mathscr{L}_R^*$ analogously by shifting $\mathscr{L}_R$ to the left so that $0$ is to the right of it. It is straightforward
to check that
\begin{equation}\label{rAiryp}
\frac{1}{2\pi i}\int_{\mathscr{L}_R^*} e^{\frac 13 w^3-uw}w^r\,dw=\frac{(-1)^r}{2\pi i}\int_{\mathscr{L}_L^*} e^{-\frac 13 z^3+uz}z^r\,dz=A_r^+(u),
\end{equation}
and
\begin{equation}\label{rAirym}
\frac{(-1)^r}{2\pi i}\int_{\mathscr{L}_R^*} e^{\frac 13 w^3-uw}\frac 1{w^r}\,dw=\frac{1}{2\pi i}\int_{\mathscr{L}_L^*} e^{-\frac 13 z^3+uz}\frac{1}{z^r}\,dz=A_r^-(u),
\end{equation}
for $r\ge 0$.

Define the polynomials $P_n(w,\xi)$ and $p_n(\xi)$ through
\begin{align} 
P_n(w,\xi):=e^{-\frac{1}{3}w^3+uw}\frac{d^n}{d w^n}e^{\frac{1}{3}w^3-uw}
\end{align}
and 
\begin{align} 
p_n(u):=P_n(0,u).
\end{align}

By Cauchy's integral formula, we have for $r\ge 0$,
\begin{equation}\label{ppoly}
\frac{1}{2\pi i}\int_{\mathscr{C}_{out}} e^{\frac 13 w^3-uw}w^r\,dw=\frac{(-1)^{r-1}}{2\pi i}\int_{\mathscr{C}_{out}} e^{-\frac 13 z^3+uz}z^r\,dz=p_{r-1}(u).
\end{equation}
Note that $\mathscr{L}_L+\mathscr{C}_{out}=\mathscr{L}_L^*$, see figure \ref{figContourSpecial}. Thus,
\begin{equation}\label{LRint}
\frac{(-1)^r}{2\pi i}\int_{\mathscr{L}_R} e^{\frac 13 w^3-uw}\frac 1{w^r}\,dw=\frac{1}{2\pi i}\int_{\mathscr{L}_L} e^{-\frac 13 z^3+uz}\frac{1}{z^r}\,dz=A_r^-(u)+(-1)^rp_{r-1}(u).
\end{equation}

\begin{figure}[H]

\centering
\begin{tikzpicture}

\draw (-8.5,0) -- (8,0);
\draw[thick] (-2,0) --(-1,0);
\draw[->] (-1.75,-0.1) --(-1.25,-0.1);
\draw[->] (-1.25,0.1) --(-1.75,0.1);
\draw (0,-0.1) -- (0,0.1);
\draw[thick] (2,0) --({2+4/2},{4*sqrt(3)/2}) ;
\draw[->,thick] (3,{2*sqrt(3)/2}) --({3+0.02},{(1.02)*sqrt(3)}) ;
\draw[thick] (2,0) --({2+4/2},{-4*sqrt(3)/2}) ;
\draw [thick](-2,0) --({-2-4/2},{4*sqrt(3)/2}) ; 
\draw[->,thick] (-3,{2*sqrt(3)/2}) --(-3.02,{(1.02)*sqrt(3)}) ;
\draw[thick] (-2,0) --({-2-4/2},{-4*sqrt(3)/2}) ;
%\draw (0,0) circle (1.5cm);
\draw[thick] (0,0) circle (1.0cm);
%\draw[->,thick] (0,1.5) -- (0.02,1.5);
\draw[->,thick] (0,1.0) -- (0.02,1.0);
\draw (-2.8,2.1) node {$\mathscr{L}_L$}; 
\draw (2.6,2.1) node {$\mathscr{L}_R$};
\draw (0,1.4) node {$\mathscr{C}_{out}$}; 
%\draw (0,0.6) node {$\g^2$}; 
\draw (0.2,0.2) node {\small$0$}; 
\draw (2.0,-0.1) -- (2.0,0.1);
\draw (-2.0,-0.1) -- (-2.0,0.1);
\draw (1.9,0.4) node {\small$x$}; 
\draw (-1.9,0.4) node {\small$y$}; 
\draw (2.7,0) arc (0:58:0.7);
\draw (2.8,0.6) node {$\frac{\pi}{3}$};
\draw ({-2-0.38},{0.7*sqrt(3)/2}) arc (120:180:0.7);
\draw (-2.8,0.6) node {$\frac{\pi}{3}$};
\end{tikzpicture}
\caption{\protect\label{figContourSpecial}Deformation of the contours $\mathscr{L}_L$ and $\mathscr{C}_{out}$ so that $\mathscr{L}_L$ can be moved to the right of 0.}

\end{figure}
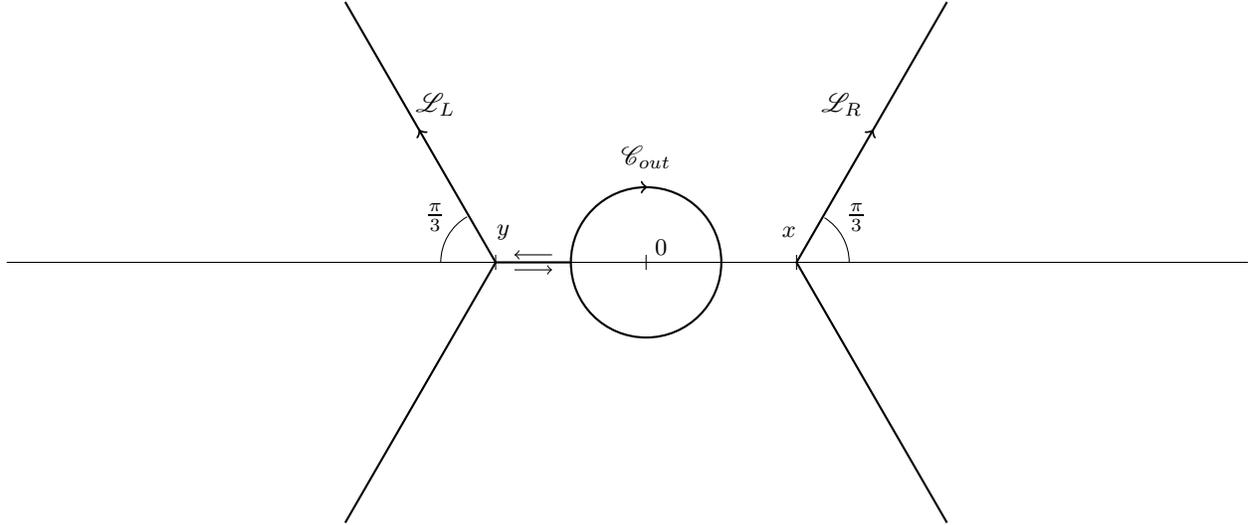

\begin{figure}[H]

\centering
\begin{tikzpicture}
\draw (-5,0) -- (5,0);
\filldraw (-2,0) circle (0.05cm);
\draw[thick] (0,0) -- (-2,{2*sqrt(3)});
\draw[thick,->] (0,0) -- (-1,{1*sqrt(3)});
\draw[thick] (0,0) -- (-2,-{2*sqrt(3)});
\draw[thick] (2,0) -- (4,{2*sqrt(3)});
\draw[thick,->] (2,0) -- (3,{1*sqrt(3)});
\draw[thick] (2,0) -- (4,-{2*sqrt(3)});
\draw (-2,0.4) node {$0$};
\draw (-1,1) node {$\mathscr{L}_L^*$}; 
\draw (3,1) node {$\mathscr{L}_R$};
\draw (-1,-1) node {$z$}; 
\draw (3,-1) node {$w$};
\end{tikzpicture}
\caption{\protect\label{figContourSpecial2}Integration contours moved to the right of 0.}

\end{figure}
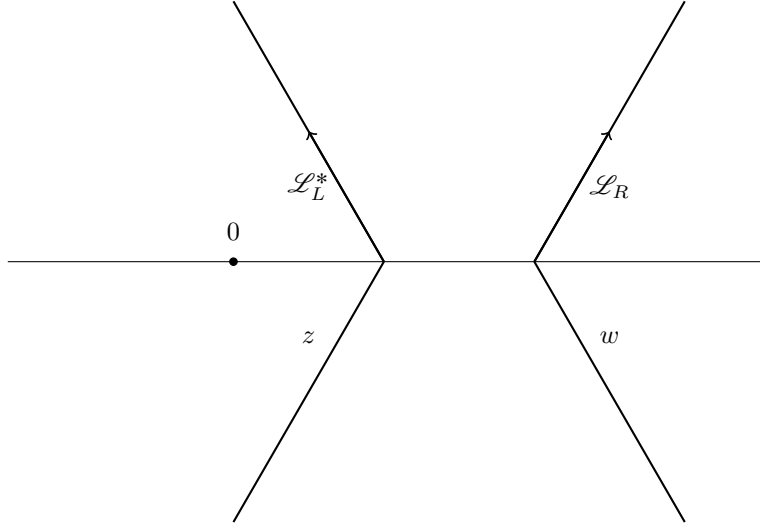

We can now give a different formula for the Cusp-Airy kernel in terms of the r-Airy integrals.

\begin{Prop}\label{rAirycorrelationformula}
The Cusp-Airy kernel can be written as
\begin{equation*}
\mathcal{K}_{CA}((\xi,r),(\tau,s))=-1_{\tau\ge\xi}1_{s>r}\frac{(\tau-\xi)^{s-r-1}}{(s-r-1)!}+\tilde{\mathcal{K}}_{CA}((\xi,r),(\tau,s)),
\end{equation*}
where $\tilde{\mathcal{K}}_{CA}$ is given 

\begin{itemize}
\item[(i)] for $r,s\ge 0$, by
\begin{equation*}
\tilde{\mathcal{K}}_{CA}((\xi,r),(\tau,s))=\int_0^{\infty}A_{s}^-(\tau+\lambda)A_r^+(\xi+\lambda)\,d\lambda,
\end{equation*}
\item[(ii)] for $r\ge 0$, $s<0$, by
\begin{equation*}
\tilde{\mathcal{K}}_{CA}((\xi,r),(\tau,s))=(-1)^s\int_0^{\infty}A_{-s}^+(\tau+\lambda)A_r^+(\xi+\lambda)\,d\lambda,
\end{equation*}
\item[(iii)] for $r<0$, $s\ge 0$, by
\begin{align*}
\tilde{\mathcal{K}}_{CA}((\xi,r),(\tau,s))&=(-1)^r\int_0^{\infty}A_s^-(\tau+\lambda)A_{-r}^-(\xi+\lambda)\,d\lambda
+(-1)^{s-r}\int_0^{\infty}p_{s-1}(\tau+\lambda)A_{-r}^-(\xi+\lambda)\,d\lambda\\
&+(-1)^{s-r}\sum_{k=0}^{s-1}p_k(\tau)A_{-r+s-k}(\xi)+\sum_{k=0}^{s-r-1}(-1)^kp_k(\tau)p_{s-r-1-k}(\xi),
\end{align*}
\item[(iv)] and for $r,s< 0$, by
\begin{equation*}
\tilde{\mathcal{K}}_{CA}((\xi,r),(\tau,s))=(-1)^{s-r}\int_0^{\infty}A_{-s}^+(\tau+\lambda)A_{-r}^-(\xi+\lambda)\,d\lambda.
\end{equation*}
\end{itemize}

\end{Prop}

\begin{proof}
From Definition \ref{CAkerneldefinition} we have that
\begin{equation*} 
\tilde{\mathcal{K}}_{CA}((\xi,r),(\tau,s))=\frac{1}{(2\pi i)^2}\int_{\mathscr{L}_L+\mathscr{C}_{out}}dz\int_{\mathscr{L}_R+\mathscr{C}_{in}}dw\frac{1}{w-z}\frac{w^r}{z^s}e^{\frac{1}{3}w^3-\frac{1}{3}z^3-\xi w+\tau z}.
\end{equation*}

Consider the case (i). If $r,s\ge0$, then $\mathscr{C}_{in}$ does not contribute, and using $\mathscr{L}_L+\mathscr{C}_{out}=\mathscr{L}_L^*$, see figure \ref{figContourSpecial},
and the formula
\begin{equation}\label{IntTrick}
\int_0^{\infty}e^{-\lambda(w-z)}\,d\lambda=\frac{1}{w-z}
\end{equation}
valid if $\text{Re\,}(w-z)>0$, we find
\begin{align*}
\tilde{\mathcal{K}}_{CA}((\xi,r),(\tau,s))&=\int_0^\infty\left(\frac{1}{2\pi i}\int_{\mathscr{L}_L^*} e^{-\frac 13 z^3+(\tau+\lambda)z}\frac 1{z^s}\,dz\right)
\left(\frac{1}{2\pi i}\int_{\mathscr{L}_R^*} e^{\frac 13 w^3-(\xi+\lambda)w}w^r\,dw\right)d\lambda\notag
\\&=\int_0^{\infty}A_{s}^-(\tau+\lambda)A_r^+(\xi+\lambda)\,d\lambda,
\end{align*}
by (\ref{rAiryp}) and (\ref{rAirym}). Here, we also used the fact that since $r\ge 0$, we can move $\mathscr{L}_R$ to $\mathscr{L}_R^*$.

In the case (ii) we have $r\ge 0$ and $s<0$, and thus neither $\mathscr{C}_{in}$ nor $\mathscr{C}_{out}$ contribute. Using (\ref{IntTrick}) and then moving $\mathscr{L}_L$ to
$\mathscr{L}_L^*$ and $\mathscr{L}_R$ to $\mathscr{L}_R^*$, we obtain
\begin{align*}
\tilde{\mathcal{K}}_{CA}((\xi,r),(\tau,s))&=\int_0^\infty\left(\frac{1}{2\pi i}\int_{\mathscr{L}_L^*} e^{-\frac 13 z^3+(\tau+\lambda)z}z^{-s}\,dz\right)
\left(\frac{1}{2\pi i}\int_{\mathscr{L}_R^*} e^{\frac 13 w^3-(\xi+\lambda)w}w^r\,dw\right)d\lambda\notag
\\&=(-1)^s \int_0^{\infty}A_{-s}^+(\tau+\lambda)A_r^+(\xi+\lambda)\,d\lambda.
\end{align*}
Next, consider the case (iii). If $r<0$, $s\ge 0$, we can write
\begin{align*}
&\int_{\mathscr{L}_L+\mathscr{C}_{out}}dz\int_{\mathscr{L}_R+\mathscr{C}_{in}}dw=\int_{\mathscr{L}_L}dz\int_{\mathscr{L}_R+\mathscr{C}_{in}}dw
+\int_{\mathscr{C}_{out}}dz\int_{\mathscr{L}_R+\mathscr{C}_{in}}dw\\
&=\int_{\mathscr{L}_L}dz\int_{\mathscr{L}_R}dw+\int_{\mathscr{C}_{out}}dz\int_{\mathscr{L}_R}dw
+\int_{\mathscr{C}_{out}}dz\int_{\mathscr{C}_{in}}dw:=I_1+I_2+I_3.
\end{align*}
Consider first $I_1$. By (\ref{rAirym}) and (\ref{LRint}) we get
\begin{align}\label{I1}
I_1&=\int_0^\infty\left(\frac{1}{2\pi i}\int_{\mathscr{L}_L} e^{-\frac 13 z^3+(\tau+\lambda)z}\frac{1}{z^s}\,dz\right)
\left(\frac{1}{2\pi i}\int_{\mathscr{L}_R^*} e^{\frac 13 w^3-(\xi+\lambda)w}\frac{1}{w^{-r}}\,dw\right)d\lambda\notag
\\&=(-1)^s \int_0^{\infty}A_{-s}^+(\tau+\lambda)A_r^+(\xi+\lambda)\,d\lambda.
\end{align}
Since $|w|>|z|$ if $z\in\mathscr{C}_{out}$ and $w\in\mathscr{L}_R$, we can use the identity
\begin{equation*}
\frac{1}{w-z}=\sum_{k=0}^\infty\frac{z^k}{w^{k+1}}
\end{equation*}
to see that
\begin{align}\label{I2}
I_2&=\sum_{k=0}^{s-1}\left(\frac{1}{2\pi i}\int_{\mathscr{C}_{out}} e^{-\frac 13 z^3+\tau z}\frac{1}{z^{s-k}}\,dz\right)
\left(\frac{1}{2\pi i}\int_{\mathscr{L}_R} e^{\frac 13 w^3-\xi w}\frac{1}{w^{-r+k+1}}\,dw\right)\notag\\
&=\sum_{k=0}^{s-1}p_{s-k-1}(\tau)\bigg((-1)^{-r+k+1}A_{-r+k+1}^-(\xi)+p_{-r+k}(\xi)\bigg)\notag\\
&=(-1)^{s-r}\sum_{k=0}^{s-1}p_k(\tau)A_{-r+s-k}^-(\xi)+\sum_{k=0}^{s-1}(-1)^kp_k(\tau)p_{-r+s-1-k}(\xi).
\end{align}

If $z\in\mathscr{C}_{out}$ and $w\in\mathscr{C}_{in}$, then
\begin{equation*}
\frac{1}{w-z}=-\sum_{k=0}^\infty\frac{w^k}{z^{k+1}}
\end{equation*}
and we see that, by (\ref{ppoly}), and the fact that $\mathscr{C}_{in}$ is negatively oriented, we have
\begin{align}\label{I3}
I_3&=-\sum_{k=0}^{-r-1}\left(\frac{1}{2\pi i}\int_{\mathscr{C}_{out}} e^{-\frac 13 z^3+\tau z}\frac{1}{z^{s+k+1}}\,dz\right)
\left(\frac{1}{2\pi i}\int_{\mathscr{C}_{in}} e^{\frac 13 w^3-\xi w}\frac{1}{w^{-r-k}}\,dw\right)\notag\\
&=\sum_{k=0}^{-r-1}p_{s+k}(\tau)p_{-r-k-1}(\xi)
=\sum_{k=s}^{s-r-1}p_k(\tau)p_{-r+s-1-k}(\xi).
\end{align}
Adding up (\ref{I1}) - (\ref{I3}) we have proved (iii). The case (iv) follows from (i) by using Proposition \ref{Reflection}.
\end{proof}

Note that if we take $r=s\ge 0$, then by (i),
\begin{equation*}
\mathcal{K}_{CA}((\xi,r),(\tau,s))=\int_0^{\infty}A_{r}^-(\tau+\lambda)A_r^+(\xi+\lambda)\,d\lambda:=K^{(r)}(\tau,\xi),
\end{equation*}
which is called the \emph{r-Airy kernel}. Note that when $r=0$ we get the standard Airy kernel.
The r-Airy kernel has appeared previously in the work \cite{Peche} on largest eigenvalues of sample covariance
matrices and in \cite{ADvM} on Dyson's Brownian motion with outliers. See also \cite{Baik2}.

Though we do not show it in this paper using the results above and some further estimates it should be possible to prove that the scaling limit of the position 
of the last (red) particle on line $r$, close to the cusp point, has the distribution function $F_r(x)=\det(I-K^{(r)})_{L^2(x,\infty)}$. In particular, when $r=0$ we should get the Tracy-Widom distribution.

\begin{rem}
\label{GUEcorner}
It is instructive to compare the Cusp-Airy kernel to the GUE-corner kernel. Recall that the GUE-corner kernel is given by
\begin{align}
K(n,x;n',x')=-1_{n>n'}1_{x>x'}2^{n-n'}\frac{(x-x')^{n-n'-1}}{(n-n'-1)!}+\frac{2}{(2\pi i)^2}\int_{\G_0}dz\int_{L}\frac{dw}{w-z}\frac{w^{n'}}{z^{n}}e^{w^2-z^2+2zx-2wx'},
\end{align}
where $x,x'\in \R$ and $n,n'\in\Z$ and $n,n'\geq 0$, and where the contours are shown in figure \ref{figGUEContour}.

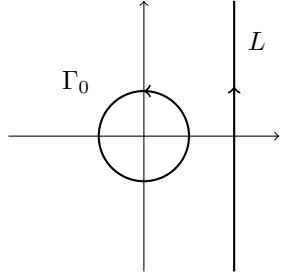
\begin{figure}[H]

\centering
\begin{tikzpicture}[xscale=0.6,yscale=0.6]

\draw[->] (-3,0) -- (3,0);
\draw[->] (0,-3) --(0,3);
\draw[thick] (2,-3) -- (2,3);
\draw[->,thick] (2,1) -- (2,1.1);
\draw[thick] (0,0) circle (1.0cm);
\draw[<-,thick] (0,1.0) -- (0.02,1.0);
\draw (-1.5,1.2) node {$\G_0$}; 
\draw (2.5,2.1) node {$L$};
\end{tikzpicture}
\caption{\label{figGUEContour} Contours for the GUE-corner kernel.}
\end{figure}
In a sense, one can regard the Cusp-Airy kernel as a double sided version of the GUE-corner kernel. If one changes the assumption that $f'_{t_c}$ has a simple root at $t_c\in R_1\bigcup R_2$, then a similar computation as in the Cusp-Airy case will yield the GUE-corner kernel for an appropriate scaling limit of $K_n((x_1^{(n)},y_1^{(n)}),(x_2^{(n)},y_2^{(n)}))$. 
\end{rem}

%\subsection{Difference-Differential Equation for the Cusp-Airy Kernel}\sec{sec:diff}
%Let $\Delta$ denote the finite difference operator defined by the action $\Delta_x f(x)=f(x+1)-f(x)$. In section 10 of \cite{OkRe2}, the authors show that the \emph{Pearcey kernel} satisfies a system of partial differential equations. We will show, similarly that the \emph{Cusp-Airy} kernel satisfies a system of difference-differential equations.
%\begin{Prop}
%The Cusp-Airy kernel satisfies the identities
%\begin{align}
%(\Delta_r-\partial_{\xi})\mathcal{K}_{CA}&=-\mathcal{K}_{CA},\\
%(\Delta_s+\partial_{\mu})\mathcal{K}_{CA}&=-\mathcal{K}_{CA}.
%\end{align} 
%\end{Prop}
%\begin{proof}
%This follows from the formula for the correlation kernel and the identities
%\begin{align*}
%(\Delta_r-\partial_{\xi})w^re^{\frac{1}{3}w^3-\xi w}&=(\Delta_rw^r)e^{\frac{1}{3}w^3-\xi w}-w^r\partial_{\xi}e^{\frac{1}{3}w^3-\xi w}\\
%&=(w-1)w^re^{\frac{1}{3}w^3-\xi w}-w^{r+1}e^{\frac{1}{3}w^3-\xi w}=-w^re^{\frac{1}{3}w^3-\xi w},
%\end{align*} 
%and
%\begin{align*}
%(\Delta_s+\partial_{\tau})z^{-s}e^{-\frac{1}{3}z^3+\tau z}&=-z^{-s}e^{-\frac{1}{3}z^3+\tau z}.
%\end{align*} 
%\end{proof}

\section{Derivation of the Correlation Kernel}\label{sec:kernel}

\subsection{Integral Representation of the Correlation Kernel for the Yellow Particles}\label{sec:interlacingkernel}
In section A.1 in \cite{Duse14a} it was shown that the correlation kernel for the interlacing particle system is given by

\begin{align}\label{KN}
K_n((x_1,y_1),(x_2,y_2))=\tilde{K}_n((x_1,y_1),(x_2,y_2))-\phi_{(y_1,y_2)}^{(n)}(x_1,x_2),
\end{align}
where 
\begin{align}\label{KNtilde}
\tilde{K}_n((x_1,y_1),(x_2,y_2))=\frac{(n-y_1)!}{(n-y_2-1)!}\sum_{k=1}^n1_{\beta_{k}^{(n)}\geq x_2}\sum_{l=x_1+y_1-n}^{x_1}\frac{\prod_{j=x_2+y_2-n+1}^{x_2-1}(\beta_k^{(n)}-j)}{\prod_{\substack{j=x_1+y_1-n\\ j\neq l}}^{x_1}(l-j)}\prod_{i\neq k}\frac{l-\beta_i^{(n)}}{\beta_k^{(n)}-\beta_i^{(n)}}
\end{align}
and
\begin{align}\label{phi}
\phi_{(y_1,y_2)}^{(n)}(x_1,x_2)=1_{x_1\geq x_2}\frac{(n-y_1)!}{(n-y_2-1)!}\sum_{k=1}^n\sum_{l=x_1+y_1-n}^{x_1}\frac{\prod_{j=x_2+y_2-n+1}^{x_2-1}(\beta_k^{(n)}-j)}{\prod_{\substack{j=x_1+y_1-n\\ j\neq l}}^{x_1}(l-j)}\prod_{i\neq k}\frac{l-\beta_i^{(n)}}{\beta_k^{(n)}-\beta_i^{(n)}}.
\end{align}

To arrive at an integral representation for the correlation kernel we now make use of the following corollary of the residue theorem:
\begin{Cor}\label{Cauchy}
Let $\Omega\in \mathbb{C}$ be an open simply connected bounded domain with positively oriented boundary Jordan curve $\gamma$, then for an analytic function $f(z)$ in $\Omega$, we have for $\{z_1,...,z_n\}\in \Omega$ and $z_i\neq z_j$ if $i\neq j$
\begin{align*}
\sum_{i=1}^nf(z_i)\prod_{\substack{k=1\\i\neq k}}^n\frac{1}{z_i-z_k}=\frac{1}{2\pi i}\oint_{\gamma}f(z)\prod_{i=1}^n\frac{1}{z-z_i}dz.
\end{align*}
\end{Cor}
\begin{Prop}\label{propYellowkenerl}
The correlation kernel for the yellow tiles have the following integral representation
\begin{align} \label{Yellowkernel}
K_{\mathcal{Y}}^{(n)}((x_1,y_1),(x_2,y_2))&=1_{x_1<x_2}\frac{(n-y_1)!}{(n-y_2-1)!}\frac{1}{(2\pi i)^2}\oint_{\mathscr{Z}_n}dz\oint_{\mathscr{W}_n}dw  \frac{\prod_{k=x_2+y_2-n+1}^{x_2-1}(z-k)}{\prod_{\substack{k=x_1+y_1-n}}^{x_1}(w-k)}\frac{1}{w-z}\prod_{i=1}^n\bigg(\frac{w-\beta_i^{(n)}}{z-\beta_i^{(n)}}\bigg)\nonumber\\&
-1_{x_1\geq x_2}\frac{(n-y_1)!}{(n-y_2-1)!}\frac{1}{(2\pi i)^2}\oint_{\mathscr{Z}_n'}dz\oint_{\mathscr{W}_n}dw  \frac{\prod_{k=x_2+y_2-n+1}^{x_2-1}(z-k)}{\prod_{\substack{k=x_1+y_1-n}}^{x_1}(w-k)}\frac{1}{w-z}\prod_{i=1}^n\bigg(\frac{w-\beta_i^{(n)}}{z-\beta_i^{(n)}}\bigg),
\end{align}
where $\mathscr{Z}_n$ is a counterclockwise oriented contour containing $\{\beta_j^{(n)}: \beta_j^{(n)}\geq x_2\}$ but not the set $\{\beta_j^{(n)}\leq x_2-1\}$ and $\mathscr{Z}'_n$ is a counterclockwise oriented contour containing $\{\beta_j^{(n)}: \beta_j^{(n)}<x_2\}$ but not the set $\{\beta_j^{(n)}\geq x_2+1\}$, and $\mathscr{W}_n$ contains the set $\{x_1+y_1-n,...,x_1\}$ and $\mathscr{Z}_n$ and $\mathscr{Z}'_n$.

\end{Prop}
\begin{proof}
The correlation kernel for the yellow tiles is the same as that for the interlacing particles, i.e. it is given by (\ref{KN}). From (\ref{KN}) to (\ref{phi}) we see that
\begin{align} \label{Yellowkernel2}
\frac{(n-y_2-1)!}{n-y_1)}K_{\mathcal{Y}}^{(n)}((x_1,y_1),(x_2,y_2))&=1_{x_1<x_2}\sum_{k=1}^n1_{\beta_{k}^{(n)}\geq x_2}\sum_{l=x_1+y_1-n}^{x_1}\frac{\prod_{j=x_2+y_2-n+1}^{x_2-1}(\beta_k^{(n)}-j)}{\prod_{\substack{j=x_1+y_1-n\\ j\neq l}}^{x_1}(l-j)}\prod_{i\neq k}\frac{l-\beta_i^{(n)}}{\beta_k^{(n)}-\beta_i^{(n)}}\notag\\
&-1_{x_1\geq x_2}\sum_{k=1}^n1_{\beta_{k}^{(n)}< x_2}\sum_{l=x_1+y_1-n}^{x_1}\frac{\prod_{j=x_2+y_2-n+1}^{x_2-1}(\beta_k^{(n)}-j)}{\prod_{\substack{j=x_1+y_1-n\\ j\neq l}}^{x_1}(l-j)}\prod_{i\neq k}\frac{l-\beta_i^{(n)}}{\beta_k^{(n)}-\beta_i^{(n)}}.
\end{align}
Consider the first expression in the right hand side of (\ref{Yellowkernel2}). Using Corollary \ref{Cauchy} we can rewrite the $l$-summation and we find
\begin{align*}
&1_{x_1<x_2}\sum_{k=1}^n1_{\beta_{k}^{(n)}\geq x_2}\frac 1{2\pi i}\oint_{\mathscr{W}_n}dw
\frac{\prod_{j=x_2+y_2-n+1}^{x_2-1}(\beta_k^{(n)}-j)}{\prod_{j=x_1+y_1-n}^{x_1}(w-j)}\prod_{i\neq k}\frac{w-\beta_i^{(n)}}{\beta_k^{(n)}-\beta_i^{(n)}}\notag\\
&=1_{x_1<x_2}\sum_{k=1}^n1_{\beta_{k}^{(n)}\geq x_2}\frac 1{2\pi i}\oint_{\mathscr{W}_n}dw
\frac{\prod_{j=x_2+y_2-n+1}^{x_2-1}(\beta_k^{(n)}-j)}{\prod_{j=x_1+y_1-n}^{x_1}(w-j)}\frac 1{w-\beta_k^{(n)}}
\frac{\prod_{i=1}^n(w-\beta_i^{(n)})}{\prod_{i\neq k}(\beta_k^{(n)}-\beta_i^{(n)})}.
\end{align*}
Since the $w$-contour is outside $\mathscr{Z}_n$ the sum over the $\beta_k^{(n)}\geq x_2$ can be rewritten using Corlollary \ref{Cauchy} and we find
\begin{equation*}
1_{x_1<x_2}\frac 1{(2\pi i)^2}\oint_{\mathscr{W}_n}dw\oint_{\mathscr{Z}_n}dz
\frac{\prod_{j=x_2+y_2-n+1}^{x_2-1}(z-j)}{\prod_{j=x_1+y_1-n}^{x_1}(w-j)}\frac 1{w-z}
\frac{\prod_{i=1}^n(w-\beta_i^{(n)})}{\prod_{i=1}^n(z-\beta_i^{(n)})}.
\end{equation*}
The second expression in the right hand side of (\ref{Yellowkernel2}) can be rewritten in exactly the same way and we have proved the proposition.

\end{proof}

\subsection{Particle Transformation}

From knowledge of a correlation kernel for the yellow particles we now want to derive an expression for a correlation kernel for the red (and blue) particles.

\begin{Lem}
\label{Lem:RedBlue}
Correlation kernels for the red and blue tiles (particles) are given by
\begin{align*} 
K_{\mathcal{R}}^{(n)}((x_1,y_1),(x_2,y_2))&=-K_{\mathcal{Y}}^{(n)}((x_1,y_1),(x_2,y_2-1))\\
 K_{\mathcal{B}}^{(n)}((x_1,y_1),(x_2,y_2))&=K_{\mathcal{Y}}^{(n)}((x_1,y_1),(x_2+1,y_2-1)).
\end{align*}
\end{Lem}

 \begin{proof}
Let $K_{\mathcal{P}}$ be the Kasteleyn matrix of the adjacency matrix of the honeycomb graph $G_{\mathcal{P}}$ of the polygon $\mathcal{P}$. It is defined according to
\begin{align} 
K_{\mathcal{P}}((x,n);(y,m))
= \left\{
\begin{array}{rl} 1 & \text{if $(y,m)=(x,n)$} \\
1 &  \text{if $(y,m)=(x,n-1)$}\\
1 & \text{if $(y,m)=(x+1,n-1)$}\\
0&  \text{otherwise}
 \end{array} \right.
\end{align}
Recall that if $(y,m)=(x,n)$ we have a yellow particle (rhombi of shape) at position $(x,n)$ in our lattice. Similarly, if  $(y,m)=(x,n-1)$ we have a red particle at position $(x,n)$ and if $(y,m)=(x+1,n-1)$ we have a blue particle at position $(x,n)$.\newline \newline It was shown in \cite{Pet12} Theorem 6.1 that inverse Kasteleyn matrix $K_{\mathcal{P}}^{-1}$ is related to the correlation kernel of the yellow particles $K_{\mathcal{Y}}$ according to
\begin{align} 
K_{\mathcal{P}}^{-1}((y,m);(x,n))=(-1)^{y-x+m-n}K_{\mathcal{Y}}^{(n)}(x,n;y,m).
\end{align}
From Corollary 3 in \cite{Kenyon} one has that the probability of finding a set of edges $\{b_1w_1,...b_kw_k\}$ is given by
\begin{align*} 
\p[\text{edges at $\{w_1b_1,...,w_kb_k\}$}]=\bigg(\prod_i^kK_{\mathcal{P}}(w_i,b_i)\bigg)\det(K_{\mathcal{P}}^{-1}(b_i,w_j))_{i,j}^k
\end{align*}
Now finding $k$ red particles at positions $\{(x_i,n_i)\}_{i=1}^k$ is equivalent to finding the edges $\{((x_i,n_i),(x_i,n_i-1))\}_{i=1}^k$
Hence, the probability of finding  $k$ red particles at positions $\{(x_i,n_i)\}_{i=1}^k$ equals
\begin{align*} 
\p[\text{red particles at positions $\{(x_i,n_i)\}_{i=1}^k$}]&=\p[\text{edges at positions $\{((x_i,n_i),(x_i,n_i-1))\}_{i=1}^k$}]\\&=\det(K_{\mathcal{P}}^{-1}((x_i,n_i-1),(x_j,n_j))_{i,j}^k
\\&=\det((-1)^{x_i-x_j+n_i-n_j+1}K_{\mathcal{Y}}^{(n)}(x_j,n_j;x_i,n_i-1))_{i,j}^k
\\&=\det(-K_{\mathcal{Y}}^{(n)}(x_i,n_i;x_j,n_j-1))_{i,j}^k
\end{align*}
However, by definition
\begin{align*} 
\rho_{\mathcal{R}}((x_1,n_1),(x_2,n_2),...,(x_k,n_k))&=\p[\text{red particles at positions $\{(x_i,n_i)\}_{i=1}^k$}]\\&=\det(K_{\mathcal{R}}^{(n)}(x_i,n_i;y_j,n_j))_{1\leq i,j\leq k}\\
\end{align*}
Hence, we find that as a correlation kernel for the red particles we can take
\begin{align*} 
K_{\mathcal{R}}^{(n)}(x_i,n_i;y_j,n_j)=-K_{\mathcal{Y}}^{(n)}(x_i,n_i;x_j,n_j-1).
\end{align*}
Similarly, for the blue particles we can take
\begin{align*} 
K_{\mathcal{B}}^{(n)}(x_i,n_i;y_j,n_j)=K_{\mathcal{Y}}^{(n)}(x_i,n_i;x_j+1,n_j-1).
\end{align*}
This concludes the proof.
\end{proof}

We can now prove Proposition \ref{Transform}.
Combining Lemma \ref{Lem:RedBlue} with (\ref{Yellowkernel}) we see that a correlation kernel for the red particles is given by
\begin{align} 
&-1_{x_1<x_2}\frac{(n-y_1)!}{(n-y_2)!}\frac{1}{(2\pi i)^2}\oint_{\mathscr{Z}_n}dz\oint_{\mathscr{W}_n}dw  \frac{\prod_{k=x_2+y_2-n}^{x_2-1}(z-k)}{\prod_{\substack{k=x_1+y_1-n}}^{x_1}(w-k)}\frac{1}{w-z}\prod_{i=1}^n\bigg(\frac{w-\beta_i^{(n)}}{z-\beta_i^{(n)}}\bigg)\nonumber \\
&+1_{x_1\geq x_2}\frac{(n-y_1)!}{(n-y_2)!}\frac{1}{(2\pi i)^2}\oint_{\mathscr{Z}_n'}dz\oint_{\mathscr{W}_n}dw  \frac{\prod_{k=x_2+y_2-n}^{x_2-1}(z-k)}{\prod_{\substack{k=x_1+y_1-n}}^{x_1}(w-k)}\frac{1}{w-z}\prod_{i=1}^n\bigg(\frac{w-\beta_i^{(n)}}{z-\beta_i^{(n)}}\bigg)
\end{align}
We can remove the prefactor $\frac{(n-y_1)!}{(n-y_2)!}$ since it cancels in the determinantal expression for the correlation functions. This proves Proposition \ref{Transform}.

\section{Appendix. Determinantal Point Processes}

In this appendix we give a brief introduction to determinantal random point
processes that suffices for our purposes. See e.g. \cite{Joh06b},
for a more complete treatment.

Let $\Lambda$ be a Polish space. Fix $N \in \N \cup \{\infty\}$ and
$Y \subset \Lambda^N$, a space of configurations of $N$-particles of
$\Lambda$. Denote each $y \in Y$ as $y = (y_1,\ldots,y_N)$. Assume, for all
$y \in Y$ and compact Borel sets $B \subset \Lambda$, that the number of particles
from $y$ contained in $B$ is finite, i.e., $\# \{y_i \in B\} < \infty$. Let
$\mathcal{F}$ be the sigma-algebra generated by sets of the form
$\{y \in Y: \# \{y_i \in B\} = m \}$ for all $m \le N$ and Borel sets
$B \subset \Lambda$. A probability space of the form $(Y,\mathcal{F},\Pp)$ is
referred to as a {\em random point process}.

Given such a process, $m \leq N$, and $B \subset \Lambda^m$,
define $N_B^m : Y \to \N$ by,
\begin{equation*}
N_B^m (y) := \# \{(y_{i_1},\ldots,y_{i_m}) \in B : i_1 \ne \cdots \ne i_m \},
\end{equation*}
for all $y \in Y$. In words, $N_B^m (y)$ is the number of distinct
$m$-tuples of particles from $y$ that are contained in $B$. Then define a
measure on $\Lambda^m$ by $B \mapsto \E [ N_B^m ]$ for all Borel subsets
$B \subset \Lambda^m$. Assume that this is well-defined and finite whenever
$B$ is bounded. Then, given a reference measure $\lambda$ on $\Lambda$, the
density of the above measure with respect to $\lambda^m$, whenever it exists,
is referred to as the {\em $m^\text{th}$ correlation function}, $\rho_m$. That
is,
\begin{equation*}
\int_B \rho_m(x_1,\ldots,x_m) d\lambda[x_1] \ldots d\lambda[x_m] = \E [ N_B^m ],
\end{equation*}
for all Borel subset $B \subset \Lambda^m$.

A random point process is called {\em determinantal} if all correlation functions
exist and there exists a function $K : \Lambda^2 \to \C$ for which
\begin{equation*}
\rho_m(x_1,\ldots,x_m) = \det[K(x_i,x_j)]_{i,j=1}^m,
\end{equation*}
for all $x_1,\ldots,x_m \in \Lambda$ and $m \leq N$. $K$ is called the {\em correlation
kernel} of the process.
Note, correlation kernels are not unique. For example, when
$\Lambda = \R$, a new kernel $J : \R^2 \to \C$ can be defined by
$J(u,v) := \frac{w(u)}{w(v)} K(u,v)$ for all $u,v \in \R$, where $w$ is any
non-zero complex function. Also, kernels can be viewed as
integral operators on $L^2(\Lambda)$ by the relation,
\begin{equation*}
K f(u) := \int K(u,v) f(v) d\lambda[v],
\end{equation*}
whenever the right hand side is well-defined.

Finally consider a measurable function $\phi: \Lambda \to \C$ with bounded support,
$B$, for which
\begin{equation*}
\sum_{n=0}^\infty \frac{ \Vert \phi \Vert_\infty ^n}{n!}
\int_{B^n} \det[K(x_i,x_j)]_{i,j=1}^n d\lambda[x_1] \ldots d\lambda[x_n] < \infty.
\end{equation*}
Proposition 2.2 of \cite{Joh06b}, then gives, 
\begin{equation*}
\E[ \prod_j(1-\phi(x_j)) ] = \sum_{n=0}^\infty \frac{(-1)^n}{n!} \int_{B^n}
\prod_{j=1}^n \phi(x_j) \det[K(x_i,x_j)]_{i,j=1}^n d\lambda[x_1] \ldots d\lambda[x_n].
\end{equation*}
This quantity is referred to as the {\em Fredholm determinant}, denoted
$\det[I-\phi K]_{L^2(B)}$.

\vskip 0.5cm

\noindent
{\sc Erik Duse}, Department of Mathematics, KTH Royal Institute of Technology, Stockholm, Sweden, \texttt{duse@kth.se}
\newline\newline
{\sc Kurt Johansson}, Department of Mathematics, KTH Royal Institute of Technology, Stockholm, Sweden, \texttt{kurtj@kth.se}
\newline\newline
{\sc Anthony Metcalfe}, Department of Mathematics, Uppsala University, Uppsala, Sweden, \newline\texttt{anthony.metcalfe@math.uu.se}

\end{document}